\documentclass[11pt]{amsart}
\usepackage[english]{babel}
\usepackage[utf8]{inputenc}
\usepackage{fullpage}
\usepackage{enumerate}
\usepackage[toc]{appendix}

\usepackage{amsmath}
\usepackage{amsfonts}
\usepackage{amsthm}
\usepackage{mathtools}
\usepackage{dsfont}
\usepackage{esint}
\usepackage{appendix}

\usepackage{graphics,graphicx}

\usepackage[backref=page]{hyperref}
\hypersetup{
    backref=true, 
    plainpages=false,
    bookmarks=true,         
    unicode=false,          
    pdftoolbar=true,        
    pdfmenubar=true,        
    pdffitwindow=true,      
    pdftitle={My title},    
    pdfauthor={Author},     
    pdfsubject={Subject},   
    pdfcreator={Creator},   
    pdfproducer={Producer}, 
    pdfkeywords={keywords}, 
    pdfnewwindow=true,      
    colorlinks=true,       
    linkcolor=blue,          
    citecolor=blue,        
    filecolor=blue,      
    urlcolor=magenta,          
    urlbordercolor=0 1 1
}

\newtheorem{theo}{Theorem}[section]
\newtheorem{prop}[theo]{Proposition}
\newtheorem{lemma}[theo]{Lemma}
\newtheorem{dfn}[theo]{Definition}

\newtheorem{remk}[theo]{Remark}

\newtheorem{xmpl}[theo]{Example}
\newtheorem{hyp}{Hypothesis}

\usepackage[hang]{subfigure}

\newcommand{\R}{\mathbb R}
\newcommand{\N}{\mathbb N}

\newcommand{\cH}{\mathcal{H}}
\newcommand{\cK}{\mathcal{K}}
\newcommand{\cC}{\mathcal{C}}
\newcommand{\cE}{\mathcal{E}}
\newcommand{\cF}{\mathcal{F}}

\newcommand{\cL}{\mathcal{L}}

\newcommand{\cM}{\mathcal{M}}

\newcommand{\cT}{\mathcal{T}}
\newcommand{\cV}{\mathcal{V}}

\def\xC{{\rm C}}
\def\xLip{ {\rm Lip} }
\def\xlip{ {\rm lip} }

\def\xL{{\rm L}}
\def\xW{{\rm W}}
\def\xdiv{{\rm div}}
\def\xdiam{{\rm diam}}

\newcommand{\V}{\|V\|}

\newcommand{\hv}{{\hat v}}

\newcommand{\G}{G_{d,n}}
\newcommand{\Div}{\text{div}}

\DeclareMathOperator*{\argmin}{arg\,min}
\DeclareMathOperator*{\supp}{spt}
\DeclareMathOperator*{\dist}{dist}

\DeclareMathOperator*{\diam}{diam}

\usepackage{color}
\definecolor{grey}{rgb}{.7,.7,.7}

\newcommand{\one}{\mathds{1}}
\renewcommand{\phi}{\varphi}
\renewcommand{\epsilon}{\varepsilon}
\newcommand{\e}{\epsilon}
\renewcommand{\G}{G_{d,n}}

\title{A varifold approach to surface approximation}
\author{Blanche \textsc{Buet}, Gian Paolo \textsc{Leonardi}, and Simon \textsc{Masnou}}
\date{\today}
\keywords{Varifold, regularized first variation, approximate mean curvature, point cloud}
\thanks{G.P.~Leonardi has been supported by GNAMPA-INdAM (2015 project: \textit{Problemi isoperimetrici e Teoria Geometrica della Misura in Spazi Metrici}). B. Buet and S. Masnou acknowledge the support of French {\it Agence Nationale de la Recherche} under grant ANR-12-BS01-0014-01 (project GEOMETRYA). Part of this work was supported by the LABEX MILYON (ANR-10-LABX-0070) of Universit\'e de Lyon, within the program "Investissements d'Avenir" (ANR-11-IDEX- 0007) operated by the French National Research Agency (ANR).}
\thanks{We warmly thank Julie Digne for her precious help and suggestions regarding C++ libraries and algorithmic implementations of point clouds.}
\subjclass[2010]{Primary: 49Q15. Secondary: 68U05, 65D18}

\numberwithin{equation}{section}

\begin{document}
     
\begin{abstract}
We show that the theory of varifolds can be suitably enriched to open the way to applications in the field of discrete and computational geometry. Using appropriate regularizations of the mass and of the first variation of a varifold we introduce the notion of approximate mean curvature and show various convergence results that hold in particular for sequences of \textit{discrete varifolds} associated with point clouds or pixel/voxel-type discretizations of $d$-surfaces in the Euclidean $n$-space, without restrictions on dimension and codimension. The variational nature of the approach also allows to consider surfaces with singularities, and in that case the approximate mean curvature is consistent with the generalized mean curvature of the limit surface. A series of numerical tests are provided in order to illustrate the effectiveness and generality of the method.
\end{abstract}

\maketitle

\section*{Introduction}
Shape visualization and processing are fundamental tasks in many different fields, such as physics, engineering, biology, medicine, astronomy, art and architecture. This substantially motivates the extremely active research on image processing and computer graphics, that has been carried out in the last decades. Due to the large variety of applications and of capture systems, different discrete models are used for representing shapes: among them we mention polygonal/polyhedral meshes, level sets, pixel/voxel representations, point clouds, splines, CAD models. Despite the obvious differences due to variable data sources and application ranges, it makes sense to ask whether or not different representations might be interpreted as particular instances of a common (and possibly more general) formalism.

In what follows we shall also distinguish between ``structured'' and ``unstructured'' discretizations: the former being those discrete representations which directly encode dimensional and topological properties of the underlying shapes (polygonal/polyhedral meshes, level sets), the latter being those encoding only spatial distribution (point clouds, pixel/voxel-type). Concerning in particular the unstructured discretizations, which arise from statistical sampling and are typically produced by most image capture devices, some post-processing is often needed to extract the geometric information that characterize the underlying, ``real'' surfaces. In this sense, one of the major tasks is the reconstruction of curvatures. 

In the case of point clouds, the reconstruction of curvatures is usually performed by finding a local smooth surface via regression, and then by computing the curvature of the reconstructed surface, like in the Moving Least Squares (MLS) technique, introduced in the seminal article \cite{levin1998} (see also \cite{alexa2003,amenta2004}). 
There exist also techniques based on geometric integral invariants, that is on asymptotic behaviour of area, volume, covariance matrix (or other integral quantities) inside infinitesimal balls, where the leading term contains the wanted geometric information, see \cite{Yang2006,Clarenz2004} where such approaches are developed on point clouds, \cite{Coeurjolly2014} on digital shapes and \cite{Merigot2009} where robustness to noise is proved. A common drawback of these methods is that they are not suitable for the reconstruction of curvature in presence of singularities, and convergence results are generally obtained under quite strong regularity assumptions.

A different approach is followed in \cite{morvan_cohen_steiner} and \cite{thibert}. There, a notion of \emph{second fundamental measure}, valid both for surfaces and for their discrete approximations, is introduced by relying on the theory of \emph{normal cycles} developed in seminal works of Wintgen~\cite{Wintgen1982}, Z{\"a}hle~\cite{zahle} and Fu~\cite{fu1993} (see also \cite{morvan_book}). Roughly speaking, the idea of those papers is to reconstruct curvature measure information from the offsets of distance-like functions associated with the discrete data. On the one hand, this approach is quite general as it relies on Geometric Measure Theory (in principle, the reconstructed curvature measures may contain singular parts allowing for singularities in the limit shape). On the other hand, this method has been efficiently developed on triangulated surfaces, while its adaptation to point cloud data is not straightforward, see \cite{thibert} for more details.

In the recent works \cite{BlancheThesis,Blanche-rectifiability} the framework of \textit{varifolds} has been proposed as a possible answer to the above question. Varifolds have been originally introduced by Almgren in 1965~\cite{almgren} for the study of critical points of the area functional, and as a model of soap films, soap bubble clusters, and more general physical systems where surfaces or interfaces come into play. Formally, a $d$--varifold $V$ is a Radon measure on the product space $\R^{n}\times \G$, where $\G$ denotes the Grassmannian manifold of (unoriented) $d$-planes in $\R^{n}$, $1\le d< n$ (see Definition \ref{def:gdv}). Roughly speaking, a varifold consists of a joint distribution of mass and of ``tangent'' $d$--planes. It is natural to associate a varifold to a smooth $d$--dimensional surface in $\R^{n}$, and more generally to a $d$--rectifiable set with locally finite $d$--dimensional Hausdorff measure, possibly endowed with a multiplicity function (see Definition \ref{def:rdv}). In this sense, varifolds provide a natural generalization of (unoriented, weighted) surfaces. Varifolds satisfy nice geometric and variational properties (compactness, mass continuity, criteria of rectifiability) and possess a generalized notion of mean curvature encoded in the so-called \textit{first variation} operator~\cite{Allard72,simon}. 

Notwithstanding these nice properties, the theory of varifolds has been essentially developed and used in the context of Geometric Measure Theory and Calculus of Variations by a quite restricted number of specialists, and almost no substantial applications of this theory to discrete and computational geometry, as well as to numerical analysis and image processing, have been proposed up to now. The only exception we are aware of is the work by Charon and Trouv\'e~\cite{trouve_charon}, where varifolds are employed to perform comparisons between triangularized shapes, with applications to computational anatomy. The reason for the consistent lack of practical application of the varifold theory is, probably, twofold. On one hand, most of the theory (and of Geometric Measure Theory in general) consists of very deep, but also extremely technical results, which make the theory appear particularly exotic and only meant to address purely theoretical questions. On the other hand, despite its potentials, the theory of varifolds has been conceived and developed neither for the discrete approximation of surfaces, nor for the analysis of discrete geometric objects. 

The main aim of this paper is to show that varifolds can indeed be successfully used to represent and analyze not only continuous shapes, like curves, surfaces, and rectifiable sets, but also discrete shapes, like point clouds, pixel/voxel-type surfaces, polyhedral meshes, etc. Varifolds associated with discrete shapes will be generically called {\it discrete varifolds}. In order to make discrete varifolds a truly useful tool, a certain extension of the classical theory is required, especially with reference to approximation results, compactness properties, and the choice of appropriate notions of curvature. 

First of all we introduce a \textit{regularized first variation} of a varifold $V$ as the convolution of the standard first variation $\delta V$ (which is a distribution of order $1$) with a regularizing kernel $\rho_{\e}$. Then we convolve the mass $\V$ of the varifold by another kernel $\xi_{\e}$ and define the \textit{approximate mean curvature} vector field $H_{\rho,\xi,\e}^{V}$ as the regularized first variation divided by the regularized mass (whenever the latter is not zero). This way we can define approximate mean curvatures for \textit{any} varifold. Of course these notions of mean curvature depend on the choice of the pair $(\rho_{\e},\xi_{\e})$ of regularizing kernels. Assuming that $\rho_{\e}$ and $\xi_{\e}$ are defined by suitably scaling two given kernels $\rho_{1}$ and $\xi_{1}$, the parameter $\e$ can be understood as a scale at which the mean curvature is evaluated. We remark that the idea of using convolutions of varifolds and of their first variations is not new, as it already appears as a technical tool in Brakke's paper \cite{brakke} on the mean curvature motion of a varifold. What to our knowledge is new is the simple, but at the end very effective, idea of taking regularizations of the mass and of the first variation in order to give consistent notions of mean curvature for all varifolds (and in particular for the discrete ones).
 
The approximate mean curvatures satisfy some nice convergence properties, that are stated in Theorems \ref{thm:convergence1}, \ref{thm:convergence2} and \ref{thm:convergence3}. In these results the notion of \textit{Bounded Lipschitz distance}, a metric similar to Wasserstein distance (see Definition \ref{dfn_flat_distance}), comes into play. By relying on such convergence theorems one can for instance show Gamma-convergence for a class of Willmore-type functionals defined on discrete varifolds, that will be presented in a forthcoming paper. 

A thorough comparison between our approach and the different notions of discrete (mean) curvature, which have been proposed and studied in the past literature, is out of the scope of this paper (we refer the interested reader to the extended survey contained in \cite{najman:hal-01090755}).
However, as an illustrative example of the power and the generality of the varifold setting we show in Section~\ref{section:cotangent} the following, remarkable fact: the classical \textit{Cotangent Formula}, that is widely used for defining the mean curvature of a polyhedral surface $\mathcal P$ at a vertex $v$, can be simply understood as the first variation of the associated varifold $V_{\mathcal P}$ applied to any Lipschitz extension of the piecewise affine basis function $\phi_{v}$ that takes the value $1$ on $v$ and is identically zero outside the patch of triangles around $v$. In this sense, the Cotangent Formula can be understood as the regularization of $\delta V_{\mathcal P}$ by means of the finite family of piecewise affine kernels $\{\phi_{v}(x):\ v \text{ is a vertex of }\mathcal P\}$.

Moreover, in the last section we provide some numerical simulations that show the efficacy and generality, as well as the ease of implementation, of our notion of approximate mean curvature. In particular, our simulations confirm the convergence results proved in Section \ref{section:AMC} with even better rates than those theoretically expected, not only in the case of smooth limit surfaces but also for a much larger class of generalized surfaces, possibly endowed with singularities.
\medskip

For the reader's convenience we provide a more detailed description of the contents of the paper.

In Sections \ref{section:pre} and \ref{section:varifolds} we set some basic notation and briefly introduce the notion of varifold together with some essential facts from the theory of varifolds, that will be recalled in the subsequent sections. We also define the Bounded Lipschitz Distance $\Delta^{1,1}$, which locally metrizes the weak-$*$ convergence of varifolds and will be systematically used throughout the paper. At the end of the section we define the class of discrete varifolds, focusing in particular on those of ``volumetric'' and ``point-cloud'' type. We then prove an equivalence result, Proposition \ref{prop_comparison_volumetric_point_cloud}, which says that one can switch between the volumetric and point-cloud types up to paying some explicit price in terms of the $\Delta^{1,1}$ distance.

Section \ref{section:RFV} is devoted to the definition of the regularized first variation $\delta V\ast \rho_{\e}$, as well as to show some related properties. We point out that, here and in the rest of the paper, we take radially symmetric kernels with some minimal regularity, as specified in Section \ref{section:pre} (see also Hypothesis \ref{rhoepsxieps}); we shall also denote by $\rho$ the one-dimensional profile of the kernel $\rho_{1}$. A compactness and rectifiability result using the regularized first variation is shown at the end of the section, see Theorem \ref{bounded_first_variation_condition_for_sequences}. We stress that Theorem \ref{bounded_first_variation_condition_for_sequences}, a partial extension of a classical result due to Allard \cite{Allard72}, can be used for showing existence of solutions to variational problems on discrete varifolds.

The notion of approximate mean curvature $H_{\rho,\xi,\e}^{V}$ of a varifold $V$ is introduced and studied in Section \ref{section:AMC}. The consistency of the notion stems from two approximation results that we show to hold for rectifiable varifolds with locally bounded first variation. The first one, Theorem \ref{thm:convergence1}, states the pointwise convergence of the approximate mean curvature $H_{\rho,\xi, \e}^{V}$ to the generalized mean curvature $H$ of $V$ at $\V$-almost every point, without an explicit convergence rate. The second one, Theorem \ref{thm:convergence2}, shows the following fact: given a rectifiable varifold $V$ with bounded first variation, a sequence $V_{i}$ of general varifolds, a sequence of points $z_{i}$ converging to a point $x$ in the support of $V$, and two infinitesimal sequences of parameters $d_{i},\e_{i}$, such that the local $\Delta^{1,1}$ distance between $V$ and $V_{i}$ is controlled from above by $d_{i}$ up to a renormalization (see \eqref{eq_thm_pointwise_cv_hyp2}), then one obtains that the modulus of the difference $H_{\rho,\xi,\e_{i}}^{V_{i}}(z_{i}) - H_{\rho,\xi,\e_{i}}^{V}(x)$ can be asymptotically estimated by $\frac{d_{i}+|z_{i}-x|}{\e_{i}^{2}}$. We remark that the infinitesimality of the parameter $d_{i}$ appearing in the theorem is always guaranteed by the convergence of $V_{i}$ to $V$ in the sense of varifolds. Moreover, by combining the above estimate with Theorem \ref{thm:convergence1} one deduces that $H_{\rho,\xi,\e_{i}}^{V_{i}}(z_{i})$ converges to $H(x)$ as soon as $\frac{d_{i}+|z_{i}-x|}{\e_{i}^{2}}$ is infinitesimal as $i\to \infty$. 
For the last convergence result of the section, Theorem \ref{thm:convergence3}, we introduce an orthogonal approximate mean curvature $H_{\rho,\xi,\e}^{V,\perp}(x)$ which essentially consists in projecting $H_{\rho,\xi,\e}^{V}(x)$ onto the ``normal space'' to $V$ at $x$. Then we assume that the varifold $V$ is associated with a $C^{2}$ manifold $M$, and take a sequence of varifolds $V_{i}$ that converge to $V$ with explicit controls on the local $\Delta^{1,1}$ distance between the mass measures $\V$ and $\|V_{i}\|$ as well as on suitably defined, $L^{\infty}$-type local distances between the ``tangential parts'' of the varifolds (see assumptions \eqref{eq_thm_pointwise_cvSmooth_hyp2} and \eqref{eq_thm_pointwise_cvSmooth_hyp3}). Under these assumptions the theorem shows better convergence rates with respect to the ones guaranteed by Theorem \ref{thm:convergence2} (basically, we obtain $\frac{d_{i}}{\e_{i}}$ instead of $\frac{d_{i}}{\e_{i}^{2}}$ as in the asymptotic convergence estimate \eqref{epsdeltaboh}).

Section \ref{section:pairs} addresses the question whether some specific choice of the pair $(\rho,\xi)$ of kernel profiles might lead to improved convergence rates in Theorems \ref{thm:convergence2} and \ref{thm:convergence3}. Even though we are not able to give a completely satisfactory answer to this question, we provide some heuristic motivation for choosing $(\rho,\xi)$ within a class of kernel pairs that satisfy what we call the \textit{Natural Kernel Pair} (NKP) property, defined as follows: we say that a pair of kernel profiles $(\rho,\xi)$ satisfies the (NKP) property if $\rho$ is monotonically decreasing on $[0,+\infty)$ and of class $W^{2,\infty}$, while $\xi$ is given by the formula
\[
\xi(t) = -\frac{t\rho'(t)}{n}\,.
\]
The (NKP) property relies on an observation that we made after performing a Taylor expansion of the difference $H_{\rho,\xi,\e}^{V}(x) - H(x)$ when $V$ is associated with a smooth manifold $M$ of class $C^{3}$ and $x$ is a point in the relative interior of $M$. While it is still unclear whether (NKP) guarantees or not some better convergence rates than those proved in Theorems \ref{thm:convergence2} and \ref{thm:convergence3}, we have observed a substantial improvement of convergence and stability in all numerical tests we have performed, some of which are presented in Section \ref{section:numerics}. 
\medskip

In the remaining sections we consider some theoretical and applied aspects of discrete varifolds.

In Section \ref{section:DAV} we show that point cloud varifolds as well as discrete volumetric varifolds can be used to approximate more general rectifiable varifolds. Specifically we show in Theorem \ref{diffuse_discrete_varifolds_theorem} that any rectifiable varifold $V$ with bounded first variation can be approximated by a sequence $V_{i}$ of discrete varifolds of the previous types, constructed on generic meshes of the space with mesh-size $\delta_{i}$. If moreover $V$ belongs to the narrower, but quite natural class of \textit{piecewise $C^{1,\beta}$} varifolds (see Definition \ref{def:piecewiseC1beta}) then we also show in the same theorem an explicit estimate on the distance $\Delta^{1,1}(V,V_{i})$, which is proved to be smaller than $\delta_{i}^{\beta}$, up to multiplicative constants and for $i$ sufficiently large. 

Section \ref{section:cotangent} is devoted to the illustration of a simple, but remarkable fact concerning the classical \textit{Cotangent Formula}, which is one of the fundamental tools in discrete geometry as it allows to define a vector mean curvature functional $\hat H$ for a triangulated surface. Using the varifold formalism one can understand the formula as the first variation of the associated polyhedral varifold $V$ evaluated on a Lipschitz extension of a generic nodal function $\phi$. This fact, which is proved in Proposition~\ref{prop_CotangentFormula}, gives in our opinion a strong support to the suitability of the varifold approach to discrete geometry.

Finally, we put in Section \ref{section:numerics} some numerical tests for confirming the theoretical convergence properties of the approximate mean curvature as well as the effectiveness of the (NKP) property, according to the results of Sections \ref{section:AMC} and \ref{section:pairs}. For the sake of computing the exact mean curvture vector, most of the tests have been made by taking some parametrized shapes (a circle, an ellipse, a ``flower'', an ``eight'', a union of two circles, and some standard double bubbles in 2D and 3D) and by discretizing them as point cloud varifolds. Then we choose various kernel pairs $(\rho,\xi)$ and compute the approximate mean curvatures (with or without projection on the orthogonal space) for various choices of the scale parameter $\e$, as well as the relative error with respect to the theoretical ones. We also examine the behavior of the approximate mean curvature near the singularities in the ``eight'', in the union of two circles, and in a couple of standard double bubbles in 2D and 3D. In particular we obtain slightly nicer results by taking averages of $H^{V}_{\e,\rho,\xi}$ over balls of radius $2\e$, which actually does not affect the validity of our theoretical convergence results. Our tests confirm the robustness of the notion of approximate mean curvature in both 2D and 3D cases, even in presence of singularities. We finally conclude the section with two figures showing the intensity of approximate mean curvatures computed for point clouds representation of a dragon and a statue.

\section{Preliminary notations and definitions}
\label{section:pre} 
We adopt the following notations:
\begin{itemize}

\item $\N$ and $\R$ denote, respectively, the set of natural and of real numbers.


\item Given $n\in \N$, $n\ge 1$, $x\in \R^{n}$ and $r>0$, we set $B_r(x) = \left\lbrace y\in \R^{n} \, | \, |y-x| < r \right\rbrace$.

\item $A^\delta = \bigcup_{x \in A} B_\delta (x) = \{ y \in \R^n \, | \, d( y,A) <\delta \}$.

\item $\cL^n$ is the $n$--dimensional Lebesgue measure and $\omega_n = \cL^n (B_1(0))$.

\item $\cH^d$ is the $d$--dimensional Hausdorff measure in $\R^{n}$.

\item If $B$ is an open set, writing $A\subset\subset B$ means that $A$ is a relatively compact subset of $B$.

\item $\G$ denotes the Grassmannian manifold of $d$-dimensional vector subspaces of $\R^n$. A $d$-dimensional subspace $T$ will be often identified with the orthogonal projection onto $T$ (sometimes denoted as $\Pi_{T}$). $\G$ is equipped with the metric $d(T,P) = \Vert T - P \Vert$, where $\Vert \cdot \Vert$ denotes the operator norm on the space $L(\R^{n};\R^{n})$ of linear endomorphisms of $\R^{n}$.

\item Given a continuous $\R^m$--valued function $f$ defined in $\Omega$, its support $\supp f$ is the closure in $\Omega$ of $\{ y \in \Omega \: | \: f(y) \neq 0 \}$.

\item $\xC_o^0 (\Omega)$ is the closure of $\xC_c^0(\Omega)$ in $\xC^{0}(\Omega)$ with respect to the norm $\|u \|_{\infty} = \sup_{x\in \Omega}|u(x)|$.

\item Given $k \in \N$, $\xC_c^k (\Omega)$ is the space of real--valued functions of class $\xC^k$ with compact support in $\Omega$.

\item $\xC^{0,1}(\R)$ and $\xC^{1,1}(\R)$ denote, respectively, the space of Lipschitz functions and the space of functions of class $C^{1}$ with derivative of class Lipschitz on $\R$. On such spaces we shall consider the norms $\|u\|_{1,\infty} = \|u\|_{\infty}+\xlip(u)$ and $\|u\|_{2,\infty} = \|u\|_{\infty}+\|u'\|_{1,\infty}$, respectively.

\item $\xLip_L (X)$ is the space of Lipschitz functions in the metric space $(X,\delta)$ with Lipschitz constant $\leq L$.

\item We denote by $| \mu |$ the total variation of a measure $\mu$.
\item $\cM_{loc} (X)^m$ is the space of $\R^m$--valued Radon measures and $\cM (X)^m$ is the space of $\R^m$--valued finite Radon measures on the metric space $(X,\delta)$.
\end{itemize}

From now on, we fix $d, \, n \in \N$ with $1\leq d < n$. By $\Omega \subset \R^n$ we shall always denote an open set. 

We also fix the notations and some basic assumptions about the regularizing kernels that will be used in the sequel. We shall consider radially symmetric, non-negative kernels $\rho_{1},\xi_{1}$ defined on $\R^{n}$, such that $\rho_{1}(x) = \rho(|x|)$ and $\xi_{1}(x) = \xi(|x|)$ for suitable one-dimensional, even profile functions $\rho,\xi$ with compact support in $[-1,1]$. We also assume the normalization condition
\[
\int_{\R^{n}}\rho_{1}(x)\, dx = \int_{\R^{n}}\xi_{1}(x)\, dx = 1\,,
\]
which amounts to 
\[
n\omega_{n}\int_{0}^{1}\rho(t) t^{n-1}\, dt = n\omega_{n}\int_{0}^{1}\xi(t) t^{n-1}\, dt = 1\,.
\]
For any given $\e>0$ and $x\in\R^{n}$ we set
\[
\rho_{\e}(x) = \e^{-n}\rho_{1}(x/\e)\qquad\text{and}\qquad \xi_{\e}(x) = \e^{-n}\xi_{1}(x/\e)\,.
\]
We shall at least assume that $\rho\in \xC^{1}(\R)$ and that $\xi\in \xC^{0}(\R)$. At some point, we shall require extra regularity on $\rho$ and $\xi$, namely that $\rho\in W^{2,\infty}$ and $\xi\in W^{1,\infty}$ (see Hypothesis \ref{rhoepsxieps}).

\section{Varifolds}
\label{section:varifolds}

\subsection{Some classical definitions and properties of varifolds}

We recall here a few facts about varifolds, see for instance \cite{simon} for more details.

\noindent Let us start with the general definition of varifold:
\begin{dfn}[General $d$--varifold]\label{def:gdv}\rm 
Let $\Omega \subset \R^n$ be an open set. A $d$--varifold in $\Omega$ is a non-negative Radon measure on $\Omega \times G_{d,n}$. 
\end{dfn}

An important class of varifolds is represented by the so-called \emph{rectifiable} varifolds.
\begin{dfn}[Rectifiable $d$--varifold]\label{def:rdv}\rm
Given an open set $\Omega \subset \R^n$, let $M$ be a countably $d$--rectifiable set and $\theta$ be a non negative function with $\theta > 0$ $\cH^d$--almost everywhere in $M$. A rectifiable $d$--varifold $V= v(M,\theta)$ in $\Omega$ is a non-negative Radon measure on $\Omega \times G_{d,n}$ of the form $V= \theta \mathcal{H}^d_{| M} \otimes \delta_{T_x M}$ i.e.
\[
\int_{\Omega \times G_{d,n}} \varphi (x,T) \, dV(x,T) = \int_M \varphi (x, T_x M) \, \theta(x) \, d \mathcal{H}^d (x) \quad \forall \varphi \in \xC_c^0 (\Omega \times G_{d,n} , \mathbb{R})
\] where $T_x M$ is the approximate tangent space at $x$ which exists $\mathcal{H}^d$--almost everywhere in $M$. The function $\theta$ is called the \emph{multiplicity} of the rectifiable varifold. If additionally $\theta(x)\in \N$ for $\mathcal{H}^d$--almost every $x\in M$, we say that $V$ is an \emph{integral} varifold.

\end{dfn}

\begin{dfn}[Mass]\rm
The mass of a general varifold $V$ is the positive Radon measure defined by $\Vert V \Vert (B) = V (\pi^{-1} (B))$ for every $B \subset \Omega$ Borel, with $\pi:\Omega\times G_{d,n}\to \Omega$ defined by $\pi(x,S) = x$. 
In particular, the mass of a $d$--rectifiable varifold $V=v(M,\theta)$ is the measure $\V=\theta \mathcal{H}^d_{| M}$.
\end{dfn}

The following result is proved via a standard disintegration of the measure $V$ (see for instance \cite{ambrosio}).
\begin{prop}[Young-measure representation]\label{prop:Young} Given a $d$--varifold $V$ on $\Omega$, there exists a family of probability measures $\{\nu_{x}\}_{x}$ on $\G$ defined for $\V$-almost all $x\in \Omega$,  such that $V = \V\otimes \{\nu_{x}\}_{x}$, that is,
\[
V(\phi) = \int_{x\in\Omega}\int_{S\in\G}\phi(x,S)\, d\nu_{x}(S)\, d\V(x)
\]
for all $\phi\in \xC^{0}_{c}(\Omega\times\G)$.
\end{prop}

\begin{dfn}[Convergence of varifolds]\rm
A sequence of $d$--varifolds $(V_i)_i$ weakly--$\ast$ converges to a $d$--varifold $V$ in $\Omega$ if, for all $\varphi \in \xC_c^0 (\Omega \times G_{d,n})$,
\[
\langle V_i,\phi\rangle=\int_{\Omega \times \G} \phi (x,P) \, dV_i (x,P) \xrightarrow[i \to \infty]{} \langle V,\phi\rangle= \int_{\Omega \times \G} \phi(x,P) \, dV(x,P) \: .
\]
\end{dfn}

We now recall the definition of \emph{Bounded Lipschitz distance} between two Radon measures.  It is also called \emph{flat metric} and can be seen as a modified $1$--Wasserstein distance which allows the comparison of measures with different masses (see \cite{villani_book}, \cite{piccoli_rossi}). In contrast, the $1$--Wasserstein distance between two measures with different masses is infinite.

\begin{dfn}[Bounded Lipschitz distance] \label{dfn_flat_distance}\rm 
Being $\mu$ and $\nu$ two Radon measures on a locally compact metric space $(X,d)$, we define
\[
\Delta^{1,1} (\mu,\nu) = \sup \left\lbrace \left| \int_X \phi \, d\mu - \int_X \phi \, d\nu \right| \: : \:\phi \in \xLip_1 (X), \:  \| \phi \|_{\infty} \leq 1 \right\rbrace\,.
\]
It is well-known that $\Delta^{1,1}(\mu,\nu)$ defines a distance on the space of Radon measures on $X$, called the \emph{Bounded Lipschitz distance}.
\end{dfn}

Hereafter we introduce some special notation for the $\Delta^{1,1}$ distance between varifolds.
\begin{dfn} \label{dfn_localized_bounded_lipschitz_distance}\rm 
Let $\Omega \subset \R^n$ be an open set and let $V,W$ be two $d$--varifolds on $\Omega$. For any open set $U\subset \Omega$ we define 
\[
\Delta^{1,1}_U (V,W) = \sup \left\lbrace \left| \int_{\Omega \times \G} \phi \, dV - \int_{\Omega \times \G} \phi \, dW \right| \: : \: \begin{array}{l}
\phi \in \xLip_1 (\Omega \times \G), \:  \| \phi \|_{\infty} \leq 1 \\
\text{ and } \supp \phi \subset U \times \G \end{array} \right\rbrace 
\]
and
\[
\Delta^{1,1}_U(\| V \|, \| W \|) = \sup \left\lbrace \left| \int_{\Omega } \phi \, d \| V \| - \int_{\Omega } \phi \, d \| W \| \right| \: : \: \begin{array}{l}
\phi \in \xLip (\Omega), \:  \| \phi \|_{\infty} \leq 1 \\
\text{ and } \supp \phi \subset U  \end{array} \right\rbrace \: . 
\]
We shall often drop the subscript when $U  = \Omega$, that is we set 
\[
\Delta^{1,1}(V,W) = \Delta^{1,1}_{\Omega}(V,W)\quad\text{and}\quad \Delta^{1,1}(\|V\|,\|W\|) = \Delta^{1,1}_{\Omega}(\|V\|,\|W\|)\,,
\] 
thus making the dependence upon the domain implicit whenever this does not create any confusion.
\end{dfn}

The following fact is well-known (see \cite{villani_book,bogachev2007}).
\begin{prop}\label{BLDversusWeakstar}
Let $\mu,(\mu_{i})_{i}$, $i\in \N$, be Radon measures on a locally compact metric space $(X,d)$. Assume that $\mu(X) + \sup_{i}\mu_{i}(X) <+\infty$ and that there exists a compact set $K\subset X$ such that the supports of $\mu$ and of $\mu_{i}$ are contained in $K$ for all $i\in \N$. Then $\mu_{i}\xrightharpoonup[]{\: \ast \:}\mu$ if and only if $\Delta^{1,1}(\mu_{i},\mu)\to 0$ as $i\to\infty$.
\end{prop}

We define the tangential divergence of a vector field, as follows. For all $P \in G$ and $X =(X_1, \ldots , X_n) \in \xC_c^1 (\Omega , \R^n )$ we set
\[
{\rm{div}}_P X (x) = \sum_{j=1}^n \langle \nabla^P X_j (x) , e_j \rangle = \sum_{j=1}^n \langle P (\nabla X_j (x)) , e_j \rangle 
\]
with $(e_1, \ldots, e_n)$ being the canonical basis of $\R^n$. Then we introduce the following definition.
\begin{dfn}[First variation of a varifold, \cite{Allard72}] The first variation of a $d$--varifold in $\Omega \subset \R^n$ is the vector--valued distribution (of order $1$) defined for any vector field $X\in {\xC}_c^1 (\Omega , \R^n )$ as
\[
\delta V(X) = \int_{\Omega \times G_{d,n}} {\xdiv}_P X (x) \, d V (x,P)\,.
\]
\end{dfn}
\begin{remk}\label{remk:morefirstvar}\rm
It is convenient to define the action of $\delta V$ on a function $\phi\in C^{1}_{c}(\Omega)$ as the vector
\[
\delta V(\phi) = \big(\delta V(\phi\, e_{1}),\dots,\delta V(\phi\, e_{n})\big)\,.
\]
We also notice that $\delta V(X)$ is well-defined whenever $X$ is a Lipschitz vector field such that the measure $\V$ of the set of non-differentiability points for $X$ is zero.  
\end{remk}

\noindent
The definition of first variation can be motivated as follows. Let $\phi^{X}_{t}$ be the one-parameter group of diffeomorphisms generated by the flow of the vector field $X$. Let $\Phi^{X}_{t}$ be the mapping defined on $\R^{n}\times\G$ as
\[
\Phi^{X}_{t}(x,S) = (\phi^{X}_{t}(x), d\phi^{X}_{t}(S))\,.
\]
Set $V_{t}$ as the push--forward of the varifold measure $V$ by the mapping $\Phi^{X}_{t}$. Then, assuming $\supp(X) \subset \subset A$ for some relatively compact open set $A\subset \Omega$, one has the identity
\[
\delta V(X) = \frac{d}{dt} \|V_{t}\|(A)_{|t=0}\,.
\]

\noindent The linear functional $\delta V$ is by definition continuous with respect to the $C^{1}$-topology on $C^{1}_{c}(\Omega,\R^{n})$, however in general it is not continuous with respect to the $\xC_c^0$ topology. In the special case when this is satisfied, that is, for any fixed compact set $K\subset\Omega$ there exists a constant $c_{K}>0$, such that for any vector field $X\in \xC_c^1 (\Omega , \R^n)$ with $\supp X \subset K$, one has
\[
| \delta V (X) | \leq c_K \sup_K |X| \: ,
\] 
we say that $V$ has a \emph{locally bounded first variation}. In this case, by Riesz Theorem, there exists a vector--valued Radon measure on $\Omega$ (still denoted as $\delta V$) such that
\[
\delta V (X) = \int_\Omega X \cdot \delta V \quad \text{for every } X \in \xC_c^0 (\Omega,\mathbb{R}^n)
\]
Thanks to Radon-Nikodym Theorem, we can decompose $\delta V$ as 
\begin{equation}\label{RNLdecomp}
\delta V = - H \|V\| + \delta V_s \: ,
\end{equation}
where $H \in \left( \xL^1_{loc}(\Omega, \| V \|) \right)^n$ and $\delta V_s$ is singular with respect to $\V$. The function $H$ is called the {\it generalized mean curvature vector}. By the divergence theorem, $H$ coincides with the classical mean curvature vector if $V=v(M,1)$, where $M$ is a $d$-dimensional submanifold of class $\xC^2$.

\subsection{Discrete varifolds}
\label{subsection:DV}

We introduce some types of varifolds, that we call \textit{discrete} as they are defined by a finite set of parameters. Although a varifold structure can be of course associated with any polyhedral complex, we shall essentially focus on ``unstructured'' discrete varifolds (discrete volumetric or point cloud, see below). It is however worth noticing that all definitions and results of Sections \ref{section:RFV} and \ref{section:AMC} are valid in particular for all sequences of discrete varifolds, including those of polyhedral type. Concerning the results of Section \ref{section:DAV}, the construction of approximations $V_{i}$ of a rectifiable varifold $V$, such that $\Delta^{1,1}(V_{i},V)$ is infinitesimal, as shown in Theorem \ref{diffuse_discrete_varifolds_theorem}, seems to be quite delicate in the polyhedral setting, since the tangent directions are prescribed by the directions of the cells of the polyhedral surface, which are not necessarily converging when the polyhedral surfaces converge to a smooth one in Hausdorff distance: think of Schwarz Lantern for instance. (However, a related polyhedral approximation has been obtained in \cite{fu1993}, see also \cite{feuvrier2008}).
The construction of such approximations is much simpler in the case of volumetric and point cloud varifolds.

Let $\Omega \subset \R^n$ be an open set. By a \emph{mesh} of $\Omega$ we mean a countable partition $\cK$ of $\Omega$, that is, a finite collection of pairwise disjoint subdomains (``cells'') of $\Omega$ such that $\left\lbrace K \in \cK \, : \, K \cap B \neq \emptyset \right\rbrace$ is finite for any bounded set $B \subset \Omega$ and
\[
\Omega = \bigsqcup_{K \in \cK} K\,.
\]
Here, no other assumptions on the geometry of the cells $K\in \cK$ are needed. We shall often refer to the \emph{size} of the mesh $\cK$, denoted by
\[
\delta = \sup_{K \in \cK} \diam K \,.
\]

We come to the definition of discrete volumetric varifold (see \cite{Blanche-rectifiability}).
\begin{dfn}[Discrete volumetric varifold] \label{diffuse_discrete_varifolds} 
Let $\cK$ be a mesh of $\Omega$ and let $\{(m_K ,P_K)  \}_{K \in \cK} \subset \R_+ \times \G$. We define the \emph{discrete volumetric varifold}
\[
V_\cK^{vol} = \sum_{K \in \cK} \frac{m_K}{|K|} \cL^n_{| K} \otimes \delta_{P_K}\,,\quad \text{ where } |K| = \cL^n (K) \: .
\]
\end{dfn}
We remark that discrete volumetric $d$--varifolds are typically not $d$--rectifiable (indeed their support is $n$--rectifiable, while $d<n$).

We can similarly define point cloud varifolds.

\begin{dfn}[Point cloud varifolds] 
Let $\{ x_i \}_{i=1 \ldots N} \subset \R^n$ be a point cloud, weighted by the masses $\{ m_i \}_{i=1 \ldots N}$ and provided with directions $\{ P_i \}_{i=1 \ldots N} \subset \G$. We associate the  collection of triplets $\{(x_{i},m_{i},P_{i}):\ i=1,\dots,N\}$ with the \emph{point cloud $d$--varifold}
\[
V^{pt} = \sum_{i=1}^N m_i \, \delta_{x_i} \otimes \delta_{P_i} \: ,
\]
so that for $\phi \in \xC_c^0 (\Omega \times \G)$,
\[
\int \phi \, dV^{pt} = \sum_{i=1}^N m_i\phi (x_i,P_i) \: .
\]
Of course, a point cloud varifold is not $d$-rectifiable as its support is zero-dimensional.
\end{dfn}

The idea behind these ``unstructured'' types of discrete varifolds is that they can be used to discretize more general varifolds. For instance, given a $d$--varifold $V$, and defining 
\[
m_K = \V ( K) \;\text{ and }\; P_K \in \argmin_{P \in \G} \int_{K \times \G} \left\| P - S \right\| \, dV(x,S) \: ,
\]
one obtains a volumetric approximation of $V$. Similarly one can construct a point cloud approximation of $V$. The possibility of switching between discrete volumetric varifolds and point cloud varifolds, up to a controlled error depending on the size of a given mesh, is shown in the following proposition.

\begin{prop} \label{prop_comparison_volumetric_point_cloud}
Let $\Omega \subset \R^n$ be an open set. Consider a mesh $\cK$ of $\Omega$ of size $\displaystyle \delta = \sup_{K \in \cK} \diam K $ and a family $\{x_{K}, m_K ,P_K \}_{K \in \cK} \subset \R^{n}\times \R_+ \times \G$ such that $x_K \in K$, for all $K \in \cK$. Define the volumetric varifold $V_\cK^{vol}$ and the point cloud varifold $V_\cK^{pt}$ as
\[
V_\cK^{vol} = \sum_{K \in \cK} \frac{m_K}{|K|} \cL^n_{| K} \otimes \delta_{P_K} \text{ and } V_\cK^{pt} = \sum_{K \in \cK} m_K \delta_{x_K} \otimes \delta_{P_K} \: .
\] 
Then, for any open set $U \subset \Omega$ one obtains 
\begin{equation*}
\Delta^{1,1}_{U} (V_\cK^{vol},V_\cK^{pt})  \leq \delta \min \left( \|V_\cK^{vol}\|(U^\delta) , \|V_\cK^{pt}\|(U^\delta) \right)
\end{equation*}
\end{prop}

\begin{proof}
Let $\phi \in \xLip_1 (\R^n \times \G)$ such that $\supp \phi \subset U$, then
\begin{align*}
\Bigg| \int_{U \times \G} \phi \, dV_\cK^{vol}  &- \int_{U \times \G} \phi \, d V_\cK^{pt}  \Bigg| = \left| \sum_{ K \in \cK  } \frac{m_K}{|K|} \int_K \phi (x,P_K) \cL^n (x) -  \sum_{ K \in \cK } m_K \phi (x_K,P_K) \right| \\
& \leq \sum_{\substack{ K \in \cK \\ U \cap K \neq \emptyset } } m_K \fint_K \left| \phi (x,P_K) - \phi (x_K,P_K) \right| \, d \cL^n(x) \\
& \leq \sum_{\substack{ K \in \cK \\ U \cap K \neq \emptyset } } \fint_K \underbrace{\xlip(\phi)}_{ = 1} |x-x_K| \, d \cL^n(x) \leq  \delta \sum_{\substack{ K \in \cK \\ U \cap K \neq \emptyset } } m_K  \\
& = \delta \,  \|V_\cK^{vol}\| \left( \bigcup_{U \cap K \neq \emptyset} K \right) = \delta  \,  \|V_\cK^{pt}\| \left( \bigcup_{U \cap K \neq \emptyset} K \right)\\
& \le \delta\, \min \left( \|V_\cK^{vol}\|(U^\delta) , \|V_\cK^{pt}\|(U^\delta) \right) \: ,
\end{align*}
as wanted.
\end{proof}

\begin{remk}\label{remk:FVdiscrete}\rm 
We note that the first variation of a point cloud varifold is not a measure but only a distribution: indeed it is obtained by directional differentiation of a weighted sum of Dirac deltas. On the other hand, the first variation of a discrete volumetric varifold is bounded as soon as the cells in $\cK$ have a boundary with $\cH^{n-1}$ finite measure (or even as soon as the cells in $\cK$ have finite perimeter), but its total variation typically blows up as the size of the mesh goes to zero (see for instance Example~6 in \cite{Blanche-rectifiability}). Nevertheless, this bad behavior of the first variation, when applied to discrete varifolds, can be somehow controlled via regularization, as described in Section \ref{section:RFV}.
\end{remk}

Though the paper does not focus on polyhedral surfaces, we give a notion of polyhedral varifold which will be studied in Section~\ref{section:cotangent}, where we evidence the connection between the so--called \textit{Cotangent Formula} and the first variation of such a varifold.

\begin{dfn}[Polyhedral varifold] \label{dfn_polyhedralVarifold}
Let $M$ be a $d$--dimensional polyhedral surface whose set of $d$--faces is denoted $\cF$. Then, we associate $M$ with the $d$--varifold
\[
V = \sum_{F \in \cF} \cH^d_{| F} \otimes \delta_{P_F},
\]
where $P_F \in \G$ is the vector direction of the face $F$. Notice that $V$ is an integral $d$--varifold.
\end{dfn}

\section{Regularized First Variation}
\label{section:RFV}

Given a sequence of varifolds $(V_i)_i$ weakly--$\ast$ converging to a varifold $V$, a sufficient condition for $V$ to have locally bounded first variation, i.e. for $\delta V$ to be a Radon measure, is
\begin{equation} \label{allard_condition}
\sup_i \| \delta V_i \| < + \infty \: .
\end{equation}
However, the typical sequences of discrete varifolds that have been introduced in Section \ref{subsection:DV} may not have uniformly bounded first variations, or it may even happen that the first variations themselves are not measures, as in the case of point cloud varifolds (see Remark \ref{remk:FVdiscrete}). Nevertheless, $\delta V_i$ are distributions of order $1$ converging to $\delta V$ (in the sense of distributions). The idea is to compose the first variation operator $\delta$ with convolutions defined by a sequence of regularizing kernels $(\rho_{\e_{i}})_{i\in \N}$ as in Definition \ref{dfn:RegFirstVar} below, and then to require a uniform control on the $L^{1}$-norm of $\delta V_{i}\ast \rho_{\e_{i}}$. 

We also point out that the parameter $\e_{i}$ may be viewed as a ``scale parameter''. 

Besides some technical results, that will be used in the next sections, we prove in Theorem \ref{bounded_first_variation_condition_for_sequences} a compactness and rectifiability result, which relies on the assumption that $\delta V_{i}\ast \rho_{\e_{i}}$ is uniformly bounded in $L^{1}$.

As we are going to regularize the first variation of a varifold $V$ in $\Omega$ by convolution, we conveniently extend $\delta V$ to a linear and continuous form on $\xC^{1}_{c}(\R^{n},\R^{n})$. Let $\Omega \subset \R^n$ be an open set and $V$ be a $d$--varifold in $\Omega$ with mass $\V(\Omega)<+\infty$. First of all, we notice that $(x,S) \mapsto \Div_S X(x)$ is continuous and bounded, and that $V$ is a finite Radon measure, thus we set
\begin{equation}\label{extendedfv}
\delta V(X) = \int_{\Omega\times \G} \Div_S X(x) \, dV(x,S),\qquad \forall X\in \xC^{1}_{c}(\R^{n},\R^{n}).
\end{equation}
For more simplicity, in \eqref{extendedfv} the extended first variation is denoted as the standard first variation. We immediately obtain
\[
\delta V(X) \leq \|X\|_{\xC^1}\, \V (\Omega)\,, \qquad \forall X\in \xC^{1}_{c}(\R^{n},\R^{n})\,, 
\]
which means that the linear extension is continuous with respect to the $\xC^{1}$-norm. Notice that the extended first variation coincides with the standard first variation whenever $X\in \xC^{1}_{c}(\Omega,\R^{n})$ but contains additional boundary information if the support of $\V$ is not relatively compact in $\Omega$.

For the reader's convenience we recall from Section \ref{section:pre} that $\rho$ denotes a non-negative kernel profile, such that $\rho_{1}(x) = \rho(|x|)$ is of class $C^{1}$, has compact support in $B_{1}(0)$, and satisfies $\int \rho_{1}(x)\, dx = 1$. Then, given $\e>0$ we set $\rho_{\e}(x) = \e^{-n}\rho_{1}(x/\e)$.

\begin{dfn}[regularized first variation]\label{dfn:RegFirstVar} 
Given a vector field $X \in \xC_c^1 (\R^n , \R^n )$, for any $\e>0$ we define
\begin{equation}\label{def:regfirstvar}
\delta V \ast \rho_\epsilon (X) := \delta V (X \ast \rho_\epsilon) = \int_{\Omega \times \G}\xdiv_S (X \ast \rho_\epsilon) (y)\, dV(y,S) \: .
\end{equation}
We generically say that $\delta V \ast \rho_\epsilon$ is a \emph{regularized first variation} of $V$. 
\end{dfn}
Of course \eqref{def:regfirstvar} defines $\delta V \ast \rho_\epsilon$ in the sense of distributions. The following, elementary proposition shows that $\delta V \ast \rho_\epsilon$ is actually represented by a smooth vector field with bounded $L^{1}$-norm.
\begin{prop}\label{expression_of_delta_V^epsilon}
Let $\Omega \subset \R^n$ be an open set and $V$ be a $d$--varifold in $\Omega$ with finite mass $\V(\Omega)$. Then $\delta V \ast \rho_\epsilon$ is represented by the continuous vector field
\begin{equation} \label{convolution_general_kernel}
\delta V \ast \rho_\epsilon (x) =  \int_{\Omega \times \G} \nabla^S \rho_\epsilon (y-x) \, dV(y,S) = \frac{1}{\epsilon^{n+1}}\int_{\Omega \times \G} \nabla^S \rho_{1} \left(\frac{y-x}{\epsilon}\right) \, dV(y,S)
\end{equation}
and moreover one has $\delta V \ast \rho_\epsilon \in \xL^1(\R^n;\R^{n})$.
\end{prop}
\begin{proof}
Taking into account \eqref{def:regfirstvar}, for every $y\in \R^{n}$ we find $\Div_S (X \ast \rho_\epsilon)(y) = X \ast \nabla^S \rho_\epsilon(y) := \sum_{i=1}^{n} X_{i}\ast \partial^{S}_{i}\rho_{\e}(y)$, thus by Fubini-Tonelli's theorem we get
\begin{align*}
 \delta V \ast \rho_\epsilon ( X ) & = \int_{\Omega \times \G} (X \ast \nabla^S \rho_\epsilon)(y) \, dV(y,S)\\
& = \int_{\Omega \times \G} \int_{x \in \R^n} X(x)\cdot \nabla^S \rho_\epsilon(y-x) \, d \cL^n(x) \, dV(y,S) \\
& = \int_{x \in\R^n} X(x) \cdot \left( \int_{\Omega \times \G}   \nabla^S \rho_\epsilon(y-x) \, dV(y,S) \right) \, d \cL^n (x) \: ,
\end{align*}
which proves \eqref{convolution_general_kernel}. The fact that $\delta V \ast \rho_\epsilon \in \xL^1(\R^n;\R^{n})$ is an immediate consequence of the fact that $\nabla\rho_{\e}$ is bounded on $\R^{n}$.
\end{proof}

\begin{remk}\rm 
We stress that $\delta V \ast \rho_\epsilon$ is in $L^{1}(\R^{n})$ even when $\delta V$ is not locally bounded.
\end{remk}

\begin{remk}\rm 
If the support of $\|V\|$ is compactly contained in $\Omega$ then using the extended or the standard first variation in the convolution $\delta V\ast \rho_{\e}$ is equivalent up to choosing $\e$ small enough. In general, the same equivalence holds up to restricting the distribution $\delta V\ast \rho_{\e}$ to $\xC^{1}_{c}(\Omega_{\e},\R^{n})$, where $\Omega_{\e} = \{x\in \Omega:\ \dist(x,\partial \Omega)>\e\}$, which amounts to restricting the function $\delta V\ast \rho_{\e}$ to $\Omega_{\e}$. 
\end{remk}

In the next proposition we show that the classical first variation of a varifold $V$ is the weak--$*$ limit of regularized first variations of $V$, under the assumption that $\delta V$ is a bounded measure. This will immediately follow from the basic estimate \eqref{fvweakstar1}, which is true for all varifolds.
\begin{prop} \label{convergence_regularized_first_variation}
Let $\Omega \subset \R^n$ be an open set and let $V$ be a varifold in $\Omega$ with $\V (\Omega) < + \infty$. Then for any $X \in \xC_c^1 (\R^n, \R^n)$ we have
\begin{equation}\label{fvweakstar1}
\left| \delta V \ast \rho_\epsilon (X) - \delta V (X) \right| \leq \V \Big(\Omega \cap \big(\supp X + B_\epsilon(0)\big) \Big) \left\| \rho_\epsilon \ast X - X \right\|_{\xC^1} \xrightarrow[\epsilon \to 0]{} 0 \: .
\end{equation}
Moreover, if $V$ has bounded extended first variation then
\begin{equation}\label{fvweakstar2}
\delta V \ast \rho_\epsilon \xrightharpoonup[\epsilon \to 0]{\: \ast \:} \delta V \: .
\end{equation}
\end{prop}

\begin{proof}
Let $X \in \xC_c^1 (\R^n, \R^n)$. Since $\delta V \ast \rho_\epsilon (X) =  \delta V (\rho_\epsilon \ast X)$ we obtain
\[
\left| \delta V \ast \rho_\epsilon (X) - \delta V (X) \right| = \left| \delta V(\rho_\epsilon \ast X - X ) \right| \leq \V (\Omega) \left\| \rho_\epsilon \ast X - X \right\|_{\xC^1} \: .
\]
On observing that $\left\| \rho_\epsilon \ast X - X \right\|_{\xC^1} \xrightarrow[\epsilon \to 0]{} 0$ we get \eqref{fvweakstar1}. If in addition $V$ has bounded extended first variation, then for all $X \in \xC_c^0 (\R^n, \R^n)$ we obtain
\[
\left| \delta V \ast \rho_\epsilon (X) - \delta V (X) \right| \leq \| \delta V \| \left\| \rho_\epsilon \ast X - X \right\|_{\infty} \xrightarrow[\epsilon \to 0]{} 0 \: ,
\]
which proves \eqref{fvweakstar2}.
\end{proof}

The next theorem is a partial generalization of Allard's compactness theorem for rectifiable varifolds. It shows that, given an infinitesimal sequence $(\e_i)_i$ of positive numbers and a sequence of $d$-varifolds $(V_i)_i$ with uniformly bounded total masses, such that $\delta V_i\ast \rho_{\e_i}$ satisfies a uniform boundedness assumption, there exists a subsequence of $V_i$ that weakly-$\ast$ converges to a limit varifold $V$ with bounded first variation. If in addition $\|V_i\|(B_{r}(x))\ge \theta_{0}r^{d}$ for $\V$-almost every $x$ and for $\beta_{i}\le r\le r_{0}$, with $(\beta_{i})_{i\in \N}$ an infinitesimal sequence, then the limit varifold $V$ is rectifiable. We stress that $V_i$ is required neither to have bounded first variation, nor to be rectifiable. Notice also the appearance of the scale parameters $\beta_{i}$ providing infinitesimal lower bounds on the radii to be used for approximate density estimates.
\begin{theo}[compactness and rectifiability] \label{bounded_first_variation_condition_for_sequences}
Let $\Omega \subset \R^n$ be an open set and $(V_i)_i$ be a sequence of $d$-varifolds. Assume that there exists a positive, decreasing and infinitesimal sequence $(\epsilon_i)_i$, such that
\begin{equation} \label{dynamical_condition_bounded_first_variation}
M := \sup_{i\in \N} \left\lbrace \|V_i\|(\Omega) +  \|\delta V_i \ast \rho_{\e_i}\|_{\xL^1} \right\rbrace < +\infty \: .
\end{equation}
Then there exists a subsequence $(V_{\phi(i)})_i$ weakly--$\ast$ converging in $\Omega$ to a $d$-varifold $V$ with bounded first variation, such that $\V(\Omega) + | \delta V |(\Omega) \le M$. Moreover, if we further assume the existence of an infinitesimal sequence $\beta_i \downarrow 0$ and $\theta_0,r_0>0$ such that, for any $\beta_i<r<r_0$ and for $\|V_i\|$-almost every $x\in \Omega$, 
\begin{equation}\label{estbelow}
\|V_i\|(B_r(x)) \ge \theta_0 r^d\: ,
\end{equation}
then $V$ is rectifiable.
\end{theo}

\begin{proof}
Since $M$ is finite, there exists a subsequence $(V_{\phi(i)})_i$ weakly--$\ast$ converging in $\Omega$ to a varifold $V$. By Proposition \ref{convergence_regularized_first_variation}, for any $X \in \xC_c^1 (\Omega, \R^n)$ we obtain
\begin{align*}
\left| \delta V_{ \phi(i) } \ast \rho_{\epsilon_{ \phi(i) }} (X) - \delta V (X) \right| & \leq \left| \delta V_{ \phi(i) } \ast \rho_{\epsilon_{ \phi(i) }} (X) - \delta V_{ \phi(i) } (X) \right| + \left| \delta V_{ \phi(i) } (X) - \delta V (X) \right| \\
& \leq \underbrace{\|V_i\| (\Omega)}_{\leq C < +\infty} \, \left\| X \ast \rho_{\epsilon_{ \phi(i) }} - X \right\|_{\xC^1} + \left| \delta V_{ \phi(i) } (X) - \delta V (X) \right| \\
& \xrightarrow[i \to \infty]{} 0 \: . \\
\end{align*}
Consequently, for any $X \in \xC_c^1 (\Omega , \R^n)$ one has $\displaystyle \left| \delta V (X) \right| \leq \sup_i \left\| \delta V_i \ast \rho_{\epsilon_i} \right\|_{\xL^1} \, \| X \|_{\infty}$. We conclude that $\delta V$ extends into a continuous linear form in $\xC_c^0 (\Omega , \R^n)$ whose norm is bounded by $ \sup_i \left\| \delta V_i \ast \rho_{\epsilon_i}  \right\|_{\xL^1}$, thus $\V(\Omega) + | \delta V |(\Omega) \le M$.

Assuming the additional hypothesis \eqref{estbelow}, it is not difficult to pass to the limit and prove the same inequality for $\V$--a.e. $x$ and for all $0<r<r_0$. We refer to Proposition $3.3$ in \cite{Blanche-rectifiability} for more details on this point. By Theorem 5.5(1) in \cite{Allard72} we obtain the last part of the claim.
\end{proof}

\section{Approximate Mean Curvature}\label{amc}
\label{section:AMC}

\subsection{Definition and convergence}
We now introduce the notion of \textit{approximate mean curvature} associated with $V$, in a consistent way with the notion of regularized first variation. We refer to Section \ref{section:pre} for the notations and the basic assumptions on the kernel profiles $\rho,\xi$. We also set
\begin{equation} \label{eqCrhoCxi}
C_\rho = d\, \omega_{d} \int_0^1 \rho (r) r^{d-1} \, dr\,, \qquad C_\xi = d\, \omega_{d} \int_0^1 \xi (r) r^{d-1} \, dr \: .
\end{equation}

\begin{dfn}[approximate mean curvature]\label{dfn_regularizedMeanCurvature}
Let $\Omega \subset \R^n$ be an open set and let $V$ be a $d$--varifold in $\Omega$. For every $\e>0$ and $x\in \Omega$, such that $\V\ast \xi_{\e}(x)>0$, we define
\begin{equation}\label{HVgen}
H^{V}_{\rho,\xi,\e}(x) = - \frac{C_{\xi}}{C_{\rho}}\, \frac{\delta V\ast \rho_{\e}(x)}{\V \ast \xi_{\e}(x)}\,,
\end{equation}
where $C_{\rho}$ and $C_{\xi}$ are as in \eqref{eqCrhoCxi}. We generically say that the vector $H^{V}_{\rho,\xi,\e}(x)$ is an \emph{approximate mean curvature} of $V$ at $x$.
\end{dfn}

\begin{xmpl}[approximate mean curvature of a point cloud varifold]\rm 
\label{ex:pointcloud}
Let us consider a point cloud varifold $V = \sum_{j=1}^N m_j \delta_{x_j} \otimes \delta_{P_j}$. We remark that $\delta V$ is not a measure. An approximate mean curvature of $V$ is given by the formula
\begin{equation} \label{eq_pointCloudMC}
H_{\rho,\xi,\e}^{V} (x)= - \frac{\delta V \ast \rho_\epsilon (x)}{\V \ast \xi_\epsilon (x)} = \frac{1}{\epsilon} \frac{\displaystyle \sum_{x_j \in B_\epsilon(x)\setminus \{x\}} m_j\, \rho^\prime \left(\frac{|x_j -x|}{\epsilon}\right) \Pi_{P_j}\frac{x_j-x}{|x_j-x|}}{\displaystyle \sum_{x_j \in B_\epsilon(x)} m_j\, \xi \left(\frac{|x_j -x|}{\epsilon}\right)} \: .
\end{equation}
The formula is well-defined for instance when $x=x_{i}$ for some $i=1,\dots,N$.
The choice of $\epsilon$ here is crucial: it must be large enough to guarantee that the ball $B_{\e}(x)$ contains points of the cloud different from $x$, but not too large to avoid over-smoothing. 
\end{xmpl}

If $\delta V$ is locally bounded then we recall the Radon-Nikodym-Lebesgue decomposition \eqref{RNLdecomp}, which says that $\delta V = -H\V + \delta V_{s}$, where $H = H(x)$ is the generalized mean curvature of $V$.
Note that the approximate mean curvature introduced in Definition~\ref{dfn_regularizedMeanCurvature} can be equivalently defined as the Radon--Nikodym derivative of the regularized first variation with respect to the regularized mass of $V$.
When $V$ is rectifiable, it turns out that formula \eqref{HVgen} gives a pointwise $\V$--almost everywhere approximation of $H(x)$, as proved by the following result. 
\begin{theo}[Approximation I]
\label{thm:convergence1}
Let $\Omega \subset \R^n$ be an open set and let $V = v(M,\theta)$ be a rectifiable $d$--varifold with locally bounded first variation in $\Omega$. Then for $\V$--almost all $x \in \Omega$ we have 
\begin{equation} \label{eq:PointwiseCVfixedV}
 H^V_{\rho,\xi,\epsilon} (x) \xrightarrow[\epsilon \to 0]{} H(x) \,.
\end{equation}
\end{theo}

\begin{proof}
For $\V$--almost all $x \in \Omega$ we have
\begin{align}\notag
\Big| H_{\rho,\xi,\epsilon} (x) &- H(x) \Big |  = \left| - \frac{C_\xi}{C_\rho} \frac{\left( - H \V + \delta V_s \right) \ast \rho_\e (x)}{\V \ast \xi_\e (x)} - H(x) \right|  \\\notag
 &\leq \left| \frac{C_\xi}{C_\rho} \frac{\left( H \V \right) \ast \rho_\e (x)}{\V \ast \xi_\e (x)} - H(x) \frac{C_\xi}{C_\rho} \frac{ \V  \ast \rho_\e (x)}{\V \ast \xi_\e (x)} \right|  + \left| \frac{C_\xi}{C_\rho} \frac{ \V  \ast \rho_\e (x)}{\V \ast \xi_\e (x)} H(x) - H(x) \right| \\\notag 
 &\qquad + \frac{C_\xi}{C_\rho} \frac{| \delta V_s \ast \rho_\e (x) |}{\V \ast \xi_\e (x)}  \\\notag
&\leq \frac{C_\xi}{C_\rho} \frac{1}{\V \ast \xi_\e (x)} \int_{y \in \R^n} \hspace*{-0.3cm} \left| H(x) - H(y) \right|  \rho_\epsilon (x-y) \, d\V(y) \\ 
&\qquad + | H(x) | \left| \frac{C_\xi}{C_\rho} \frac{ \V  \ast \rho_\e (x)}{\V \ast \xi_\e (x)} - 1 \right| + \frac{| \delta V_s | \ast \rho_\e (x)}{\V \ast \xi_\e (x)} \: .  \label{eq:PointwiseCV1}
\end{align}
Being $V$ rectifiable, we can assume without loss of generality that the approximate tangent plane is defined at $x$, whence
\begin{align}
\label{eq:PointwiseCVdensity}
\e^{-d} & \V (B_\e (x)) \xrightarrow[\e \to 0]{} \omega_d \theta (x)\,, \\
\label{Cirho}
\epsilon^{n-d} & \V \ast \rho_\epsilon (x) = \frac{1}{\epsilon^d} \int_\Omega \rho \left( \frac{y-x}{\epsilon} \right) \, d \V(y) \xrightarrow[\epsilon \to 0]{} \theta (x) \int_{T_x M} \rho (y) \,d \cH^d (y) = C_\rho \theta (x) > 0 \: .
\end{align}
Thus by \eqref{eq:PointwiseCVdensity} and \eqref{Cirho} we obtain
\begin{equation} \label{eq:PointwiseCVthirdTerm}
\left| \frac{C_\xi}{C_\rho} \frac{ \V  \ast \rho_\e (x)}{\V \ast \xi_\e (x)} - 1 \right| \xrightarrow[\e \to 0]{} 0 \: .
\end{equation}
Again by \eqref{eq:PointwiseCVdensity} and \eqref{Cirho}, and for $\V$-almost any $x \in \Omega$ (precisely, at any Lebesgue point $x$ of $H \in \xL^1 (\V)$), we get
\begin{align}
\frac{1}{ \V \ast \xi_\epsilon (x) } & \int_{y \in \R^n} \left| H(x) - H(y) \right|  \rho_\epsilon (x-y) \, d\V(y) \nonumber \\
& \leq \frac{\| \rho \|_{\infty}\V (B_\epsilon (x))}{\epsilon^n \V \ast \xi_\epsilon (x) } \cdot \frac{1}{\V(B_\epsilon(x))} \int_{y \in B_\epsilon (x)} \left| H(x) - H(y) \right|  \, d\V(y) \nonumber \\
& = \| \rho \|_{\infty} \underbrace{ \frac{\epsilon^{-d} \V (B_\epsilon (x))}{\epsilon^{n-d} \V \ast \xi_\epsilon (x) } }_{\displaystyle \xrightarrow[\epsilon \to 0]{}\frac{\omega_d}{C_\xi}}  \underbrace{ \fint_{y \in B_\epsilon (x)} \left| H(x) - H(y) \right|  \, d\V(y) }_{\displaystyle \xrightarrow[\epsilon \to 0]{} 0 }\  \xrightarrow[\epsilon \to 0]{} 0 \: . \label{eq:PointwiseCVfirstTerm}
\end{align}
Similarly, for $\V$--almost every $x$ we get
\begin{equation} \label{eq:PointwiseCVsecondTerm}
\frac{\left| \delta V_s \right| \ast \rho_\epsilon (x) }{ \V \ast \xi_\epsilon (x) } \leq \| \rho \|_{\infty} \underbrace{ \frac{\epsilon^{-d} \V (B_\epsilon (x))}{\epsilon^{n-d} \V \ast \xi_\epsilon (x) } }_{\displaystyle \xrightarrow[\epsilon \to 0]{}\frac{\omega_d}{C_\xi}} \: \underbrace{ \frac{| \delta V_s |(B_\epsilon (x))}{\V(B_\epsilon(x))} }_{\displaystyle \xrightarrow[\epsilon \to 0]{} 0 } \xrightarrow[\epsilon \to 0]{} 0 \: .
\end{equation}
Then \eqref{eq:PointwiseCVthirdTerm}, \eqref{eq:PointwiseCVfirstTerm}, \eqref{eq:PointwiseCVsecondTerm}, and \eqref{eq:PointwiseCV1} yield \eqref{eq:PointwiseCVfixedV}. 

\end{proof}

Given a rectifiable $d$-varifold $V$ with locally bounded first variation and a sequence of generic $d$-varifolds $(V_{i})_{i}$ weakly--$*$ converging to $V$, our goal is now to determine an infinitesimal sequence of \textit{regularization scales} $(\e_{i})_{i}$, in dependence of an infinitesimal sequence $(d_i)_{i}$ measuring how well the $V_{i}$'s are locally approximating $V$, in order to derive an asymptotic, quantitative control of the error between the approximate mean curvatures of $V_{i}$ and $V$. In this spirit we obtain two convergence results, Theorem~\ref{thm:convergence2} and Theorem~\ref{thm:convergence3} that we describe hereafter.

In Theorem \ref{thm:convergence2} we extend the basic convergence property proved in Theorem \ref{thm:convergence1}. More specifically we show the pointwise convergence of $H^{V_{i}}_{\rho,\xi,\e_{i}}$ to $H$ as $i\to \infty$, up to an infinitesimal offset and for a suitable choice of $\e_{i}>0$ tending to zero as $i\to\infty$. The presence of an offset in the evaluation of $H^{V_{i}}_{\rho,\xi,\e_{i}}$ and $H$ (that is, we compare $H^{V_{i}}_{\rho,\xi,\e_{i}}(z_{i})$ with $H(x)$, where $z_{i}$ is a sequence of points converging to $x$) is motivated by the fact that we do not have $\supp \| V_i \| \subset \supp \V$ in general. Moreover, in typical applications one first constructs the varifold $V_{i}$ (which for instance could be a varifold solving some ``discrete approximation'' of a geometric variational problem or PDE) and then, by possibly applying Theorem \ref{bounded_first_variation_condition_for_sequences}, one infers the existence of a limit varifold $V$ of the sequence $(V_{i})_{i}$, up to extraction of a subsequence. In this sense, $V_{i}$ is typically explicit while $V$ is not. We also provide in \eqref{epsdeltaboh} an asymptotic, quantitative estimate of the gap between $H^{V_{i}}_{\rho,\xi,\e_{i}}(z_{i})$ and $H^{V}_{\rho,\xi,\e_{i}}(x)$ (notice that for this estimate we take the $\e_{i}$-regularized mean curvatures for both varifolds $V_{i}$ and $V$) in terms of the parameters $\e_{i},d_{i}$ and of the offset $|x-z_{i}|$. We stress that the regularity of $V$ that is assumed in Theorem \ref{thm:convergence2} is in some sense minimal (for instance the singular part $\delta V_{s}$ of the first variation may not be zero). The price to pay for such a generality is a non-optimal convergence rate, which can be improved under stronger regularity assumptions on $V$ and by using a modified notion of approximate mean curvature (see Definition \ref{dfn_modifiedMC} and Theorem~\ref{thm:convergence3}). 

From now on we require a few extra regularity on the pair of kernel profiles $(\rho,\xi)$, according to the following hypothesis. 
\begin{hyp}\label{rhoepsxieps}
We say that the pair of kernel profiles $(\rho,\xi)$ satisfies Hypothesis \ref{rhoepsxieps} if $\rho,\xi$ are as specified in Section \ref{section:pre} and, moreover, $\rho$ is of class $W^{2,\infty}$ while $\xi$ is of class $W^{1,\infty}$. 
\end{hyp}

We begin with a technical lemma providing the key estimates that are needed in the proofs of Theorems~\ref{thm:convergence2} and \ref{thm:convergence3}.

\begin{lemma} \label{lemma_technical_pointwise_convergence}
Let $\Omega \subset \R^n$ be an open set and let $V = v(M,\theta)$ be a rectifiable $d$--varifold in $\Omega$ with locally bounded first variation. Let $(\rho, \xi)$ satisfy Hypothesis~\ref{rhoepsxieps}.
Let $(V_i)_i$ be a sequence of $d$--varifolds.
Then, for every $0 < \epsilon <1$, for $\V$--almost every $x$, and for every sequence $z_i \to x$, one has
\begin{align}
\epsilon^n \Big| \| V_i \| \ast \xi_\epsilon (z_i) - \V \ast \xi_\epsilon (x) \Big| & \leq \frac{1}{\epsilon} \| \xi \|_{\xW^{1,\infty}} \left( \Delta^{1,1}_{B_{\epsilon + |x-z_i|} (x)} (\| V_i \| , \V) + |x-z_i|  \V \left( B_{\epsilon + |x - z_i|} (x) \right) \right) \label{eq_lemma_Vzi}
\end{align}
and
\begin{align}
\epsilon^n \Big| \delta V_i \ast \rho_\epsilon (z_i) - \delta V \ast \rho_\epsilon (x) \Big| & \leq
 \frac{1}{\epsilon^2} \| \rho \|_{\xW^{2,\infty}} \left( \Delta^{1,1}_{B_{\epsilon+ |x-z_i| } (x)} (V_i,V) + |x-z_i| \V \left( B_{\epsilon + |z_i - x|} (x)\right) \right).
 \label{eq_lemma_deltaVzi}
\end{align}
Moreover, if there exist two decreasing and infinitesimal sequences $(d_i)_i, \,(\eta_i)_i$, such that for any ball $B\subset\Omega$ centered at $x$ one has
\begin{equation} \label{eq_delta11_V}
\Delta^{1,1}_B(\|V\|,\|V_i\|) \leq d_i \min \Big( \V(B^{\eta_i}) , \| V_i \| (B^{\eta_i}) \Big) \,,
\end{equation}
then for any infinitesimal sequence $(\e_{i})_{i}$, such that $\eta_i + d_i + |x-z_i| = o(\e_i)$ as $i\to\infty$, one has  
\begin{equation} \label{eq_lemma_masszi}
\lim_{i\to\infty} \frac{\epsilon_i^n \|V_i\| \ast \xi_{\epsilon_i}(z_i)}{\V(B_{\epsilon_i} (x))}\ =\ \omega_{d}^{-1} \int_{B^{d}_{1}}\xi(|z|)\, d\cH^{d}(z)\,.  
\end{equation}
\end{lemma}

\begin{proof}
We start with the proof of \eqref{eq_lemma_Vzi}. By definition of $\Delta^{1,1}_B$, for all $\phi \in \xLip (\Omega )$ such that $\supp \phi \subset B$,
\begin{equation} \label{eq_recall_lipschitz_convergence_with_holder_condition}
\left| \int \phi(y)\, d \| V_i \|(y)  - \int \phi(y)\, d \V(y) \right| \leq  (\| \phi \|_{\infty} + \xlip (\phi))\, \Delta^{1,1}_B (\|V_i\|,\V) \: .
\end{equation}
Since the function $y \mapsto \xi \left( \frac{|y-z_i|}{\epsilon} \right)$ is $\frac{1}{\epsilon} \xlip(\xi)$--Lipschitz and supported in $B_{\epsilon + |z_i - x|} (x)$, we have
\begin{align}
\epsilon^n \Big| \| V_i \| \ast \xi_\epsilon (z_i) - \V \ast \xi_\epsilon (z_i) \Big| & = \left| \int_\Omega \xi  \left( \frac{|y-z_i|}{\epsilon} \right) \, d \| V_i \| (y) - \int_\Omega \xi  \left( \frac{|y-z_i|}{\epsilon} \right) \, d \V (y) \right| \nonumber \\
& \leq \left( \| \xi \|_{\infty} + \frac{1}{\epsilon} \xlip(\xi) \right) \Delta^{1,1}_{B_{\epsilon + |x-z_i|} (x)} (\| V_i \| , \V) \nonumber \\
& \leq \frac{1}{\epsilon} \| \xi \|_{\xW^{1,\infty}} \Delta^{1,1}_{B_{\epsilon + |x-z_i|} (x)} (\| V_i \| , \V) \quad \text{since } \epsilon \leq 1 \: . \label{proof_technicalLemma_1}
\end{align}
Then we have
\begin{align}
\epsilon^n \Big| \V \ast \xi_\epsilon (z_i) - \V \ast \xi_\epsilon (x) \Big| & \leq \left| \int_\Omega \xi \left( \frac{|y-z_i|}{\epsilon} \right) \, d \V(y) - \int_\Omega \xi \left( \frac{|y-x|}{\epsilon} \right) \, d \V(y) \right| \nonumber \\
& \leq \xlip (\xi) \frac{1}{\epsilon} \int_{B_{\epsilon + |x - z_i|} (x)} \left( \Big| |y-x| - |y-z_i| \Big| \right) \, d \V(y) \nonumber \\
& \leq \frac 1\e \| \xi \|_{\xW^{1,\infty}} \big|x-z_i\big| \V \left( B_{\epsilon + |x - z_i|} (x) \right) \:  \label{proof_technicalLemma_2}
\end{align}
By combining \eqref{proof_technicalLemma_1} and \eqref{proof_technicalLemma_2} we get \eqref{eq_lemma_Vzi}.

We similarly prove \eqref{eq_lemma_deltaVzi}. The mapping
$ \displaystyle
(y,S ) \in \Omega \times \G \mapsto \nabla^S \rho_1 \left(\frac{y-z_i}{\epsilon}\right)
$ 
has a Lipschitz constant $\displaystyle \leq \frac{1}{\epsilon} \| \rho \|_{\xW^{2,\infty}}$ and support in $B_{\epsilon + |x-z_i|}(x) \times \G$, therefore
\begin{align}
\epsilon^n \Big| \delta V_i \ast \rho_\epsilon (z_i) - \delta V \ast \rho_\epsilon (z_i)  \Big| & = \frac{1}{\epsilon} \left| \int_{\Omega \times \G} \hspace{-13.2pt} \nabla^S \rho_1 \left(\frac{y-z_i}{\epsilon}\right) \, dV_i(y,S) - \int_{\Omega \times \G} \hspace{-13.2pt} \nabla^S \rho_1 \left(\frac{y-z_i}{\epsilon}\right) \, dV(y,S)  \right| \nonumber \\
& \leq \frac{1}{\epsilon^{2}} \| \rho \|_{\xW^{2,\infty}} \Delta^{1,1}_{B_{\epsilon + |x-z_i|} (x)} (V_i , V) \: . \label{proof_technicalLemma_3}
\end{align}
Moreover one has
\begin{align}
\epsilon^n \Big| \delta V \ast \rho_\epsilon (z_i) - \delta V \ast \rho_\epsilon (x)  \Big| & = \frac{1}{\epsilon} \left| \int_{\Omega \times \G} \Pi_S \left( \nabla \rho_1 \left(\frac{y-z_i}{\epsilon}\right) - \nabla \rho_1 \left(\frac{y-x}{\epsilon}\right) \right) \, dV(y,S) \right| \nonumber \\
& \leq  \frac{1}{\epsilon^2} \xlip (\nabla \rho_1) \int_{B_{\epsilon + |z_i-x|} (x) \times \G} \Big| |x-y| - |z_i - y|  \Big| \, dV(y,S) \nonumber \\
& \leq \frac{|x-z_i|}{\epsilon^2} \| \rho \|_{\xW^{2,\infty}} \V \left( B_{\epsilon + |z_i-x|} (x) \right) \: . \label{proof_technicalLemma_4}
\end{align}
By combining \eqref{proof_technicalLemma_3} and \eqref{proof_technicalLemma_4} we get \eqref{eq_lemma_deltaVzi}.

We finally prove \eqref{eq_lemma_masszi}. We take $x$ in the support of $\V$, such that the approximate tangent plane $T_{x}M$ is defined and $x$ is a Lebesgue point for the multiplicity function $\theta$ (of course we can also assume $\theta(x) >0$). We thus have 
\[
\e_{i}^{n}\V\ast \xi_{\e_{i}}(x) = \int \xi(|y-x|/\e_{i})\, d\V \sim \e_{i}^{d}\theta(x) \int_{T_{x}M} \xi(|z|)\, d\cH^{d}(z)\,,
\]
where $a_{i}\sim b_{i}$ means that $a_{i} = b_{i} + o(b_{i})$ as $i\to\infty$.
Then we notice that 
\[
\V(B_{\e_{i}}(x)) \sim \omega_{d}\e_{i}^{d} \theta(x)\,.
\]
By combining the two relations above we obtain
\begin{equation}\label{eq:controllo1}
\e_{i}^{n}\V\ast \xi_{\e_{i}}(x) \sim \omega_{d}^{-1} \int_{B^{d}_{1}}\xi(|z|)\, d\cH^{d}(z)\, \V(B_{\e_{i}}(x))\,,
\end{equation}
which corresponds to \eqref{eq_lemma_masszi} in the special case $z_{i}=x$ and $V_{i}=V$ for all $i$. On the other hand the general case is easily proved as soon as we check that 
\[
A_{i} = \Big|\V\ast\xi_{\e_{i}}(z_{i}) - \|V_{i}\|\ast \xi_{\e_{i}}(z_{i})\Big|\quad
\text{and}\quad
B_{i} = \Big|\V\ast\xi_{\e_{i}}(z_{i}) - \V\ast \xi_{\e_{i}}(x)\Big|
\]
satisfy 
\begin{equation}\label{aipiubiopiccolo}
A_{i}+B_{i} = o(\e_{i}^{d-n})\qquad \text{when $i\to\infty$.}
\end{equation}
Indeed we first notice that, since $\xlip(\xi_{\e_{i}}) = \e_{i}^{-n-1}\xlip(\xi)$ and owing to \eqref{eq_delta11_V}, we have up to multiplicative constants
\begin{align}\notag
A_{i} &\le \e_{i}^{-n-1} \Delta^{1,1}_{B_{\e_{i}}(z_{i})}(\V,\|V_{i}\|)\\\notag 
&\le \e_{i}^{-n-1} d_{i}\V(B_{\e_{i}+\eta_{i}}(z_{i}))\\\notag 
&\le \e_{i}^{-n-1} d_{i}\V(B_{\e_{i}+\eta_{i}+|z_{i}-x|}(x))\\\label{ineg-Ai} 
&\le \e_{i}^{d-n-1}\, d_{i}\,.
\end{align}
Then, we notice that 
\begin{align}\notag
B_{i} &= \left|\int (\xi(|y-x|/\e_{i}) - \xi(|y-z_{i}|/\e_{i})\, d\V \right|\\\notag 
&\le \xlip(\xi) |x-z_{i}|\, \frac{\V(B_{\e_{i}+|x-z_{i}|}(x))}{\e_{i}^{n+1}}\\\label{ineg-Bi} 
&\sim C \e_{i}^{d-n-1}\, |x-z_{i}|
\end{align}
where $C = \omega_{d}\,\xlip(\xi)\theta(x)$. Finally, by combining \eqref{ineg-Ai} and \eqref{ineg-Bi} we conclude that for $i$ large enough
\[
A_{i}+B_{i} \le C \e_{i}^{d-n-1}(d_{i}+|x-z_{i}|) = o(\e_i^{d-n})\,,
\]
i.e., that \eqref{aipiubiopiccolo} holds true. This proves \eqref{eq_lemma_masszi} at once.
\end{proof}

\begin{theo}[Convergence II] \label{thm:convergence2}
Let $\Omega \subset \R^n$ be an open set and let $V = v(M,\theta)$ be a rectifiable $d$--varifold in $\Omega$ with bounded first variation. Let $(\rho, \xi)$ satisfy  Hypothesis~\ref{rhoepsxieps}.
Let $(V_i)_i$ be a sequence of $d$--varifolds, for which there exist two positive, decreasing and infinitesimal sequences $(\eta_i)_i,(d_i)_i$, such that for any ball $B\subset\Omega$ centered in $\supp \V$, one has
\begin{equation}  \label{eq_thm_pointwise_cv_hyp2}
\Delta^{1,1}_B(V,V_i) \leq d_i \min \big( \V(B^{\eta_i}) , \| V_i \|(B^{\eta_i}) \big) \,.
\end{equation}
For $\V$--almost any $x \in \Omega$ and for any sequence $(z_i)_{i}$ tending to $x$, let $(\e_{i})_{i}$ be a positive, decreasing and infinitesimal sequence such that
\begin{equation} \label{eq_eps_condition}
\frac{d_i + |x-z_i|}{\epsilon_i^2} \xrightarrow[i \to \infty]{} 0 \quad \text{and} \quad \frac{\eta_i}{\epsilon_i} \xrightarrow[i \to \infty]{} 0\,.
\end{equation}
Then we have
\begin{align}
\label{epsdeltaboh}
\left| H_{\rho,\xi,\epsilon_i}^{V_i} (z_i) - H_{\rho,\xi,\epsilon_i}^V(x) \right| &\leq C \| \rho \|_{\xW^{2,\infty}} \frac{d_i + |x-z_i|}{\epsilon_i^2}\qquad \text{for $i$ large enough,} \\\label{convergenceboh}
H_{\rho,\xi,\epsilon_i}^{V_i} (z_i)  \xrightarrow[i \to \infty]{} H(x) \: .
\end{align}
\end{theo}

\begin{proof}
We focus on the proof of \eqref{epsdeltaboh}. We have
\begin{multline}
\left| H_{\rho,\xi,\epsilon_i}^{V_i} (z_i)  - H_{\rho,\xi,\epsilon_i}^V(x) \right|  \leq \frac{C_\xi}{C_\rho} \left|  \frac{\delta V_i \ast \rho_{\epsilon_i} (z_i)}{\| V_i \| \ast \xi_{\epsilon_i} (z_i)} - \frac{\delta V \ast \rho_{\epsilon_i} (x)}{\V \ast \xi_{\epsilon_i} (x)} \right|   \\
 \leq \frac{C_\xi}{C_\rho} \frac{\left| \delta V_i \ast \rho_{\epsilon_i} (z_i) - \delta V \ast \rho_{\epsilon_i} (x)  \right|}{\| V_i \| \ast \xi_{\epsilon_i} (z_i)}  + \frac{C_\xi}{C_\rho} \left| \delta V \ast \rho_{\epsilon_i}(x) \right| \left| \frac{1}{\| V_i \| \ast \xi_{\epsilon_i} (z_i)} - \frac{1}{\V \ast \xi_{\epsilon_i} (x)} \right| \: . \label{eq_proof_pointwise_approximate_mean_curvature_1}
\end{multline}

We study the convergence of the first term in \eqref{eq_proof_pointwise_approximate_mean_curvature_1}. Notice that assumption \eqref{eq_thm_pointwise_cv_hyp2} implies \eqref{eq_delta11_V} since 
\begin{equation}  \label{eq_delta11VV}
\Delta^{1,1}_B (\|V_i\|,\V) \leq \Delta^{1,1}_B (V_i,V) \: .
\end{equation}
We conveniently set $\mu_{i} = \e_{i}+|x-z_{i}|+\eta_{i}$. Then 
owing to Lemma~\ref{lemma_technical_pointwise_convergence}  \eqref{eq_lemma_masszi} and  \eqref{eq_thm_pointwise_cv_hyp2}, for $i$ large enough we obtain 
\begin{align}
\frac{C_\xi}{C_\rho} & \frac{\left| \delta V_i \ast \rho_{\epsilon_i} (z_i) - \delta V \ast \rho_{\epsilon_i} (x)  \right|}{\| V_i \| \ast \xi_{\epsilon_i} (z_i)} \nonumber \\
& \leq \frac{C_\xi}{C_\rho} \frac{1}{\| V_i \| \ast \xi_{\epsilon_i} (z_i)} \frac{1}{\epsilon_i^{n+2}} \| \rho \|_{\xW^{2,\infty}} \Big( d_i \V \left(B_{\mu_{i}}(x) \right) + |x-z_i| \V \left(B_{\mu_{i}}(x) \right) \Big) \nonumber \\
& \leq \frac{4}{C_\rho} \| \rho \|_{\xW^{2,\infty}} \frac{d_i+ |x - z_i|}{\epsilon_i^2} \frac{\V \left(B_{\mu_{i}}(x) \right)}{\V \left(B_{\epsilon_i}(x) \right)} \nonumber \\
& \leq \frac{8}{C_\rho} \| \rho \|_{\xW^{2,\infty}} \frac{d_i + |x-z_i|}{\epsilon_i^2} \: , \label{proof_thm_pointwiseCV_1}
\end{align}
where in the last inequality we have used that 
\begin{equation} \label{eq_proof_thmCV_density}
\frac{\V \left(B_{\mu_{i}}(x \right))}{\V \left(B_{\epsilon_i}(x)\right)} = \frac{\mu_{i}^d}{\epsilon_i^d}  +o(1) \xrightarrow[i \to \infty]{} 1
\end{equation}
for $\V$-almost every $x$. It remains to study the second term in \eqref{eq_proof_pointwise_approximate_mean_curvature_1}. Applying Lemma~\ref{lemma_technical_pointwise_convergence} \eqref{eq_lemma_Vzi} and \eqref{eq_lemma_masszi} together with \eqref{eq_thm_pointwise_cv_hyp2}, \eqref{eq_delta11VV}, Theorem \ref{thm:convergence1}, and \eqref{eq_proof_thmCV_density},
we obtain for $i$ large enough
\begin{align}
\frac{C_\xi}{C_\rho} & \left| \delta V \ast \rho_{\epsilon_i}(x) \right|  \left| \frac{1}{\| V_i \| \ast \xi_{\epsilon_i}(z_i)} - \frac{1}{\V \ast \xi_{\epsilon_i}(x)} \right| \nonumber \\
 & =  \frac{C_\xi}{C_\rho} \frac{\left| \delta V \ast \rho_{\epsilon_i}(x) \right|}{\V \ast \xi_{\epsilon_i}(x)}  \frac{1}{\| V_i \| \ast \xi_{\epsilon_i}(z_i)} \Big| \V \ast \xi_{\epsilon_i}(x) - \| V_i \| \ast \xi_{\epsilon_i}(z_i) \Big| \nonumber \\
& \leq  \frac{C_\xi}{C_\rho} \frac{\left| \delta V \ast \rho_{\epsilon_i}(x) \right|}{\V \ast \xi_{\epsilon_i}(x)}  \frac{1}{\e_i^n \| V_i \| \ast \xi_{\epsilon_i}(z_i)} \frac{1}{\epsilon_i} \| \xi \|_{\xW^{1,\infty}} \Big( d_i \V(B_{\mu_{i}}(x)) + |x-z_i| \V(B_{\mu_{i}}(x)) \Big) \nonumber \\
& \leq \left| H_{\rho,\xi,\e_i}^V(x) \right| \frac{\V(B_{\epsilon_i}(x))}{\e_i^n \| V_i \| \ast \xi_{\epsilon_i}(z_i)} \frac{\V(B_{\mu_{i}}(x))}{\V(B_{\epsilon_i}(x))} \| \xi \|_{\xW^{1,\infty}} \frac{d_i + |x-z_i|}{\epsilon_i}  \nonumber \\
& \leq \frac{8\| \xi \|_{\xW^{1,\infty}}}{C_\xi} \big(|H(x)|+1\big) \frac{d_i + |x-z_i|}{\epsilon_i} \xrightarrow[i \to \infty]{} 0 \: . \label{eq_proof_pointwiseCV_2ndTerm}
\end{align}

Thanks to \eqref{proof_thm_pointwiseCV_1} and \eqref{eq_proof_pointwiseCV_2ndTerm}, for $\V$--almost any $x$ and for $i$ large enough (possibly depending on $x$) one has
\begin{align*}
\left| H_{\rho,\xi,\e_i}^{V_i} (z_i) - H_{\rho,\xi,\e}^V(x) \right| &\leq \frac{8}{C_\rho} \| \rho \|_{\xW^{2,\infty}} \frac{d_i + |x-z_i|}{\epsilon_i^2} + \frac{8\| \xi \|_{\xW^{1,\infty}}}{C_\xi} \big(|H(x)|+1\big) \frac{d_i + |x-z_i|}{\epsilon_i} \\
&= \frac{8}{C_\rho} \| \rho \|_{\xW^{2,\infty}} \frac{d_i+|x-z_i|}{\epsilon_i^2} + O\left(\frac{d_{i}+|x-z_i|}{\e_{i}}\right)\: ,
\end{align*}
which implies \eqref{epsdeltaboh}.

Finally, thanks to Theorem~\ref{thm:convergence1}, for $\V$--almost any $x$ we find
\begin{align*}
\left| H_{\rho,\xi,\epsilon_i}^{V_i} (z_i) - H(x) \right| \leq \left| H_{\rho,\xi,\epsilon_i}^{V_i} (z_i) - H_{\rho,\xi,\epsilon_i}^V(x) \right| + \underbrace{ \left| H_{\rho,\xi,\epsilon_i}^V (x) - H(x) \right| }_{\xrightarrow[\: \: i\to \infty \: \:]{} 0}\: ,
\end{align*}
which combined with \eqref{epsdeltaboh} gives \eqref{convergenceboh} at once.

\end{proof}

Below we prove a third, pointwise convergence result where a better convergence rate shows up when the limit varifold is (locally) a manifold $M$ of class $\xC^2$ endowed with multiplicity $=1$.
First we notice that $H_{\rho,\xi,\e}^V(x)$ is an integral of tangentially projected vectors, while the (classical) mean curvature of $M$ is a normal vector. This means that even small errors affecting the mass distribution of the approximating varifolds $V_{i}$ might lead to non-negligible errors in the tangential components of the approximate mean curvature. 
A workaround for this is, then, to project $H_{\rho,\xi,\e}^V(x)$ onto the normal space at $x$. In order to properly define the orthogonal component of the mean curvature of a general varifold $V$, we recall the Young measures-type representation of $V$ (see Proposition \ref{prop:Young}):
\[
V(\phi) = \int_{x\in \Omega}\int_{P\in \G} \phi(x,P)\, d\nu_{x}(P)\, d\V(x),\qquad \forall\, \phi\in C^{0}_{c}(\Omega\times \G)\,.
\]
At this point we can introduce the following definition.
\begin{dfn}[orthogonal approximate mean curvature] \label{dfn_modifiedMC}
Let $\Omega \subset \R^n$ be an open set and let $V$ be a $d$--varifold in $\Omega$. For $\V$-almost every $x$ an orthogonal approximate mean curvature of $V$ at $x$ is defined as
\begin{equation} 
H_{\rho,\xi,\e}^{V, \perp} (x) =  
 \int_{P \in \G} \Pi_{P^\perp} \left( H_{\rho,\xi,\e}^{V}(x) \right) \, d \nu_x(P)\: .\label{eq_modified_Heps_1} 
 \end{equation}
\end{dfn}

We first check a basic approximation property of the orthogonal approximate mean curvature (Proposition~\ref{prop_consistency_modifiedMC} below, an immediate consequence of Theorem~\ref{thm:convergence1} and of a classical result due to Brakke). Then, in Theorem~\ref{thm:convergence3} we prove a better convergence rate under stronger regularity assumptions on $V$ and sufficient accuracy in the approximation of $V$ by $V_{i}$.

\begin{prop} \label{prop_consistency_modifiedMC}
Let $\Omega \subset \R^n$ be an open set and let $V = v(M,\theta)$ be an integral $d$--varifold with bounded first variation $\delta V = - H \, \V + \delta V_s$. Then, for $\cH^{d}$--almost any $x \in M\cap \Omega$ we have
\begin{equation} 
 H^{V,\perp}_{\rho,\xi,\epsilon} (x) \xrightarrow[\epsilon \to 0]{} H(x) \: .
\end{equation}
\end{prop}

\begin{proof}

As $V$ is integral, we know from a result of Brakke \cite{brakke} that $H(x) \perp T_x M$ for $\cH^{d}$--almost every $x\in M\cap \Omega$. Thanks to Theorem~\ref{thm:convergence1}, for $\cH^{d}$--a. e. $x\in M\cap\Omega$ we conclude that 
\begin{align*}
\left| H_{\rho,\xi,\e}^{V, \perp} (x) - H(x) \right| & = \left| \Pi_{T_x M^\perp} \left( H_{\rho,\xi,\e}^V (x) \right) - \Pi_{T_x M^\perp} H(x) \right| \\
& \leq \left| H_{\rho,\xi,\e}^V (x) -  H(x) \right| \xrightarrow[\e \to 0]{} 0\,.
\end{align*}

\end{proof}

The orthogonal approximate mean curvature introduced in Definition~\ref{dfn_modifiedMC} is very sensitive to the pointwise estimate of the tangent space $T_x M$ at $x$. Therefore, it is not possible to use it for generalizing Theorem~\ref{thm:convergence2}, unless we know that $\nu_x^i$ is close enough to $\nu_x = \delta_{T_x M}$. Indeed, under this stronger assumption (see \eqref{eq_thm_pointwise_cvSmooth_hyp2} and \eqref{eq_thm_pointwise_cvSmooth_hyp3} below) we recover the following pointwise convergence result with a substantially improved convergence rate. 

\begin{theo}[Convergence III]\label{thm:convergence3}
Let $\Omega \subset \R^n$ be an open set, $M \subset \Omega$ be a $d$--dimensional submanifold of class $\xC^2$ without boundary, and let $V = v(M,1)$ be the rectifiable $d$--varifold in $\Omega$ associated with $M$, with multiplicity $1$. Let us extend $T_y M$ to a $\xC^1$ map $\widetilde{T_y M}$ defined in a tubular neighbourhood of $M$. Let $(V_i)_i$ be a sequence of $d$--varifolds in $\Omega$. Let $(\rho, \xi)$ satisfies Hypothesis~\ref{rhoepsxieps}.
Let $x \in M$ and let $(z_i)_i \subset \Omega$ be a sequence tending to $x$ and such that $z_i \in \supp \| V_i\|$. Assume that there exist positive, decreasing and infinitesimal sequences $(\eta_i)_i, \,(d_{1,i})_i, \, (d_{2,i})_i, \, (\e_i)_i $, such that for any ball $B\subset\Omega$ centered in $\supp \V$ and contained in a neighbourhood of $x$, one has
\begin{equation} \label{eq_thm_pointwise_cvSmooth_hyp2}
\Delta^{1,1}_B(\V,\|V_i\|) \leq d_{1,i} \min \left( \V(B^{\eta_i}) , \|V_i\|(B^{\eta_i}) \right) \: ,
\end{equation}
and, recalling the decomposition $V_i = \| V_i \| \otimes \nu_x^i$,
\begin{equation} \label{eq_thm_pointwise_cvSmooth_hyp3}
\sup_{ \{y \in B_{\epsilon_i + |x-z_i|} (x) \cap \supp  \| V_i \| \} } \int_{S \in \G} \| \widetilde{T_y M} - S \| \, d \nu_y^i (S) \leq d_{2,i} \: .
\end{equation}
Then, there exists $C >0$ such that
\begin{equation}
\left| H_{\rho,\xi,\epsilon_i}^{V_i,\perp} (z_i) - H_{\rho,\xi,\epsilon_i}^{V,\perp}(x) \right| \leq C \frac{d_{1,i} + d_{2,i} + |x-z_i|}{\epsilon_i} \,.
\end{equation}
Moreover, if we also assume that $d_{1,i}+d_{2,i}+\eta_{i}+|x-z_{i}| = o(\e_{i})$ as $i\to\infty$, then
\begin{equation*}
H_{\rho,\xi,\epsilon_i}^{V_i,\perp} (z_i) \xrightarrow[i \to \infty]{} H(x) \,.
\end{equation*}
\end{theo}

\begin{proof}
Let us set 
\begin{equation*}
a_{i} := \frac{1}{\e_{i}} \left|  \int_{y \in \Omega } \Pi_{\widetilde{T_{z_i} M}^\perp} \circ \Pi_{\widetilde{T_y M}}  \nabla \rho_1 \left( \frac{y-z_i}{\e_{i}} \right) \Big[ d \V(y) -  d  \|V_i \|(y) \Big] \right|\,,
\end{equation*}
\begin{align*}
b_{i} &:=  \frac{1}{\e_{i}}\left| \int_{y \in B_{\e_{i}}(z_i) }  \Pi_{\widetilde{T_{z_i} M}^\perp} \circ \Pi_{\widetilde{T_y M}} \nabla \rho_1 \left( \frac{y-z_i}{\e_{i}} \right)  d \|V_i \|(y) \right. \\ &\qquad \left. - \int_{(y,S) \in B_{\e_{i}}(z_i) \times \G} \int_{P \in \G} \Pi_{P^\perp} \circ \Pi_S  \nabla \rho_1 \left( \frac{y-z_i}{\e_{i}} \right) \, d\nu_{z_i}^i (P) dV_i(y,S) \right|\,,
\end{align*}
\begin{align*}
c_{i} := \frac{1}{\e_{i}}  \int_{y \in \Omega } \left|  \Pi_{\widetilde{T_{x} M}^\perp} \circ \Pi_{\widetilde{T_y M}}  \nabla \rho_1 \left( \frac{y-x}{\e_{i}} \right) - \Pi_{\widetilde{T_{z_i} M}^\perp} \circ \Pi_{\widetilde{T_y M}}  \nabla \rho_1 \left( \frac{y-z_i}{\e_{i}} \right)\right| d \V(y)\,.
\end{align*}
By definition of orthogonal approximate mean curvature, we have
\begin{align}\notag
\left| H_{\rho,\xi,\e_i}^{V,\perp} (x) - H_{\rho,\xi,\e_i}^{V_i,\perp} (z_i) \right| & = \frac{C_\xi}{C_\rho} \frac{a_{i} + b_{i} + c_{i}}{\e_i^n \|V_i\| \ast \xi_{\e_i} (z_i)} \\\label{eq:diffortappmc}
&\qquad\qquad + \frac{\left| H_{\rho,\xi,\e_i}^{V,\perp} (x) \right|}{\e_i^n \|V_i\| \ast \xi_{\e_i}(z_i)} \Big| \V \ast \xi_{\e_i} (x) - \|V_i\| \ast \xi_{\e_i} (z_i) \Big|\,.
\end{align}
Let $\phi_{i}:\R^{n}\to \R^{n}$ be the map defined as
\[ 
\phi_{i}(y) = \Pi_{\widetilde{T_{z_i} M}^\perp} \circ  \Pi_{\widetilde{T_y M}} \displaystyle \nabla \rho_1 \left( \frac{y-z_i}{\e_{i}} \right) \: .
\]
Then, $\phi_{i}$ is Lipschitz and for $y,w \in \R^n$, if $y, w \in \R^n \setminus B_{\e_{i}}(z_i)$ then $\phi_{i}(y) = \phi_{i}(w) =0$. Thus, assuming that $w \in B_{\e_{i}}(z_i)$, one has
\begin{align*}
| \phi_{i}(y) - \phi_{i}(w) | & \leq \left| \Pi_{\widetilde{T_{z_i} M}^\perp} \circ \left( \Pi_{\widetilde{T_y M}} - \Pi_{\widetilde{T_w M}} \right) \nabla \rho_1 \left( \frac{y-z_i}{\e_{i}} \right) \right|\\
& + \left| \Pi_{\widetilde{T_{z_i} M}^\perp} \circ \left( \Pi_{\widetilde{T_w M}} - \Pi_{\widetilde{T_{z_i} M}} \right) \left[ \nabla \rho_1 \left( \frac{y-z_i}{\e_{i}} \right) - \nabla \rho_1 \left( \frac{w-z_i}{\e_{i}} \right) \right] \right| \\
& \leq \| \widetilde{T_y M} - \widetilde{T_w M} \| \| \nabla \rho_1 \|_{\infty} + \| \widetilde{T_w M} - \widetilde{T_{z_i} M} \| \xlip (\nabla \rho_1) \frac{|y-w|}{\e} \\
& \leq \xlip(\widetilde{T_\cdot M}) \|  \rho^\prime \|_{\infty} |y-w|  + \xlip(\widetilde{T_\cdot M})  \frac{|w-z_i|}{\e_{i}} \xlip (\rho^\prime) |y-w| \\
& \leq \xlip(\widetilde{T_\cdot M})  \| \rho \|_{\xW^{2,\infty}} |y-w| \: .
\end{align*}
where the first inequality follows from the identity $\Pi_{\widetilde{T_{z_i} M}^\perp} \circ \Pi_{\widetilde{T_{z_i} M}} = 0$, and the last one since $|w-z_i| \leq \e_{i}$. Moreover $\phi_{i}$ is uniformly bounded by $ \|  \rho^\prime \|_\infty$ and supported in $B_{\e_{i} + |x-z_i|} (x)$, hence by \eqref{eq_thm_pointwise_cvSmooth_hyp2} we obtain
\begin{align}\notag
a_{i} &= \frac{1}{\e_{i}} \left| \int \phi_{i}(y) \, d \| V_i \|(y) - \int \phi_{i}(y) \, d \V(y) \right| \\\notag 
&\leq \frac{1}{\e_{i}} \left( 1+ \xlip(\widetilde{T_\cdot M}|_{B_{\e_{i} + |x-z_i|}(x)}) \right) \| \rho \|_{\xW^{2,\infty}} \Delta^{1,1}_{B_{\e_{i} + |x-z_i| } (x)} (\|V_i\|,\V) \\\label{eq_proof_smoothCV_1}
&\le \frac{1}{\e_{i}} \left( 1+ \xlip(\widetilde{T_\cdot M}|_{B_{\e_{i} + |x-z_i|}(x)}) \right) \| \rho \|_{\xW^{2,\infty}} d_{1,i} \|V\|(B_{\e + |x-z_i| + \eta_i}(x))\,. 
\end{align}

Then, for $P,S \in \G$ we have
\begin{align*}
\left\| \Pi_{\widetilde{T_{z_i} M}^\perp} \circ \Pi_{\widetilde{T_y M}} - \Pi_{P^\perp} \circ \Pi_S \right\| & \leq \Big\| {\widetilde{T_{z_i} M}^\perp} \Big\| \Big\| {\widetilde{T_y M}} - S \Big\| + \Big\| {\widetilde{T_{z_i} M}^\perp} - {P^\perp} \Big\| \| S \| \\
& \leq \left\| {\widetilde{T_y M}} - S \right\| + \left\| {\widetilde{T_{z_i} M}} - P \right\| \: ,
\end{align*}
hence we obtain
\begin{align}\notag
b_{i} &\leq \frac{1}{\e} \| \rho \|_{\xW^{1,\infty}} \int_{y \in B_{\e_{i}}(z_i)} \int_{(P,S) \in \G \times \G}\left( \left\| {\widetilde{T_y M}} - S \right\| + \left\| {\widetilde{T_{z_i} M}} - P \right\| \right)  \, d\nu_{z_i}^i (P) d \nu_y^i (S) \, d \| V_i \|(y) \\ \label{eq_proof_smoothCV_2}
&\leq \| \rho \|_{\xW^{1,\infty}} \frac{d_{2,i}}{\e} \| V_i \| (B_{\e_i + |x-z_i|}(x))\,, 
\end{align}
also owing to \eqref{eq_thm_pointwise_cvSmooth_hyp3}. 

By similar computations as those leading to \eqref{eq_proof_smoothCV_1}, the map $\psi_{i}:\R^{n}\to\R^{n}$ defined by
\[ 
\psi_{i}(z) = \Pi_{\widetilde{T_z M}^\perp} \circ  \Pi_{\widetilde{T_y M}} \displaystyle \nabla \rho_1 \left( \frac{y-z}{\e_{i}} \right) 
\]
satisfies $\xlip(\psi_{i}) \leq \xlip(\widetilde{T_\cdot M})  \| \rho \|_{\xW^{2,\infty}}$. Therefore,
\begin{equation} \label{eq_proof_smoothCV_3}
c_{i} \leq \frac{1}{\e_{i}} \xlip(\widetilde{T_\cdot M}|_{B_{\e_{i} + |x-z_i|}(x)}) \| \rho \|_{\xW^{2,\infty}} |x-z_i| \|V\|(B_{\e + |x-z_i|}(x)) \: .
\end{equation}

In conclusion, by plugging \eqref{eq_proof_smoothCV_1}, \eqref{eq_proof_smoothCV_2} and \eqref{eq_proof_smoothCV_3} into \eqref{eq:diffortappmc}, and owing to 
Lemma~\ref{lemma_technical_pointwise_convergence} \eqref{eq_lemma_Vzi}-\eqref{eq_lemma_masszi} and Proposition~\ref{prop_consistency_modifiedMC}, one has for $i$ large enough that
\begin{align*}
\left| H_{\rho,\xi,\e_i}^{V,\perp} (x) - H_{\rho,\xi,\e_i}^{V_i,\perp} (z_i) \right| & \leq \frac{C_\xi}{C_\rho} \frac{2 \|V\|(B_{\e_i + |x-z_i| + \eta_i}(x))}{\e_i^n \|V_i\| \ast \xi_{\e_i} (z_i)} \frac{d_{1,i} + d_{2,i} + |x-z_i|}{\epsilon_i} \\
& + \frac{ (|H (x)|+1) }{\e_i^n \|V_i\| \ast \xi_{\e_i}(z_i)} \frac{d_{1,i} \V (B_{\e_i + |x-z_i| + \eta_i}(x)) + |x-z_i| \V(B_{\e_i + |x-z_i|}(x)) }{\epsilon_i} \\
& = O \left(\frac{d_{1,i} + d_{2,i} + |x-z_i|}{\epsilon_i} \right) \: ,
\end{align*}
which concludes the proof.

\end{proof}

\section{Natural kernel pairs}\label{section:pairs}

In previous sections we have considered generic pairs $(\rho,\xi)$ of kernel profiles, as introduced in Section \ref{section:pre} and further specified in Hypothesis \ref{rhoepsxieps}. One might ask whether or not some special choice of kernel pairs could lead to better convergence rates than those proved in Theorems \ref{thm:convergence2} and \ref{thm:convergence3}. Although the pair $(\rho,\rho)$ seems quite natural, as it allows for instance some algebraic simplifications in the formula for the $\e$-mean curvature for a point cloud varifold (see also \cite{brakke}, where $\rho=\xi=$ heat kernel profile), from the point of view of numerical convergence rates there is a better choice. We thus propose a different criterion for selecting the pair $(\rho,\xi)$, that is related to what we define as the \textit{natural kernel pair} property, or shortly (NKP), see Definition \ref{dfn:NKP}. A heuristic justification of the (NKP) property is provided by the analytic computations presented below. 
\begin{dfn}[Natural Kernel Pair]\label{dfn:NKP}
We say that $(\rho, \xi)$ is a natural kernel pair, or equivalently that it satisfies the (NKP) property, if it satisfies Hypothesis~\ref{rhoepsxieps} and 
\begin{equation}\label{NKP}
\xi(s) = -\frac{s\rho'(s)}{n}\qquad \text{for all }s\in (0,1)\,.
\end{equation}
\end{dfn}
Even though it is not clear whether the (NKP) property may produce improved convergence rates in the previously mentioned theorems, we shall see in Section \ref{section:numerics} its experimental validation. In particular, all the tests that we have performed showed a significantly augmented convergence and robustness, even in presence of noise. 

We now sketch the argument leading to Definition \ref{dfn:NKP}. Given $1\le d<n$ and $\rho, \xi$ as in Hypothesis~\ref{rhoepsxieps} we set
\[
C_{\rho,\xi} =  \frac{\int_{0}^{1}\rho(t)\, t^{d-1}\, dt}{\int_{0}^{1}\xi(t)\, t^{d-1}\, dt} =  \frac{C_\rho}{C_\xi}\,.
\]
We fix a $d$--dimensional submanifold $M\subset \R^{n}$ of class $C^{3}$ and define the associated varifold $V = v(M,1)$. Then we perform a Taylor expansion of the difference $H_{\rho,\xi,\e}^V (x) - H(x)$ at a point $x \in M$ (here $H(x)$ denotes the classical mean curvature of $M$ at $x$). By focusing on the expression of the constant term of this expansion, which must be $0$ because of Theorem~\ref{thm:convergence1}, we notice that such an expression (see \eqref{eq_constantTerm}) is proportional to 
\[
\int_{0}^{1} \left(s\rho'(s) + d\, C_{\rho,\xi}\,\xi(s) \right) s^{d-1}\, ds \,.
\]
On one hand, this integral is $0$ for any kernel pair $(\rho,\xi)$, as shown through an integration by parts coupled with the definition of the constant $C_{\rho,\xi}$. On the other hand we might want to strengthen the nullity of the integral by additionally requiring the nullity of the integrand. This precisely amounts to require \eqref{NKP} and thus leads to the definition of the (NKP) property.
\medskip

We now give more details on the argument sketched above.

Let $M$ be as above and assume that $0\in M$ and that $T_{0}M \simeq \{x=(y,0)\in \R^{n}:\ y = (y_{1},\dots,y_{d})\}$. Of course this is always the case up to an isometry. Then $M$ is locally the graph of a smooth function $u:A\to \R^{n-d}$, where $A$ is a neighbourhood of $0\in\R^{d}$. Clearly our assumptions imply that $u(0) = 0$ and $\nabla u(0) = 0$, hence 
\begin{equation}\label{uipsilon}
u(y) = \frac 12\langle \nabla^{2}u(0) y,y\rangle + o(|y|^{2})\,,\qquad \nabla u(y) = \nabla^{2}u(0)y + o(|y|)\,.
\end{equation}
For $r>0$ small enough we consider the sphere $S_{r} = \partial B_{r}$ and set $M_{r} = M\cap S_{r}$. For any $x\in M_{r}$ we let $y=y(x)\in \R^{d}$ be the vector of the first $d$ coordinates of $x$. We then have $x = (y,u(y))$ and $|y|^{2} + |u(y)|^{2} = |x|^{2}$. Note that by \eqref{uipsilon} we also have $|y| = |x| + o(|x|^{2})$. Let $\{\nu_{1},\dots,\nu_{n-d}\}$ be the standard basis of the orthogonal space $(T_{0}M)^{\perp}\subset \R^{n}$, so that we have $u(y) = \sum_{j=1}^{n-d} u_{j}(y)\, \nu_{j}$. Whenever $x=(y,u(y))\in M$ is close enough to the origin, there exists an orthogonal basis $\{\hv_{1},\ldots,\hv_{n-d}\}$ of $T_{x}^{\perp}M$, such that 
\[
\hv_{j} = -\nabla u_{j}(y) + \nu_{j} + O(|y|^{2})\,.
\]
Consequently, by noting that $|x| = O(|y|)$, the projection of $x$ onto the tangent space $T_{x}M$ satisfies the following relation:
\begin{align*}
x^{M} &= x - \sum_{j=1}^{n-d} \langle x, -\nabla u_{j}(y) + \nu_{j}\rangle\cdot \frac{-\nabla u_{j}(y) + \nu_{j}}{1 + |\nabla u_{j}(y)|^{2}} + O(|y|^{3})\\ 
&= y + \sum_{j=1}^{n-d} u_{j}(y)\nu_{j}+\Big(\langle \nabla u_{j}(y),y\rangle - u_{j}(y)\Big) \frac{-\nabla u_{j}(y) + \nu_{j}}{1 + |\nabla u_{j}(y)|^{2}} + O(|y|^{3})\,.
\end{align*}
It is then easy to check that 
\begin{equation}\label{xprocirca}
x^{M} = y +\sum_{j=1}^{n-d} \langle D^{2}u_{j}(0)y,y\rangle \nu_{j} + O(|y|^{3})\,.
\end{equation} 

We introduce some extra notation. Given $r>0$ sufficiently small, we denote by $W_{r}$ the projection of $M_{r}$ onto $T_{0}M$, that is,
\begin{equation}\label{Werre}
W_{r} = \{y:\ |y|^{2} + |u(y)|^{2} = r^{2}\}\,.
\end{equation}
We will now prove that $W_{r}$ is a small deformation of a $d$-sphere of radius $r$, with explicit estimates as $r\to 0$. We thus set $\Sigma_{r} = \{z\in \R^{d}:\ |z|=r\}$ and define $f_{r}:\Sigma_{r}\to W_{r}$ as $f_{r}(z) = (1+\phi_{r}(z)) z$, such that the implicit relation 
\begin{equation}\label{implicita1}
(1+\phi_{r}(z))^{2} -1 + \sum_{j=1}^{n-d} \frac{u_{j}^{2}((1+\phi_{r}(z))z)}{r^{2}} = 0
\end{equation}
is satisfied. Thanks to the implicit function theorem, \eqref{implicita1} defines
$\phi_{r}(z)$ and thus $f_{r}(z)$ when $r>0$ is small enough. Moreover, by noticing that $|u(y)| = O(r^{2})$ thanks to \eqref{Werre}, one infers from \eqref{implicita1} that  
\begin{equation}\label{phicircazero}
\phi_{r}(z) = O(r^{2})\,.
\end{equation}
Now we estimate the difference $H^{V}_{\rho,\xi,\e}(0) - H(0)$. We first recall that $H(0) = \sum_{j=1}^{n-d} \Delta u_{j}(0)\, \nu_{j}$. Letting $V = \mathbf{v}(M,1)$ we have
\begin{align*}
H^{V}_{\rho,\xi,\e}(0) - H(0) &= \frac{1}{C_{\rho,\xi}} \frac{\delta V \ast \rho_\epsilon (0)}{\V \ast \xi_\epsilon (0)} + H(0)\\ 
&= \frac{\delta V \ast \rho_\epsilon (0) + C_{\rho,\xi}\V \ast \xi_\epsilon (0)H(0)}{C_{\rho,\xi}\V \ast \xi_\epsilon (0)} \\
&= \frac{\int_{M} \left(\e^{-n-1} \rho'(|x|/\e) \frac{x^{M}}{|x|} + C_{\rho,\xi}\, \e^{-n}\xi(|x|/\e) H(0)\right)\, d\cH^{d}(x)}{C_{\rho,\xi}\V \ast \xi_\epsilon (0)}\\
&= \frac{A}{\e^{n}C_{\rho,\xi}\V \ast \xi_\epsilon (0)}\,.
\end{align*}
Let us apply the coarea formula and rewrite the term $A$ above as follows:
\begin{align*}
A &= \int_{0}^{\e} \int_{W_{r}} \left(\frac{\rho'(r/\e)}{\e}x^{M} + C_{\rho,\xi} H(0)r\xi(r/\e)\right)|x^{M}|^{-1}\, d\cH^{d-1}(x)\, dr\,.
\end{align*}
We then apply the area formula using the map 
\[
x= g_{r}(z) = (f_{r}(z),u(f_{r}(z)))\,,\qquad  f_{r}(z) = z + O(r^{2})
\]
and obtain
\begin{align}\label{formulaA}
A &= \int_{0}^{\e} \int_{\Sigma_{r}} \left(\frac{\rho'(r/\e)}{\e}x^{M} + C_{\rho,\xi} H(0)r\,\xi(r/\e)\right)|x^{M}|^{-1}\, Jg_{r}(z)\, d\cH^{d-1}(z)\, dr\,,
\end{align}
where $Jg_{r}$ denotes the tangential Jacobian of $g_{r}$ (here we are assuming $d\ge 2$, otherwise if $d=1$ then $Jg_{r}=1$ and the subsequent computations are even simpler). It is now convenient to identify $z$ with the point of $\R^{n}$ whose first $d$ coordinates are, respectively, $z_{1},\ldots, z_{d}$. Now we write $\phi_{r}$ instead of $\phi_{r}(z)$ for more brevity, so that 
\[
g_{r}(z) = (1+\phi_{r})z + \sum_{j} u_{j}((1+\phi_{r})z)\, \nu_{j}\,.
\]
Now we fix $z\in \Sigma_{r}\subset \R^{d}$ and choose a unit tangent vector $v\in T_{z}\Sigma_{r}$. By differentiating $g_{r}$ at $z$ along the direction $v$, and taking into account \eqref{phicircazero}, we get
\begin{align}\nonumber
\partial_{v}g_{r}(z) &= \partial_{v} \phi_{r}(z) z + (1+\phi_{r})v + \sum_{j} [\partial_{v}\phi_{r}\, \partial_{z} u_{j}((1+\phi_{r})z) + (1+\phi_{r})\partial_{v}u_{j}((1+\phi_{r})z)]\, \nu_{j} \\\label{vOr}
&= (1+O(r^{2}))v + \sum_{j} O(r)\nu_{j}\,.
\end{align}
Then one can fix an orthonormal basis $\{v_{1},\ldots,v_{d-1}\}$ for $T_{z}\Sigma_{r}$, then by \eqref{vOr} one finds
\begin{equation}\label{stimaJgr}
Jg_{r}(z) = |\partial_{v_{1}}g_{r}(z)\wedge \ldots \wedge \partial_{v_{d-1}}g_{r}(z)| = \sqrt{1 + O(r^{2})} = 1+O(r^{2})\,.
\end{equation}
Moreover by \eqref{xprocirca} combined with \eqref{phicircazero} one obtains 
\begin{align*}
x^{M} &= (1+\phi_{r}) z + (1+\phi_{r})^{2}\sum_{j=1}^{n-d} \langle D^{2}u_{j}(0)z,z\rangle\, \nu_{j} + O(r^{3})\\ 
&= z + \sum_{j=1}^{n-d} \langle D^{2}u_{j}(0)z,z\rangle\nu_{j} + O(r^{3})\,,
\end{align*}
hence by exploiting the fact that $\nu_{j}\perp z$ we have $|x^{M}| = \sqrt{r^{2} + O(r^{4})} = r(1+O(r^{2}))$. Plugging this last estimate and \eqref{stimaJgr} into \eqref{formulaA} we find
\begin{align}\label{formulaA2}
A &= \int_{0}^{\e} \int_{\Sigma_{r}} \left(\frac{\rho'(r/\e)}{r\e}\Big(z+ \sum_{j=1}^{n-d} \langle D^{2}u_{j}(0)z,z\rangle\nu_{j} + O(r^{3})\Big) + C_{\rho,\xi} H(0)\,\xi(r/\e)\right)\Big(1+O(r^{2})\Big)\, d\cH^{d-1}(z)\, dr\,.
\end{align}
Now let us focus on the term in the expansion of the right-hand side of \eqref{formulaA2}, whose expression contains neither $O(r^{2})$ nor $O(r^{3})$:
\begin{align*}
A' &= \int_{0}^{\e} \int_{\Sigma_{r}} \left(\frac{\rho'(r/\e)}{r\e}\Big(z+ \sum_{j=1}^{n-d} \langle D^{2}u_{j}(0)z,z\rangle\nu_{j}\Big) + C_{\rho,\xi} H(0)\,\xi(r/\e)\right)\, d\cH^{d-1}(z)\, dr\\
&= \int_{0}^{\e} \int_{\Sigma_{r}} \left(\frac{\rho'(r/\e)}{r\e}\sum_{j=1}^{n-d} \langle D^{2}u_{j}(0)z,z\rangle\nu_{j} + C_{\rho,\xi} H(0)\,\xi(r/\e)\right)\, d\cH^{d-1}(z)\, dr\\ 
&= \int_{0}^{\e} \left(\frac{\rho'(r/\e)}{r\e}\sum_{j=1}^{n-d}\int_{\Sigma_{r}}  \langle D^{2}u_{j}(0)z,z\rangle\, d\cH^{d-1}(z)\, \nu_{j} + C_{\rho,\xi} H(0)\,\xi(r/\e) d\omega_{d}r^{d-1}\right)\, dr\,.
\end{align*}
Now, owing to the symmetry of $\Sigma_{r}$, we can assume up to a rotation that the canonical basis of $T_{0}M \simeq \R^{d}$ coincides with the spectral basis for $D^{2}u_{j}(0)$, so that $\langle D^{2}u_{j}(0)z,z\rangle = \sum_{h=1}^{d} \lambda^{j}_{h}z_{h}^{2}$, where $\lambda^{j}_{h}$ denotes the $h$-th eigenvalue of $D^{2}u_{j}(0)$. We thus find
\begin{align}
A' &= \int_{0}^{\e} \left(\frac{\rho'(r/\e)}{r\e}\sum_{j=1}^{n-d}\sum_{h=1}^{d}\lambda^{j}_{h}\int_{\Sigma_{r}}  z_{h}^{2}\, d\cH^{d-1}(z)\, \nu_{j} + C_{\rho,\xi} H(0)\,\xi(r/\e) d\omega_{d}r^{d-1}\right)\, dr \nonumber \\
&= \int_{0}^{\e} \left(\frac{\rho'(r/\e)}{r\e}\sum_{j=1}^{n-d}\sum_{h=1}^{d}\lambda^{j}_{h}\omega_{d}r^{d+1} \nu_{j} + C_{\rho,\xi} H(0)\,\xi(r/\e) d\omega_{d}r^{d-1}\right)\, dr \nonumber \\ 
&= \int_{0}^{\e} \left(\frac{\rho'(r/\e)}{r\e}\sum_{j=1}^{n-d}\Delta u_{j}(0)\omega_{d}r^{d+1}\, \nu_{j} + C_{\rho,\xi} H(0)\,\xi(r/\e) d\omega_{d}r^{d-1}\right)\, dr \nonumber \\ 
&= \omega_{d} H(0)\, \int_{0}^{\e} \left(\frac{r\rho'(r/\e)}{\e} + d\, C_{\rho,\xi}\,\xi(r/\e) \right) r^{d-1}\, dr \nonumber \\ 
&= \omega_{d}\e^{d} H(0)\, \int_{0}^{1} \left(s\rho'(s) + d\, C_{\rho,\xi}\,\xi(s) \right) s^{d-1}\, ds\,. \label{eq_constantTerm}
\end{align}
We can now observe the following two facts about \eqref{eq_constantTerm}. First, $A' = 0$ for any pair of kernels $\rho,\xi$ (this can be seen through an integration by parts coupled with the definition of the constant $C_{\rho,\xi}$). Second, if we require the additional nullity of the integrand, we come to the differential relation
\begin{equation}\label{NKP0}
s\rho'(s) = -d C_{\rho,\xi} \xi(s)\,.
\end{equation}
At this point we must check whether \eqref{NKP0} is compatible with the definition of $C_{\rho,\xi}$. To this aim we only have to show that if $(\rho,\xi)$ satisfy \eqref{NKP0}, then $\frac{\int_{0}^{1}\rho(s)\, s^{d-1}\, ds}{\int_{0}^{1}\xi(s)\, s^{d-1}\, ds} = C_{\rho,\xi}$. Indeed by integrating by parts, by using the properties of the kernels, and owing to \eqref{NKP0}, we obtain
\begin{align*}
\frac{\int_{0}^{1}\rho(s)\, s^{d-1}\, ds}{\int_{0}^{1}\xi(s)\, s^{d-1}\, ds} 
=\frac{-\int_{0}^{1}\rho'(s)\, s^{d}\, ds}{d\int_{0}^{1}\xi(s)\, s^{d-1}\, ds}
=\frac{dC_{\rho,\xi}\int_{0}^{1}\xi(s)\, s^{d-1}\, ds}{d\int_{0}^{1}\xi(s)\, s^{d-1}\, ds} 
=C_{\rho,\xi}\,,
\end{align*}
as wanted. On the other hand, if we recall that the profiles $\rho$ and $\xi$ must also satisfy
\[
\int_{0}^{1} \rho(s) s^{n-1}\, ds = \int_{0}^{1} \xi(s) s^{n-1}\, ds = 1\,,
\]
we obtain the extra condition that, together with \eqref{NKP0}, uniquely determines the value of $C_{\rho,\xi}$. Indeed, assuming \eqref{NKP0} and integrating by parts we find
\[
d C_{\rho,\xi} = d C_{\rho,\xi}\int_{0}^{1}\xi(s) s^{n-1}\, ds  = -\int_{0}^{1}\rho'(s) s^{n}\, ds = n\int_{0}^{1}\rho(s) s^{n-1}\, ds = n\,, 
\]
whence $C_{\rho,\xi} = \frac nd$ and thus \eqref{NKP0} is equivalent to \eqref{NKP}.

\section{Discrete approximations of a varifold}
\label{section:DAV}

In this section, we prove that the family of discrete volumetric varifolds and the family of point cloud varifolds approximate well the space of rectifiable varifolds in the sense of weak--$\ast$ convergence, or $\Delta^{1,1}$ metric. Moreover, we give a way of quantifying this approximation in terms of the mesh size and the mean oscillation of tangent planes. We start with a technical lemma involving a general $d$--varifold.

\begin{lemma}\label{lemma:approx-local-estimate}
Let $\Omega \subset \R^n$ be an open set and $V$ be a $d$--varifold in $\Omega$. Let $(\cK_i)_{i\in \N}$ be a sequence of meshes of $\Omega$, and set
\[
\delta_i = \sup_{K \in \cK_i} \xdiam (K)\quad \forall\,i\in \N\: .
\]
Then, there exists a sequence of discrete (point cloud or volumetric) varifolds $(V_i)_i$ such that for any open set $U \subset \Omega$,
\begin{equation} \label{eq_approx_vol_pt_varifolds} 
\Delta^{1,1}_U (V,V_i) 
\leq \delta_i \V ( U^{\delta_i} ) + \sum_{K \in \cK_i} \min_{P \in \G} \int_{(U^{\delta_i} \cap K) \times \G} \| P - S \| \, dV (x,S) \: .
\end{equation}
\end{lemma}

\begin{proof}
We define $V_{i}$ as either the volumetric varifold
\[
V_i = \sum_{K \in \cK_i} \frac{m_K^i}{|K|} \cL^n \otimes \delta_{P_K^i}\, ,
\]
or the point cloud varifold 
\[
V_i = \sum_{K \in \cK_i} m_K^i \delta_{x_K^i} \otimes \delta_{P_K^i} \, ,
\]
with
\[
m_K^i = \V(K), \quad x_K^i \in K \quad \text{and} \quad P_K^i \in \argmin_{P \in \G} \int_{K \times \G} \left\| P - S \right\| \, dV(x,S)\, .
\]
Let us now explain the proof for the case of volumetric varifolds, as it is completely analogous in the case of point cloud varifolds. 
For any open set $U \subset \Omega$ and $\phi \in \xLip_1 (\R^n \times \G)$ with $\supp \phi \subset U \times \G$, we set
\[
\Delta_{i}(\phi) = \int_{\Omega \times \G} \phi \, d V_i - \int_{\Omega \times \G} \phi \, dV
\]
and obtain
\begin{align*}
\left| \Delta_{i}(\phi) \right| & = \left| \sum_{K \in \cK_i} \int_{K \cap U}  \phi(x,P_K^i) \frac{\V(K)}{|K|} \, d \cL^n(x) - \sum_{K \in \cK_i} \int_{(K \cap U) \times \G} \hspace*{-1cm} \phi(y,T) \, dV(y,T) \right| \nonumber \\
& \leq \sum_{\substack{K \in \cK_i \\ K \cap U \neq \emptyset}} \fint_{x \in K} \int_{(y,T) \in K \times \G} \underbrace{\left| \phi(x,P_K^i) - \phi(y,T) \right| }_{\leq \left( |x-y| + \left\| P_K^i - T \right\| \right) } \, dV(y,T) \, d \cL^n(x) \nonumber \\
& \leq \delta_i \sum_{\substack{K \in \cK_i \\ K \cap U \neq \emptyset}} \V(K) + \, \sum_{\substack{K \in \cK_i \\ K \cap U \neq \emptyset}} \int_{ K \times \G} \left\| P_K^i - T \right\| \, dV(y,T) \\
& \leq \delta_i\, \V(U^{\delta_i}) + \, \sum_{K \in \cK_i} \min_{P \in \G} \int_{(U^{\delta_i} \cap K) \times \G} \left\| P - T \right\| \, dV(y,T) \: , \nonumber
\end{align*}
which concludes the proof up to taking the supremum of $\Delta_{i}(\phi)$ over $\phi$.
\end{proof}

In Theorem \ref{diffuse_discrete_varifolds_theorem} below we show that rectifiable varifolds can be approximated by discrete varifolds. Moreover we get explicit convergence rates under the following regularity assumption. 
\begin{dfn}\label{def:piecewiseC1beta}
Let $S$ be a $d$-rectifiable set, $\theta$ be a positive Borel function on $S$, and $\beta\in (0,1]$. We say that the rectifiable $d$--varifold $V = v(S,\theta)$ is \emph{piecewise $C^{1,\beta}$} if there exist $R>0,\ C\ge 1$ and a closed set $\Sigma\subset S$ such that the following properties hold:
\begin{itemize}
\item (Ahlfors-regularity of $S$) for all $x\in S$ and $0<r<R$ 
\begin{equation}\label{ahlfors-d}
C^{-1}r^{d} \le \cH^{d}(S\cap B(x,r))\le Cr^{d}\, ;
\end{equation}

\item (Ahlfors-regularity of $\Sigma$) for all $z\in \Sigma$ and $0<r<R$ \begin{equation}\label{ahlfors-d-1}
C^{-1}r^{d-1} \le \cH^{d-1}(\Sigma\cap B(z,r))\le Cr^{d-1}\,;
\end{equation}

\item ($C^{1,\beta}$ regularity of $S\setminus \Sigma$) the function 
\[
\tau(r) = \sup\{\|T_{y}S - T_{z}S\|:\ y,z\in S\cap B(x,r),\ x\in S \text{ with }\dist(x,\Sigma)>Cr\}
\]
satisfies
\begin{equation}\label{eq:C1beta}
\tau(r) \le C\, r^{\beta}\qquad \forall\, 0<r<R\,;
\end{equation}

\item for all $0<r<\e<R$ and all $z\in \Sigma$  
\begin{equation}\label{eq:d-1versusd}
C^{-1}r\,\cH^{d-1}(\Sigma\cap B(z,\e))\le \cH^{d}(S\cap [\Sigma]_{r}\cap B(z,\e)) \le C\, r\,\cH^{d-1}(\Sigma\cap B(z,\e))\,.
\end{equation}

\item for $\cH^{d}$-almost all $x\in S$ we have 
\begin{equation}\label{eq:thetauniform}
C^{-1}\le \theta(x)\le C\,.
\end{equation}
\end{itemize}
\end{dfn}

\begin{remk}\label{remk:Almgren}\rm 
We note that varifolds of class piecewise $C^{1,\beta}$ form a natural collection of varifolds, which for instance the so-called $(\mathbf{M},\e,\delta)$-minimal sets of dimension $1$ and $2$ in $\R^{3}$ in the sense of Almgren belong to. In other words the rectifiable varifold $V=v(S,1)$ is of class piecewise $C^{1,\beta}$ as a consequence of Taylor's regularity theory \cite{Taylor76} (see also \cite{David2} for an extension of Taylor's results in higher dimensions). Of course, the family of rectifiable varifolds in $\R^{3}$ that are piecewise $C^{1,\beta}$ is much larger than $(\mathbf{M},\e,\delta)$-minimal sets.
\end{remk}

In the following theorem we prove an approximation result for rectifiable $d$--varifolds, that becomes quantitative as soon as the varifolds are assumed to be piecewise $C^{1,\beta}$ in the sense of Definition \ref{def:piecewiseC1beta}. In order to avoid a heavier, localized form of Definition \ref{def:piecewiseC1beta} we set $\Omega = \R^{n}$.
\begin{theo} \label{diffuse_discrete_varifolds_theorem}
Let $(\cK_i)_{i\in \N}$ be a sequence of meshes of $\R^{n}$ and set $\delta_i = \sup_{K \in \cK_i} \xdiam (K)$ for all $i\in \N$. Let $V = v(M,\theta)$ be a rectifiable $d$--varifold in $\R^{n}$ with $\|V\|(\R^{n})<+\infty$. Then there exists a sequence of discrete (volumetric or point cloud) varifolds ${(V_i)}_i$ with the following properties:
\begin{itemize}
\item[(i)] $\Delta^{1,1} (V_i,V) \to 0$ as $i\to\infty$;

\item[(ii)] If $V$ is piecewise $C^{1,\beta}$ in the sense of Definition \ref{def:piecewiseC1beta} then there exist constants $C,R > 0$ such that for all balls $B$ with radius $r_{B}\in (0,R)$ centered on the support of $\V$ one has
\begin{equation} \label{eq_localized_flat_convergence}
\Delta^{1,1}_B (V_i,V) \leq C  \left(\delta_i^{\beta} + \frac{\delta_{i}}{r_{B}+\delta_{i}}\right)\, \V(B^{C\delta_i})
\end{equation}
and  
\begin{equation}\label{eq:globalflatconv}
\Delta^{1,1}(V_i,V) \leq C  \left(\delta_i^{\beta} + \frac{\delta_{i}}{R}\right)\, \V(\R^{n})
\end{equation}
\end{itemize}
\end{theo}

\begin{proof} The proof is split into some steps. 

\noindent
\textit{Step $1$.} We show that for all $i$ there exists $A^i : \R^{n} \rightarrow \cM_n(\R)$ constant in each cell $K \in \cK_i$, such that
\begin{equation}\label{eq:goalstep1}
\int_{\R^{n} \times \G} \left\| A^i (y) - T \right\| \, dV(y,T) = \int_{y \in \R^{n}}  \left\| A^i (y) - T_y M \right\| \, d \V(y) \xrightarrow[i \to + \infty]{} 0  \: .
\end{equation}

\noindent Indeed, let us fix $\epsilon >0$. Since $x \mapsto T_x M \in \xL^1 (\R^{n} , \cM^n (\R) , \V)$, there exists $A : \R^{n} \rightarrow \cM^n (\R) \in \xLip (\R^{n})$ such that
\[
\int_{y \in \R^{n}}  \left\| A (y) - T_y M \right\| \, d \V(y) < \epsilon \: .
\]
For all $i$ and $K \in \cK_i$, define for $x \in K$,
\[
A^i (x) = A_K^i = \frac{1}{\V(K)} \int_K A(y) \, d \V(y) \: .
\]
Then 
\begin{align*}
\int_{y \in \R^{n}}  \left\| A^i (y) - T_y M \right\| \, d \V(y) & \leq \int_{y \in \R^{n}}  \left\| A^i (y) - A(y) \right\| \, d \V(y) + \int_{y \in \R^{n}}  \left\| A (y) - T_y M \right\| \, d \V(y) \\
& \leq \epsilon + \sum_{K \in \cK_i} \int_{y \in K}  \left\|  \frac{1}{\V(K)} \int_K A(u) \, d \V(u) - A(y) \right\| \, d \V(y) \\
& \leq \epsilon + \sum_{K \in \cK_i}  \frac{1}{\V(K)} \int_{y \in K}    \int_{u \in K} \left\| A(u) - A(y) \right\| \, d \V(u)  \, d \V(y)  \\
& \leq \epsilon + \delta_i \xlip (A) \V(\R^{n}) \leq 2 \epsilon \text{ for } i \text{ large enough,}
\end{align*}
which proves \eqref{eq:goalstep1}.

\noindent
\textit{Step $2$.} Here we make the result of Step $1$ more precise, i.e., for all $i$, we prove that there exists $T^i : \R^{n} \rightarrow \G$ constant in each cell $K \in \cK_i$ such that
\begin{equation}\label{eq:goalstep2}
\int_{\R^{n} \times \G} \left\| T^i (y) - T \right\| \, dV(y,T) = \int_{y \in \R^{n}}  \left\| T^i (y) - T_y M \right\| \, d \V(y) \xrightarrow[i \to + \infty]{} 0  \: .
\end{equation}

\noindent Indeed, let $\epsilon >0$ and, thanks to Step $1$, take $i$ large enough and $A^i : \R^{n} \rightarrow \cM_n(\R)$ as in \eqref{eq:goalstep1}, such that
\[
\sum_{K \in \cK_i} \int_K  \left\| A^i (y) - T_y M \right\| \, d \V(y) < \epsilon \,.
\]
As a consequence we find
\[
\int_K  \left\| A^i (y) - T_y M \right\| \, d \V(y) = \epsilon_K^i\quad \text{with}\quad \sum_{K \in \cK_i} \epsilon_K^i < \epsilon\,.
\] 
In particular, for all $K \in \cK_i$, there exists $y_K \in K$ such that
\[
\left\| A^i (y_K) - T_{y_K} M \right\| \leq \frac{\epsilon_K^i}{\V(K)} \: .
\]
Define $T^i : \R^{n} \rightarrow \G$ by $T^i (y) = T_{y_K} M$ for $K \in \cK_i$ and $y \in K$, hence $T^{i}$ is constant in each cell $K$ and 
\begin{align}
\int_{\R^{n} \times \G} \left\| T^i (y) - T \right\| & \, dV(y,T)  = \sum_{K \in \cK_i} \int_K  \left\| T_{y_K} M  - T_y M \right\| \, d \V(y) \label{eq_approx_control_tangent_plane} \\
& \leq \sum_{K \in \cK_i} \int_K  \| T_{y_K} M  - \underbrace{A^i(y)}_{= A^i(y_K)} \| \, d \V(y) + \int_{\R^{n} \times \G} \left\| A^i (y) - T \right\| \, dV(y,T) \nonumber \\
& \leq \sum_{K \in \cK_i} \int_K \frac{\epsilon_K^i}{\V(K)} d \V(y) + \epsilon \nonumber \leq 2 \epsilon \: ,
\end{align}
which implies \eqref{eq:goalstep2}.

\noindent
\textit{Step $3$: proof of (i).} We preliminarily show that 
\begin{equation}\label{eq:osctozero}
\sum_{K \in \cK_i} \min_{P \in \G} \int_{K \times \G} \left\| P - T \right\| \, dV(y,T) \xrightarrow[i \to \infty]{} 0.
\end{equation}
Indeed, thanks to Step $2$, let $T^i : \R^{n} \rightarrow \G$ be such that \eqref{eq:goalstep2} holds. We have
\begin{align*}
\sum_{K \in \cK_i} \min_{P \in \G} \int_{K \times \G} \left\| P - T \right\| \, dV(y,T) & \leq \sum_{K \in \cK_i} \int_{K \times \G}  \left\| T^i_K  - T \right\| \, dV(y,T) \\
& = \int_{\R^{n} \times \G} \left\| T^i (y) - T \right\| \, dV(y,T)\\
& \xrightarrow[i \to +\infty]{} 0 \: ,
\end{align*}
which proves \eqref{eq:osctozero}. Then (i) follows by combining \eqref{eq:osctozero} with Lemma \ref{lemma:approx-local-estimate}.

\noindent
\textit{Step $4$.} Assume that $V$ is piecewise $C^{1,\beta}$ and let $R,C>0$ be as in Definition \ref{def:piecewiseC1beta}. We shall now prove that for any ball $B\subset \R^{n}$ centered on the support of $\V$ with radius $r_{B} < R/2$ and for any infinitesimal sequence $\eta_{i}\ge C\delta_{i}$, assuming also $i$ large enough so that $\delta_{i}\le (R-2r_{B})/(C+1)$, there exists a decomposition $\cK_{i} = \cK_{i}^{reg}\sqcup \cK_{i}^{sing}$ such that 
\begin{equation}\label{holdercondition}
\| T_x S - T_y S \| \leq C |x-y|^\beta , \qquad \forall\, K\in \cK_i^{reg},\ \forall\, x, y \in K\cap S 
\end{equation}
and
\begin{equation}\label{smallsingularset}
\V \left( \bigcup\cK_i^{sing} \cap B \right) \leq C' \frac{\delta_i}{r_{B}+\eta_{i}}\V(B^{\eta_{i}})\: .
\end{equation}
Define $\cK_{i}^{sing}$ as the set of $K\in \cK_{i}$ for which $\Sigma^{(C\delta_{i})}\cap K$ is non-empty, and set $\cK_{i}^{reg} = \cK_{i}\setminus \cK_{i}^{sing}$. It is immediate to check that \eqref{holdercondition} holds, thanks to \eqref{eq:C1beta}. Let now $B$ be a fixed ball of radius $0<r_{B}<R/2$ centered at some point $x\in S$. Take $K\in \cK_{i}^{sing}$ and assume without loss of generality that $K\cap B$ is not empty, hence there exists $p\in K\cap B$ and $z\in \Sigma$ such that $|p-z| < (C+1)\delta_{i}$. Consequently, $B\subset B(z,2r_{B}+(C+1)\delta_{i})$. Then $K\cap B\subset \Sigma^{(C+1)\delta_{i}}\cap B(z,2r_{B}+(C+1)\delta_{i})$ and thus, assuming in addition that $\eta_{i}<R-r_{B}$ for $i$ large enough, we obtain
\begin{align}\notag
\V\left(\bigcup_{K\in \cK_{i}^{sing}} K\cap B\right) &\le \V \Big(\Sigma^{(C+1)\delta_{i}}\cap B(z,2r_{B}+(C+1)\delta_{i})\Big) \\\notag
&\le C\cH^{d}\Big(S\cap \Sigma^{(C+1)\delta_{i}}\cap B(z,2r_{B}+(C+1)\delta_{i})\Big)\\\label{eq:passo1}
&\le C^{2}(C+1)\delta_{i}\, \cH^{d-1}\Big(\Sigma \cap B(z,2r_{B}+(C+1)\delta_{i})\Big)\,,
\end{align}
thanks to \eqref{eq:d-1versusd} and \eqref{eq:thetauniform}. On the other hand, since $C\delta_{i} \le \eta_{i} <R-r_{B}$ one has by \eqref{ahlfors-d}, \eqref{ahlfors-d-1} and \eqref{eq:thetauniform} that 
\begin{align*}\notag
\V(B^{\eta_{i}}) &\ge C^{-1}\cH^{d}(S\cap B^{\eta_{i}}) \ge C^{-2}(r_{B}+\eta_{i})^{d} \ge C^{-2}(r_{B}+\eta_{i})(r_{B}+C\delta_{i})^{d-1}\\\notag 
&\ge \frac{1}{2^{d-1}C^{2}(C+1)}\frac{(r_{B}+\eta_{i})}{\delta_{i}}\, (C+1)\delta_{i}\, (2r_{B}+(C+1)\delta_{i})^{d-1}\\
&\ge \frac{1}{2^{d-1}C^{3}(C+1)}\frac{(r_{B}+\eta_{i})}{\delta_{i}}\, C(C+1)\delta_{i} \cH^{d-1}(\Sigma\cap B(z,2r_{B}+(C+1)\delta_{i}))\\
&\ge \frac{1}{2^{d-1}C^{5}(C+1)}\frac{(r_{B}+\eta_{i})}{\delta_{i}}\, \V\left(\bigcup_{K\in \cK^{sing}} K\cap B\right)\,,
\end{align*}
which by \eqref{eq:passo1} gives \eqref{smallsingularset} with $C'=2^{d-1}C^{5}(C+1)$.

\noindent
\textit{Step $5$.} Define $T_K^i = T_{y_K} M$ for each cell $K \in \cK_i$ and for some $y_K \in K$. Set 
\[
A_{i} = \sum_{K \in \cK_i} \min_{P \in \G} \int_{(B\cap K) \times \G} \left\| P - T \right\| \, dV(y,T)\,.
\]
Then for every ball $B$ of radius $r>0$, and choosing $\eta_{i} = C\delta_{i}$, we have
\begin{align}\notag 
A_{i} & =  \sum_{K \in \cK_i} \int_{K \cap B}  \left\| T_{y_K} M  - T_y M \right\| \, d \V(y) \\\notag
& = \sum_{K \in \cK_i^{reg}} \int_{K \cap B}  \left\| T_{y_K} M  - T_y M \right\| \, d \V(y) + \sum_{K \in \cK_i^{sing}} \int_{K \cap B}  \left\| T_{y_K} M  - T_y M \right\| \, d \V(y) \\\notag
& \leq \sum_{K \in \cK_i^{reg}} \int_{K \cap B} C |y_K -y|^\beta \, d \V(y) + 2 \V \left(\bigcup\cK_{i}^{sing} \cap B\right) \\\notag
& \leq C \delta_i^\beta \V (B) + 2 \V \left(\bigcup\cK_{i}^{sing} \cap B\right) \le C \left(\delta_{i}^{\beta} + \frac{\delta_{i}}{r_{B}+\eta_{i}}\right)\V(B^{C\delta_{i}})\\
\label{eq:estimate-sing-part}
&\le C \left(\delta_{i}^{\beta} + \frac{\delta_{i}}{r_{B}+\delta_{i}}\right)\V(B^{C\delta_{i}})
\end{align}
(the constant $C$ appearing in the various inequalities of \eqref{eq:estimate-sing-part} may change from line to line). Then, the local estimate \eqref{eq_localized_flat_convergence} is a consequence of Lemma \ref{lemma:approx-local-estimate} combined with \eqref{eq:estimate-sing-part}. 

\noindent
\textit{Step $6$.} For the proof of the global estimate \eqref{eq:globalflatconv} we set $r=R/2$ and apply Besicovitch Covering Theorem to the family of balls $\{B_{r}(x)\}_{x\in M}$, so that we globally obtain a subcovering $\{B_{\alpha}\}_{\alpha\in I}$ with overlapping bounded by a dimensional constant $\zeta_{n}$. We notice that $I$ is necessarily a finite set of indices, by the Ahlfors regularity of $M$. We now set $U = \R^{n}\setminus M$ and associate to the family $\{B_{\alpha}\}_{\alpha\in I} \cup \{U\}$ a partition of unity $\{\psi_{\alpha}\}_{\alpha\in I}\cup \{\psi_{U}\}$ of class $C^{\infty}$, so that by finiteness of $I$ there exists a constant $L\ge 1$ with the property that $\xlip(\psi_{U})\le L$ and $\xlip(\psi_{\alpha})\le L$ for all $\alpha\in I$. Moreover, the fact that the support of $\psi_{U}$ is disjoint from the closure of $M$ implies that there exists $i_{0}$ depending only on $M$, such that the support of $\|V_{i}\|$ is disjoint from that of $\psi_{U}$ for every $i\ge i_{0}$. Then we fix a generic test function $\phi \in C^{0}_{c}(\R^{n}\times \G)$ and define $\phi_{\alpha}(x,S) = \phi(x,S)\psi_{\alpha}(x)$ and $\phi_{U}(x,S) =  \phi(x,S)\psi_{U}(x)$, so that $\phi(x,S) = \phi_{U}(x,S) + \sum_{\alpha\in I}\phi_{\alpha}(x,S)$. By the fact that $\xlip(\phi_{\alpha}) \le \xlip(\phi) + \xlip(\psi_{\alpha})$ and $\xlip(\phi_{U}) \le \xlip(\phi) + \xlip(\psi_{U})$, by the Ahlfors regularity of $M$, by \eqref{eq_localized_flat_convergence}, and for $i\ge i_{0}$, we deduce that
\begin{align*}
|V_{i}(\phi) - V(\phi)| &\le \sum_{\alpha\in I}|V_{i}(\phi_{\alpha}) - V(\phi_{\alpha})| \le (1+L)\sum_{\alpha\in I}\Delta^{1,1}_{B_{\alpha}}(V_{i},V)\\
&\le C(1+L) \sum_{\alpha\in I} \left(\delta_i^{\beta} + \frac{\delta_{i}}{r}\right)\, \V(B_{\alpha}^{C\delta_i}) \le C(1+L) \left(\delta_i^{\beta} + \frac{\delta_{i}}{r}\right) \sum_{\alpha\in I} \V(B_{\alpha})\\
&\le C(1+L)\zeta_{n} \left(\delta_i^{\beta} + \frac{\delta_{i}}{r}\right)  \V(\R^{n})\le C\left(\delta_i^{\beta} + \frac{\delta_{i}}{R}\right)  \V(\R^{n})
\end{align*}
where, as before, the constant $C$ appearing in the above inequalities can change from one step to the other. This concludes the proof of \eqref{eq:globalflatconv} and thus of the theorem.

\end{proof}

\section{A varifold interpretation of the Cotangent Formula}
\label{section:cotangent}

One of the classical tools of discrete differential geometry is the so-called \textit{Cotangent Formula} \eqref{eq_functionalCotangentFormula} which provides a notion of mean curvature for a triangulated polyhedral surface. The Cotangent Formula has been introduced in \cite{pinkall1993} as the gradient of a discrete Dirichlet energy defined on triangulations. Loosely speaking it consists in the definition of a vector mean curvature functional $\widehat{H}$ by its action on \textit{nodal functions} (see \cite{Wardetzky2008}). In Proposition~\ref{prop_CotangentFormula} we show that the formula can be interpreted as the action of the first variation of the associated polyhedral varifold $V$ on any Lipschitz extension of a given nodal function $\phi$.

Let $\cT = (\cF,\cE,\cV)$ be a triangulation in $\R^3$, where $\cV \subset \R^3$ is the set of vertices, $\cE \subset \cV \times \cV$ is the set of edges and $\cF$ is the set of triangle faces (we refer to triangulations in the sense of polyhedral surfaces homeomorphic to a $2d$--manifold, as defined for instance in \cite{Wardetzky2008}). We denote by $M_\cT = \bigcup_{F \in \cF} F$ the triangulated surface. The nodal function $\phi_v$, $v \in \cV$, associated with $\cT$ is defined on $M_\cT$ by $\phi_v(v) =1$, $\phi_v(w) = 0$ for $w \in \cV$, $w \leq v$ and $\phi_v$ affine on each face $F \in \cF$.

\begin{figure}[!h]
\centering
\subfigure[]{\includegraphics[width=0.3\textwidth]{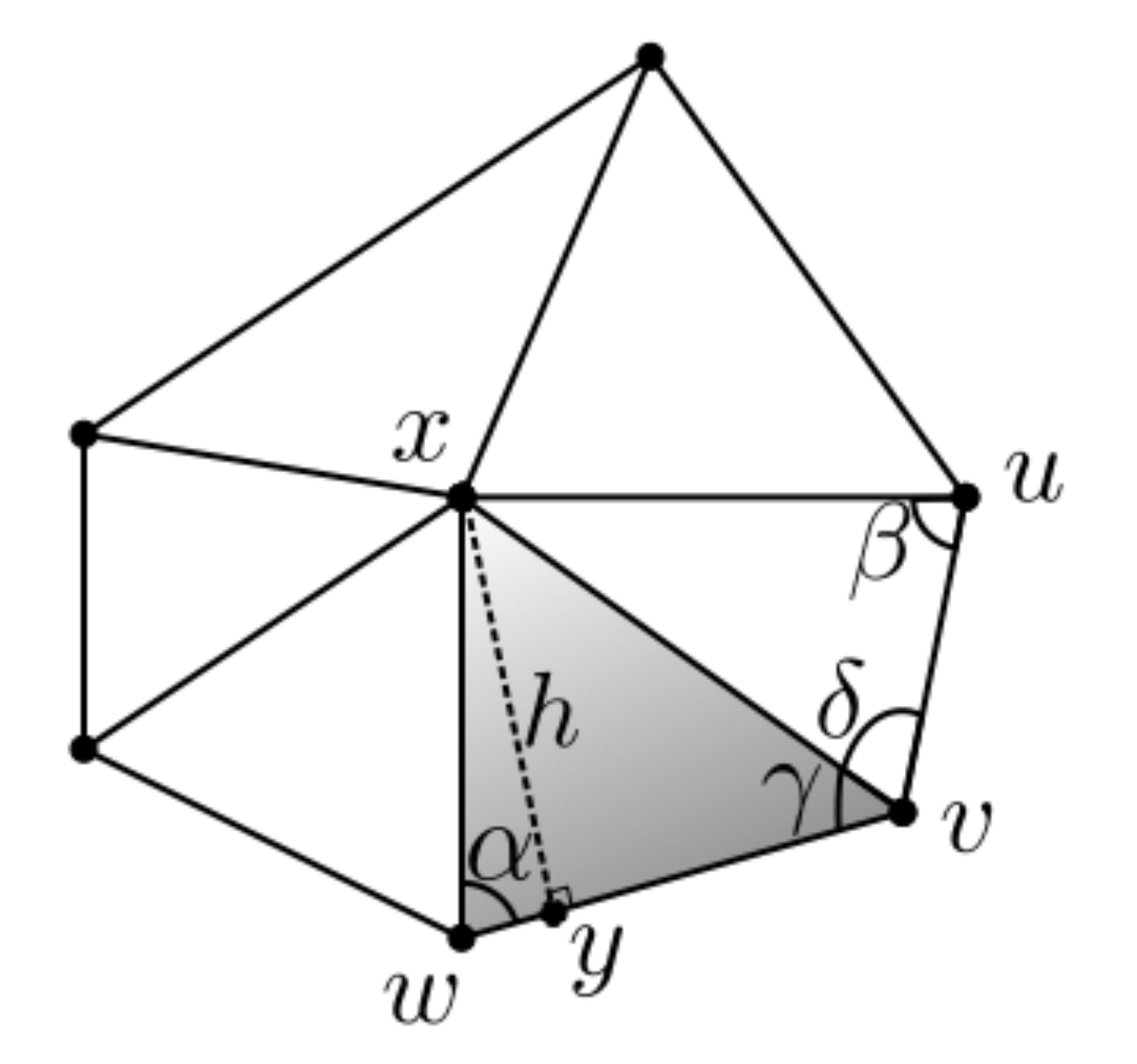}}
\qquad
\subfigure[]{\includegraphics[width=0.3\textwidth]{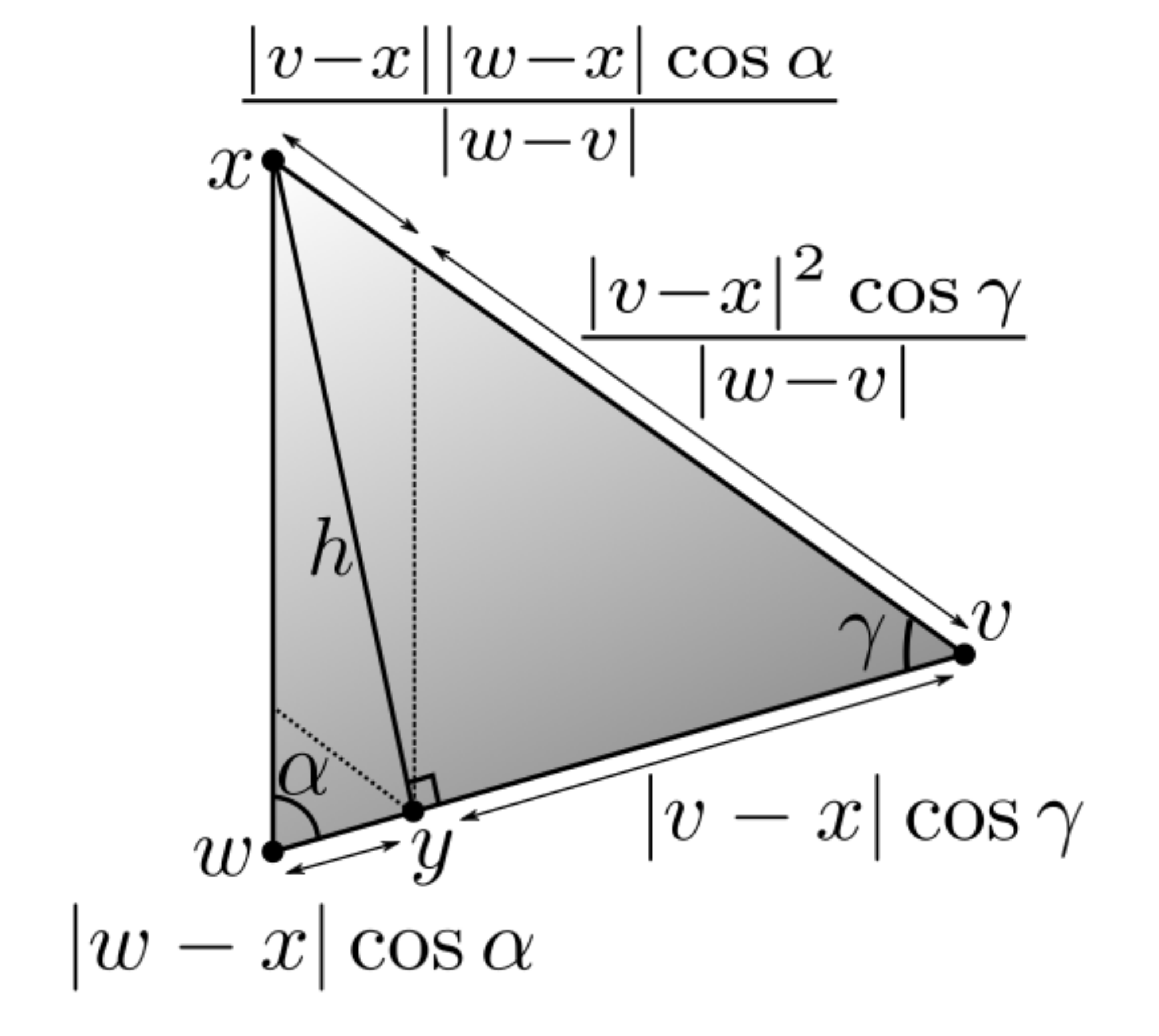}}
\caption{}
\label{figCotangent}
\end{figure}

\noindent Let $x \in \cV$ be a vertex, we denote by $\cV (x)$ the set of vertices conected to $x$ by an edge and $\cF(x)$ the set of faces containing $x$. For each $v \in \cV(x)$, $\alpha = \alpha_{xv}$ and $\beta = \beta_{xv}$ denote the angles opposite to the edge $(xv)$. See Figure~\ref{figCotangent}. With these notations, $\widehat{H}$ is defined by the Cotangent Formula
\begin{equation} \label{eq_functionalCotangentFormula}
< \widehat{H} , \phi_x > = \frac{1}{2} \sum_{v \in \cV(x)} \left( \cot \alpha_{xv} + \cot \beta_{xv} \right) (v-x).
\end{equation}

We recall that we can associate with $\cT$ the $2$--varifold
\[
V_{\cT} = \sum_{F \in \cF} \cH^2_{| F} \otimes \delta_{P_F} \: ,
\]
where $P_F$ is the plane containing the face $F$ (see Definition~\ref{dfn_polyhedralVarifold}). 
We also recall that given a Lipschitz function $g$ defined on $\R^{n}$, whose set of points of non-differentiability has zero $\V$ measure, we can compute $\delta V(g) := \big(\delta V(g\, e_{1}),\dots,\delta V(g\, e_{n})\big)$ as stated in Remark \ref{remk:morefirstvar}.

\begin{prop} \label{prop_CotangentFormula}
Let $x \in \cV$ be a vertex and let $\widehat{\phi_x} : \R^3 \rightarrow \R_+$ be a Lipschitz extension of $\phi_x$. Then,
\begin{equation}
\delta V_\cT (\widehat{\phi_x}) = - < \widehat{H} , \phi_x > = -\frac{1}{2} \sum_{v \in \cV(x)} \left( \cot \alpha_{xv} + \cot \beta_{xv} \right) (v-x).
\end{equation}
\end{prop}

\begin{proof}
For each face $F \in \cF$, $\widehat{\phi_x}_{| F}$ is affine, so that $\nabla^{P_F} \widehat{\phi_x}$ is constant in $F$. Therefore,
\begin{align}
\delta V_\cT (\widehat{\phi_x}) & = \int \nabla^S \widehat{\phi_x} (y) \, dV(y,S) 
 = \sum_{F \in \cF} \int_{F} \nabla^{P_F} \widehat{\phi_x} (y) \, d \cH^2 (y) \nonumber \\
& = \sum_{F \in \cF(x)} \int_F \frac{-1}{h_F} \frac{y_F - x}{|y_F -x |} d \cH^2 (y) 
 = - \sum_F \cH^2 (F) \frac{-1}{h_F} \frac{y_F - x}{|y_F -x |} . \label{eq_cotan_1}
\end{align}
Let us consider a face $F \in \cF(x)$ whose vertices ($\neq x$) are denoted $w,v$ as in Figure~\ref{figCotangent}. Then
\begin{equation}
\cH^2 (F) \frac{-1}{h_F} \frac{y_F - x}{|y_F -x |}  = - \frac{1}{2} h_F |w-v| \frac{1}{h_F} \frac{y_F - x}{|y_F -x |} 
= - \frac{|w-v|}{2h_F}  (y_F - x). \label{eq_cotan_2}
\end{equation}
As
\[
y_F - x = \frac{|w-x|\cos \alpha}{|w-v|}  (v-x) + \frac{|v-x|\cos \gamma}{|w-v|} (w-x),
\]
we infer from \eqref{eq_cotan_1} and \eqref{eq_cotan_2} that
\begin{align*}
\cH^2 (F) \frac{-1}{h_F} \frac{y_F - x}{|y_F -x |} & =\frac{1}{2h_F}\left( |w-x|\cos \alpha   (v-x) +  |v-x|\cos \gamma \right) (w-x) \\
& = \frac{1}{2}\left( \frac{|w-x|\cos \alpha}{|w-x| \sin \alpha}  (v-x) +  \frac{|v-x|\cos \gamma}{|v-x| \sin \gamma} \right) (w-x) \\
& = \frac{1}{2} \left( \cot \alpha (v-x) + \cot \gamma (w-x) \right).
\end{align*}
and 
\begin{equation} \label{eq_cotan_3}
\delta V_\cT (\widehat{\phi_x}) = -\frac{1}{2} \sum_{v \in \cV(x)} \left( \cot \alpha_{xv} + \cot \beta_{xv} \right) (v-x).
\end{equation}
\end{proof}

\begin{remk}\rm
It is not difficult to check that
\[
\| V_\cT \| (\widehat{\phi_x}) = \int \phi \, d \| V_\cT \| = \frac{1}{3} \sum_{F \in \cF(x)} \| V_\cT \|(F)
\]
which allows us to define the discrete mean curvature at each vertex $x$ of the triangulation as
\[
H_{V_\cT}(x) = - \frac{\delta V_\cT (\widehat{\phi_x})}{\| V_\cT \| (\widehat{\phi_x})} = \frac{3}{2} \frac{\sum_{v \in \cV(x)} \left( \cot \alpha_{xv} + \cot \beta_{xv} \right)(v-x)}{\sum_{F \in \cF(x)} \text{Area} (F)}.
\]
\end{remk}

\section{Numerical simulations for 2D and 3D point clouds}\label{section:numerics}

In this section we provide numerical computations of the approximate mean curvature of various $2D$ and $3D$ point clouds. In particular, we illustrate numerically its dependence on the regularization kernel, the regularization parameter $\epsilon$, and the sampling resolution. Our purpose is not a thorough comparison with the many numerical approaches for computing the mean curvature of point clouds, triangulated meshes, or digital objects, this will be done in a subsequent paper for obvious length reasons.

Given a point cloud varifold $V_N = \sum_{j=1}^N m_j \delta_{x_j} \otimes \delta_{P_j}$, its orthogonal approximate mean curvature is given by 
\begin{align} 
H_{\rho,\xi,\epsilon}^{V_N,\perp}(x_{j_0}) & = \int_{P \in \G} \Pi_{P^\perp} H_{\rho,\xi,\epsilon}^{V_N}(x_{j_0}) \, d \nu_{x_{j_0}} (P) \nonumber \\ 
& = -\frac{\displaystyle \sum_{j=1}^N  \one_{ \{ |x_j - x_{j_0}|<\epsilon \}} m_j \rho^\prime \left(\frac{|x_j-x|}{\epsilon} \right) \, \Pi_{P_{j_0}^\perp} \left( \frac{\Pi_{P_j} (x_j-x_{j_0})}{|x_j-x_{j_0}|} \right)}{ \displaystyle \sum_{j=1}^N \one_{ \{ |x_j - x_{j_0}|<\epsilon \} } m_j\epsilon  \xi \left(\frac{|x_j-x_{j_0}|}{\epsilon} \right) }  \: .\label{eq_formulaModifiedMC_PointCloud0}
\end{align}
We focus on the orthogonal approximate mean curvature, for it is at a given resolution more robust with respect to inhomogeneous local distribution of points than the approximate mean curvature, as it will be illustrated in Section~\ref{sec:varyingdensity}, and as it can even be seen directly on simple examples. Take indeed a sampling $\{x_j\}_1^N$ of the planar line segment $[-1,1]\times\{0\}$ with more points having a negative first coordinate, and let $P_j=P=\{y=0\}$. Assume that there exists $j_{0}$ such that $x_{j_{0}} = (0,0)$. Then the sum of all vectors $\frac{\Pi_{P_j} (x_j-x_{j_0})}{|x_j-x_{j_0}|}$ is nonzero, whereas its projection onto $P^\perp$ is zero, which is consistent with the (mean) curvature of the continuous segment at the origin. 

The formula above involves densities $m_j$,  the computation of which for a given point cloud being a question we have not focused on up to now, despite it is an important issue.  Nevertheless, if we assume that $m_j = m(1+o(1))$ whenever $x_{j}$ belongs to the ball $B_{\e}$ and for some constant $m$ possibly depending on $B_{\e}$, then we can cancel $m_{j}$ from formula \eqref{eq_formulaModifiedMC_PointCloud0} up to a small error. This justifies the following formula approximating the value of  $H_{\rho,\xi,\epsilon}^{V_N,\perp}(x_{j_0})$:
\begin{align} 
H_{\rho,\xi,\epsilon}^{V_N,\perp}(x_{j_0}) &\ \simeq\  
\frac{-\displaystyle \sum_{j=1}^N  \one_{ \{ |x_j - x_{j_0}|<\epsilon \}}  \rho^\prime \left(\frac{|x_j-x|}{\epsilon} \right) \, \Pi_{P_{j_0}^\perp} \left( \frac{\Pi_{P_j} (x_j-x_{j_0})}{|x_j-x_{j_0}|} \right)}{ \displaystyle \sum_{j=1}^N \one_{ \{ |x_j - x_{j_0}|<\epsilon \} } \epsilon  \xi \left(\frac{|x_j-x_{j_0}|}{\epsilon} \right) }  \: .\label{eq_formulaModifiedMC_PointCloud}
\end{align}

The advantages of Formula~\eqref{eq_formulaModifiedMC_PointCloud} are numerous: it is very easy to compute, it does not require a prior approximation of local length or area, it does \textit{not depend on any orientation} of the point cloud (because the formula is grounded on varifolds which have no orientation) and as we shall see right now, it behaves well from a numerical perspective.

In the next subsection, we study how this formula behaves on $2D$ point cloud varifolds built from parametric curves, for different choices of radial kernels and various sampling resolutions. The last subsection is devoted to $3D$ point clouds.

\subsection{Orthogonal approximate mean curvature of sampled parametric 2D curves}\label{section:numerical}

Let us start with some comments concerning the implementation. While the algorithmic complexity is linear in the number of points, a point cloud in 3D typically contains a huge number of points, so we choose to handle the implementation in C++ for both 2D and 3D cases. We also use \emph{nanoflann}~\cite{nanoflann} to build a KD-tree structure which can easily manage neighbor search for defining local neighborhoods, and perform local regression by means of \emph{eigen}~\cite{eigen}. Point cloud visualization can be easily done with \emph{Matlab}, \emph{CloudCompare}~\cite{cloudcompare} or \emph{Meshlab}~\cite{meshlab}.

\subsubsection{Test shapes, sample point cloud varifolds, and kernel profiles}
We test the numerical behavior of formula \eqref{eq_formulaModifiedMC_PointCloud} for different choices of 2D parametric shapes, kernel profiles $\rho,\xi$, number $N$ of points in the cloud, and values of the parameter $\epsilon$ used to define the kernels $\rho_\epsilon$ and $\xi_{\e}$. We denote as $N_{\text{neigh}}$ the average number of points in a ball of radius $\epsilon$ centered at a point of the cloud. The 2D parametric test shapes are (see Figure~\ref{2Dpts}):
\begin{enumerate}[$(a)$]
\item A circle of radius $0.5$ parametrized as $(x(t),y(t))=0.5(\cos(t),\sin(t))$, $t\in[0,2\pi]$;
\item An ellipse parametrized by
$x(t)=a\,\cos(t)$, $y(t)=b\,\sin(t)$, $t \in (0,2\pi)$ with $a=1$ and $b=0.5$;
\item A "flower" parametrized by $r(\theta) = 0.5(1+0.5\sin(6\theta+\frac{\pi}{2}))$;
\item An "eight" parametrized by $x(t)=0.5 \, \sin(t)\left( \cos t + 1 \right)$, $y(t)=0.5 \, \sin(t)\left( \cos t - 1 \right)$, $t \in (0,2\pi)$.
\end{enumerate}

\begin{center}
\setcounter{subfigure}{0}
\begin{figure}[!htp]
\subfigure[]{\includegraphics[width=0.24\textwidth]{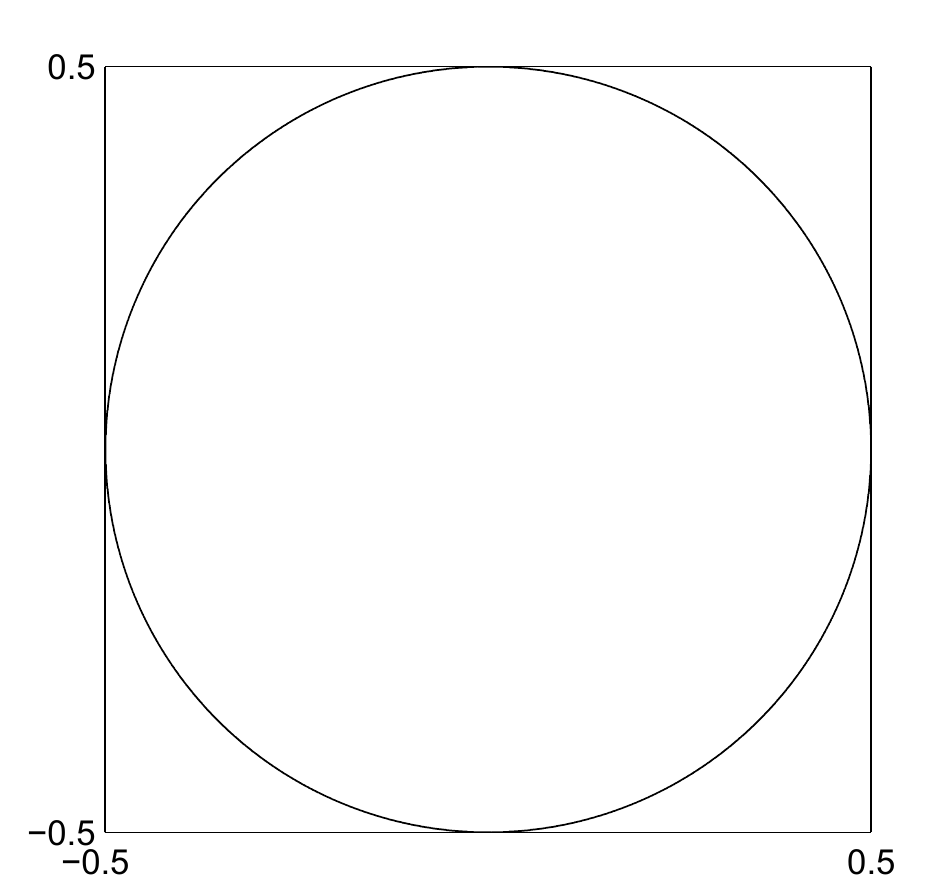}}
\subfigure[]{\includegraphics[width=0.24\textwidth]{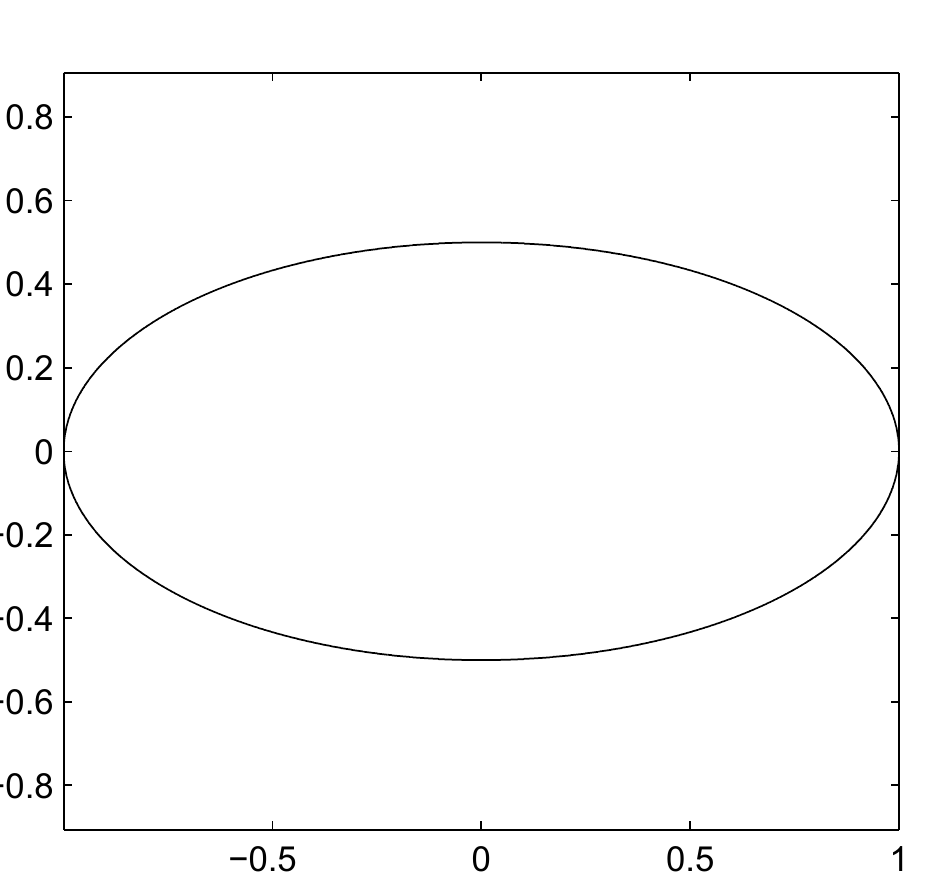}}
\subfigure[]{\includegraphics[width=0.24\textwidth]{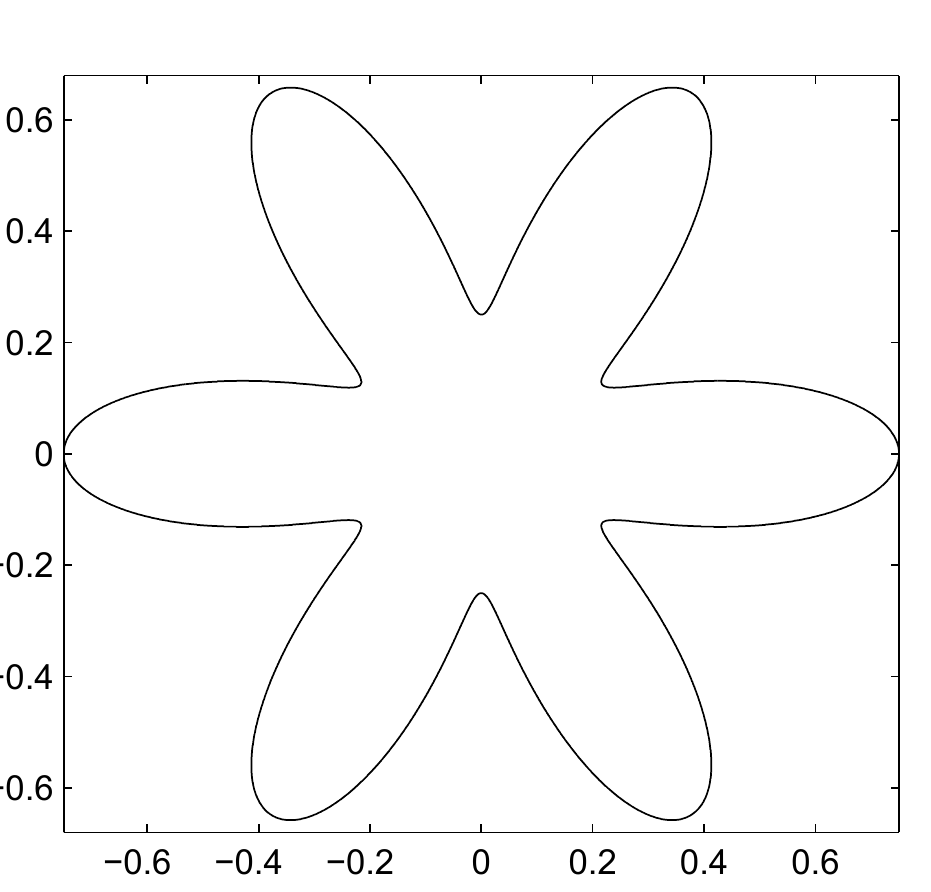}}
\subfigure[]{\includegraphics[width=0.22\textwidth]{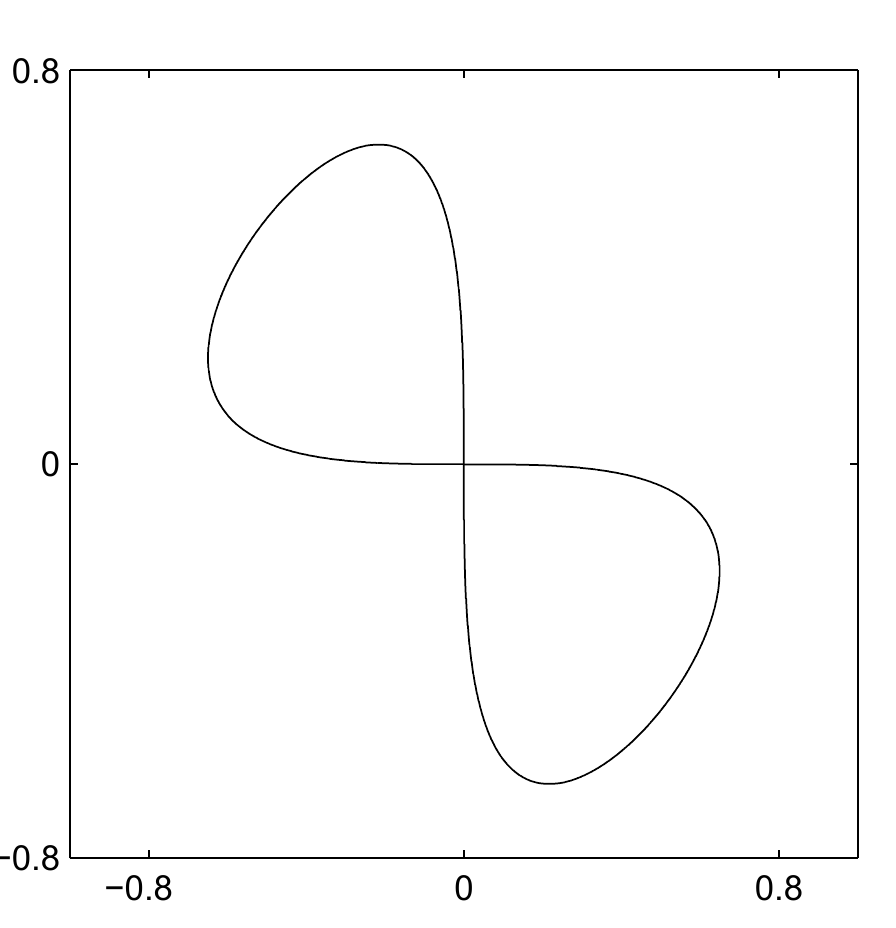}}
\caption{2D parametric test shapes\label{2Dpts}}
\end{figure}
\end{center}

We test formula \eqref{eq_formulaModifiedMC_PointCloud} with some profiles $\rho, \xi$ defined on $[0,1]$:
\begin{itemize}
\item the ``tent'' kernel pair $(\rho_{tent},\rho_{tent})$, with $\rho_{tent} (r) = (1-r)$;
\item the ``natural tent'' pair $(\rho_{tent},\xi_{tent})$, with $\xi_{tent} (r) = - \frac{1}{n} r \rho_{tent}^\prime (r) = r$;
\item the ``exp'' kernel pair $(\rho_{\exp},\rho_{exp})$, with $\rho_{exp} (r) = \exp \left( - \frac{1}{1-r^2} \right)$;
\item the ``natural exp'' pair $(\rho_{exp},\xi_{exp})$, with $\xi_{exp} (r) = - \frac{1}{n} r \rho_{exp}^\prime (r) $.
\end{itemize}
Notice that $\rho_{exp}, \xi_{exp}$ satisfy Hypothesis~\ref{rhoepsxieps}; on the contrary, $\rho_{tent}$ is only in $\xW^{1,\infty}$ and $\xi_{tent}$ is not even continuous.

To define point clouds from samples of these parametric test shapes, we use two approaches:
\begin{itemize}
\item either we compute the exact tangent line $T(t)\in G_{1,2}$ at the $N$ points $\{ 0, h, 2h, \ldots , (N-1)h \}$ for $h = \frac{2 \pi}{N}$, and we set
\begin{equation} \label{eq_parametric_varifold}
V_N = \sum_{j=1}^N m_j \delta_{(x(jh),y(jh))} \otimes \delta_{T(jh)} \: ,
\end{equation}
\item or we compute by linear regression a tangent line $T^{app} \in G_{1,2}$ at each sample point and we set
\begin{equation} \label{eq_parametric_varifold_approx}
V_N = \sum_{j=1}^N m_j \delta_{(x(jh),y(jh))} \otimes \delta_{T^{app}(jh)} \: .
\end{equation}
\end{itemize}

\noindent For all shapes under study, the exact vector curvature $H(t)$ can be computed explicitly and evaluated at $jh, \, j=0 \ldots N-1$. To quantify the accuracy of approximation \eqref{eq_formulaModifiedMC_PointCloud}, we use the following 
relative average error
\begin{equation} \label{eq_rel_average_error}
E^{rel} = \frac{1}{N} \sum_{j=1}^N \frac{| H_{\rho,\xi,\epsilon}^{V_{N}} (x_j) - H(jh)|}{\| H \|_\infty} \: ,
\end{equation}
where $x_j = (x(jh),y(jh))$.

\subsubsection{Numerical illustration of orthogonal approximate mean curvature}
We first test formula~\eqref{eq_formulaModifiedMC_PointCloud} on the ellipse and on the flower with exact normals. We represent in Figure \ref{fig_flower} the curvature vectors computed for $N=10^5$ points and $\epsilon = 0.001$ with the natural kernel pair $(\rho_{exp},\xi_{\exp})$. Arrows indicate the vectors and colors indicate their norms. Remark that the sample points are obtained from a uniform sampling in parameter space (polar angle), therefore sample points are not regularly spaced on the ellipse or the flower. Still, these spatial variations are negligible and \eqref{eq_formulaModifiedMC_PointCloud} provides a  good approximation of the continuous mean curvature, as we already know from Theorem~\ref{thm:convergence3}, and as it will be illustrated numerically in the next section.

\setcounter{subfigure}{0}
\begin{figure}[!htbp]
\begin{center}
\subfigure[]{\includegraphics[width=0.35\textwidth]{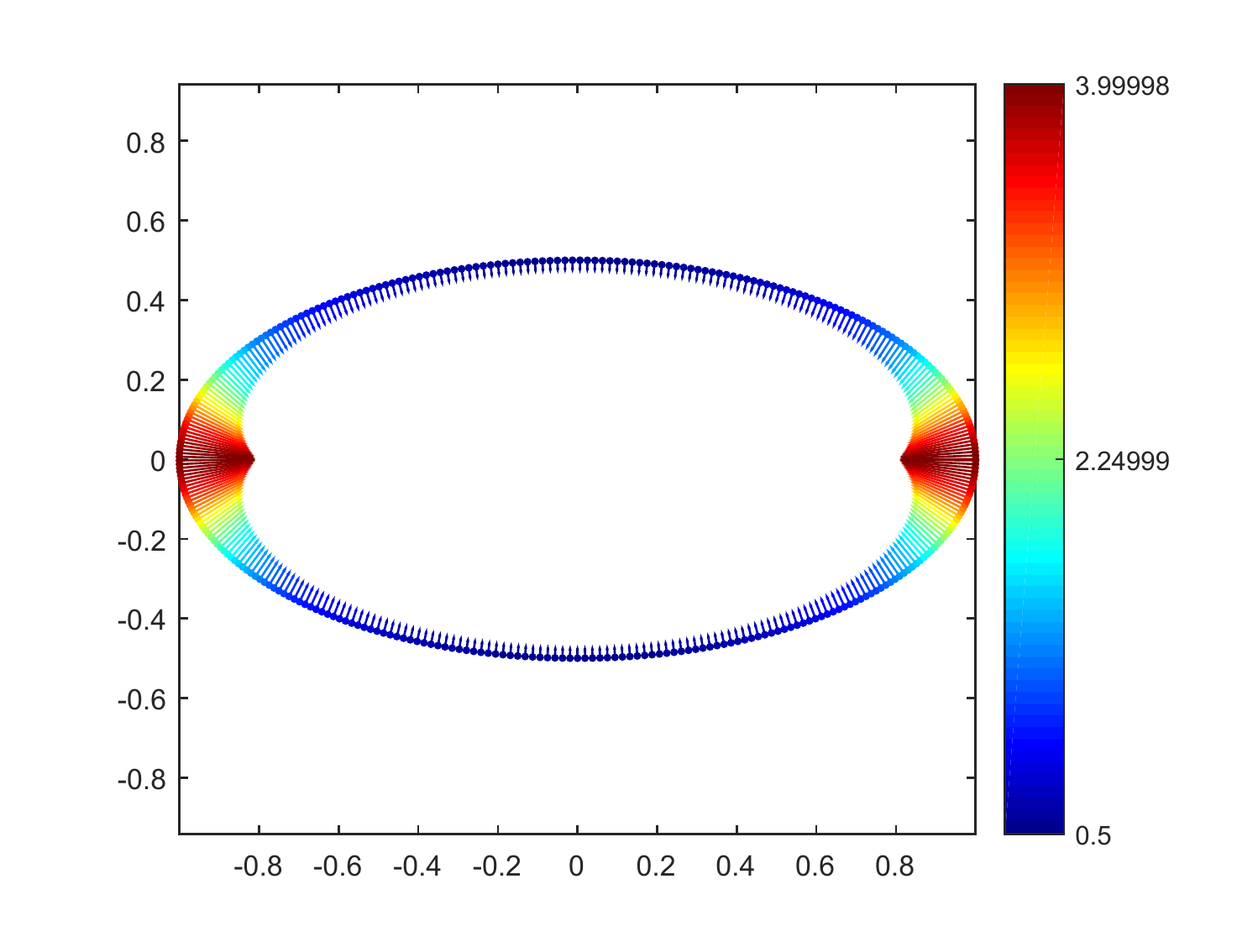}}
\subfigure[]{\includegraphics[width=0.35\textwidth]{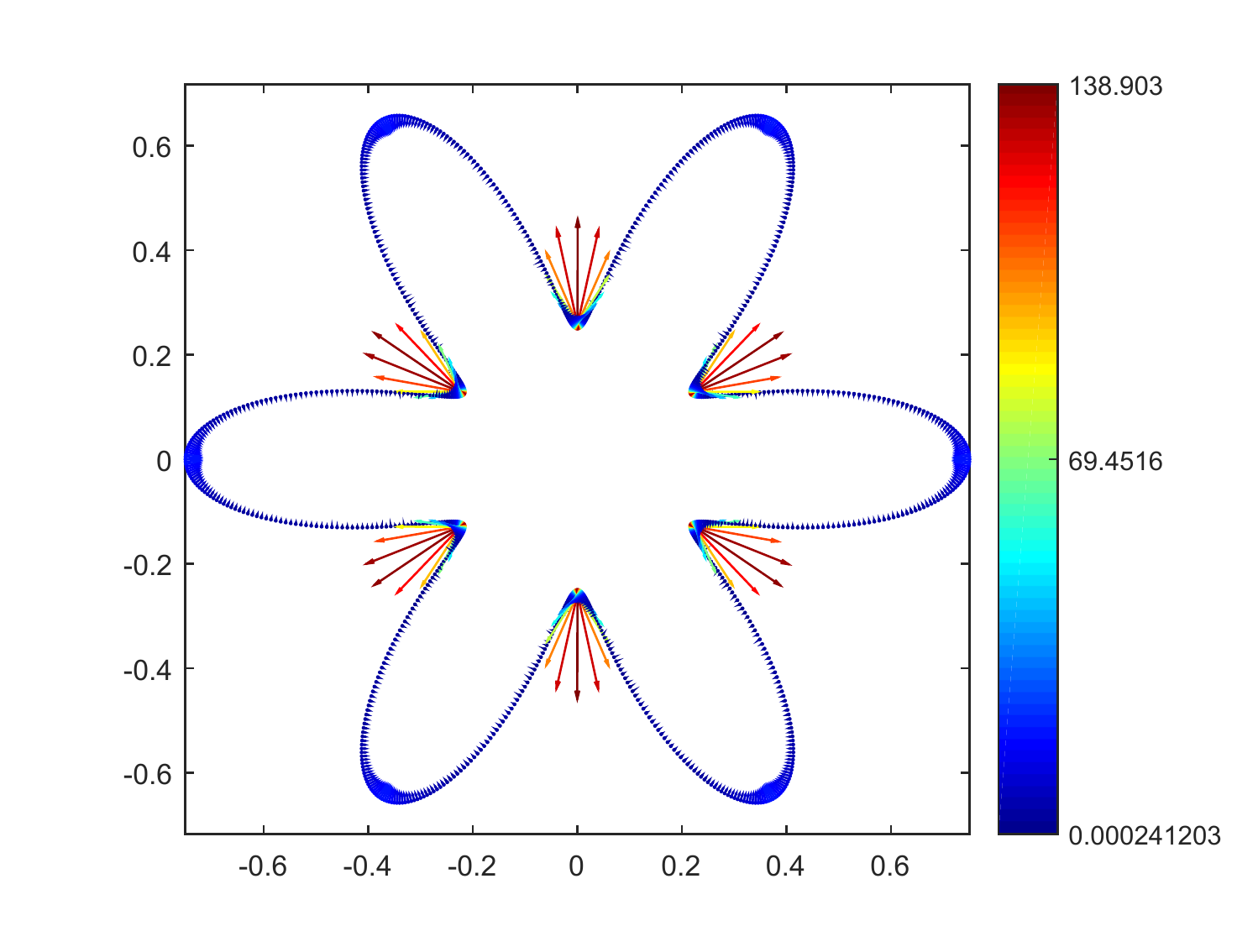}}
\subfigure[]{\includegraphics[width=0.25\textwidth]{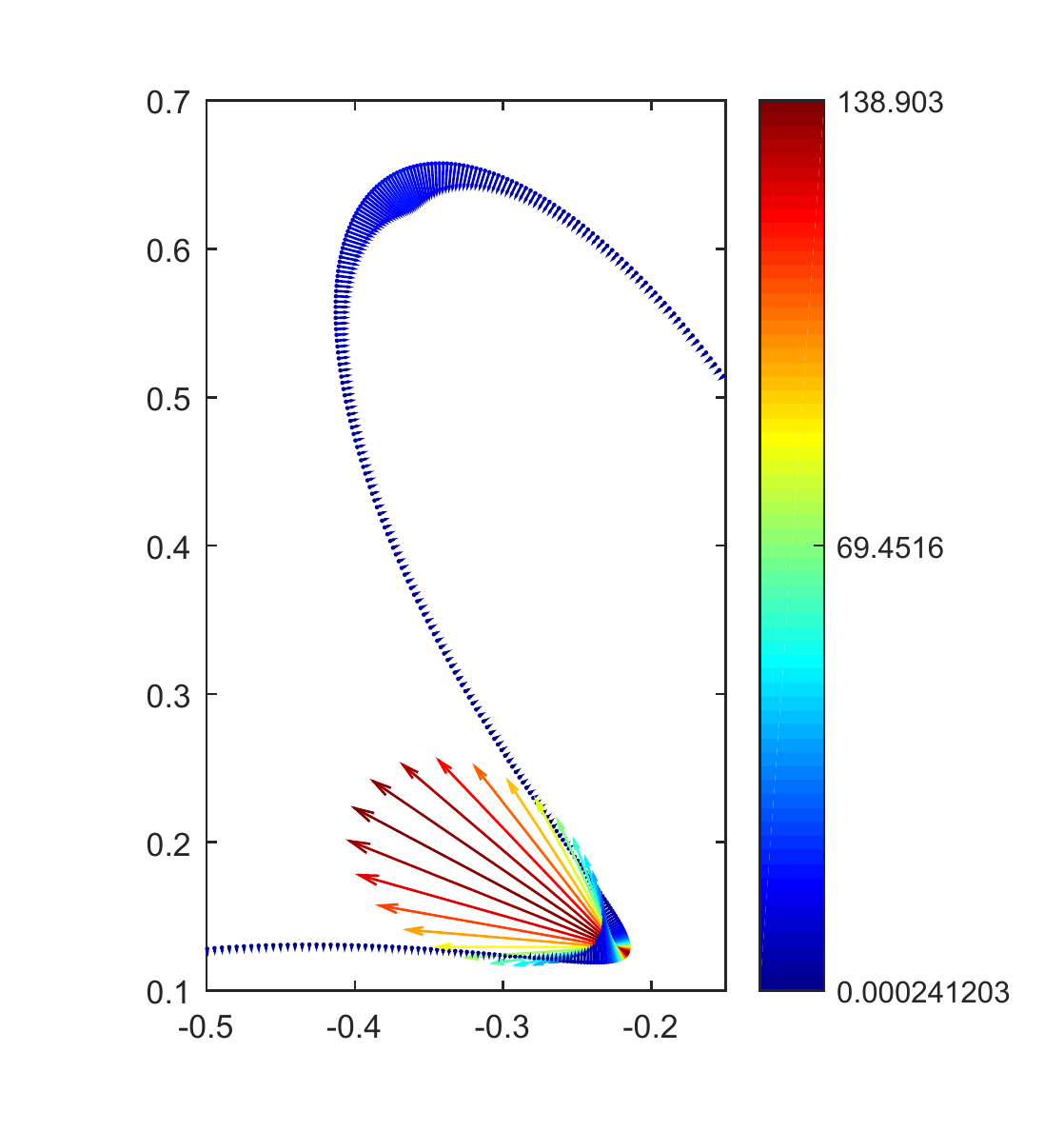}}
\end{center}
\caption{Orthogonal approximate curvature vectors along the discretized ellipse and flower. Arrows indicate the curvature vectors and colors indicate their norms.} \label{fig_flower}
\end{figure}

\subsubsection{Convergence rate}
\label{sect:convrate}

In this section, we compute and represent the evolution with respect to the number of points $N$ of the relative average error $E^{rel}=\frac{1}{N} \sum_{j=1}^N \frac{| H_\epsilon^N (x_j) - H(t_j)|}{\| H \|_\infty}$ for the orthogonal approximate mean curvature vector ~\eqref{eq_formulaModifiedMC_PointCloud} of point cloud varifolds sampled from the parametric flower. We compare the convergence rate of this error for the above choices of kernels pairs; more specifically 
we compute the convergence error for the varifold defined in \eqref{eq_parametric_varifold} both in the case where $T(jh)$ is the exact tangent and in the case where $T^{app}$ is computed by regression in an $R$--neighbourhood, with $R=\e/2$ (this situation is labelled as "regression" in all figures).

Theorem~\ref{thm:convergence3} guarantees the convergence under suitable assumptions of the orthogonal approximate mean curvature $H_{\rho,\xi,\e_i}^{V_i,\perp}$, and even provides a convergence rate. First, it is not very difficult to check that in the case where the point clouds are uniform samplings of a smooth curve, then the parameters $d_{i,1}$ and $\eta_i$ of \eqref{eq_thm_pointwise_cvSmooth_hyp2} are of order respectively $\frac{1}{N}$. As we already pointed out, our sampling is not globally uniform, but locally almost uniform and we expect the same order for $d_{i,1}$ and $\eta_i$. As for $d_{i,2}$ in \eqref{eq_thm_pointwise_cvSmooth_hyp3}, if the tangents are exact, then $d_{i,2}$ is also of order $\frac{1}{N}$, otherwise, it depends essentially on the radius of the ball used to perform the regression. Here we set $R=\e/2$, which is not a priori optimal. If we want to estimate the mean curvature at some point $x$ of the curve, then we will apply formula~\eqref{eq_formulaModifiedMC_PointCloud} to the closest point in the point cloud, which is at distance of order $\frac{1}{N}$ to $x$ (this corresponds to what is denoted $|z_i -x|$ in Theorem~\ref{thm:convergence3}). To summarize, according to these considerations together with Theorem~\ref{thm:convergence3}, we expect to observe convergence under the assumption
\[
\frac{1}{N \e} \rightarrow 0\,,
\]
with a convergence rate of order $\displaystyle \frac{1}{N \e} + \e$, at least in the case where the tangents are exact. We start with studying two different cases: first with $\frac{1}{N \e} = N^{-1/4}$, where we expect convergence with rate at least $N^{-1/4}$, and then with $\frac{1}{N \e} = 0.01$, for which Theorem~\ref{thm:convergence3} is not sufficient to guarantee that convergence holds. In both cases, we focus on $\frac{1}{N \e}$ which is the leading term.

\setcounter{subfigure}{0}
\begin{figure}[!htbp]
\begin{center}
\subfigure[\label{fig_Neps_100}]{\includegraphics[width=0.40\textwidth]{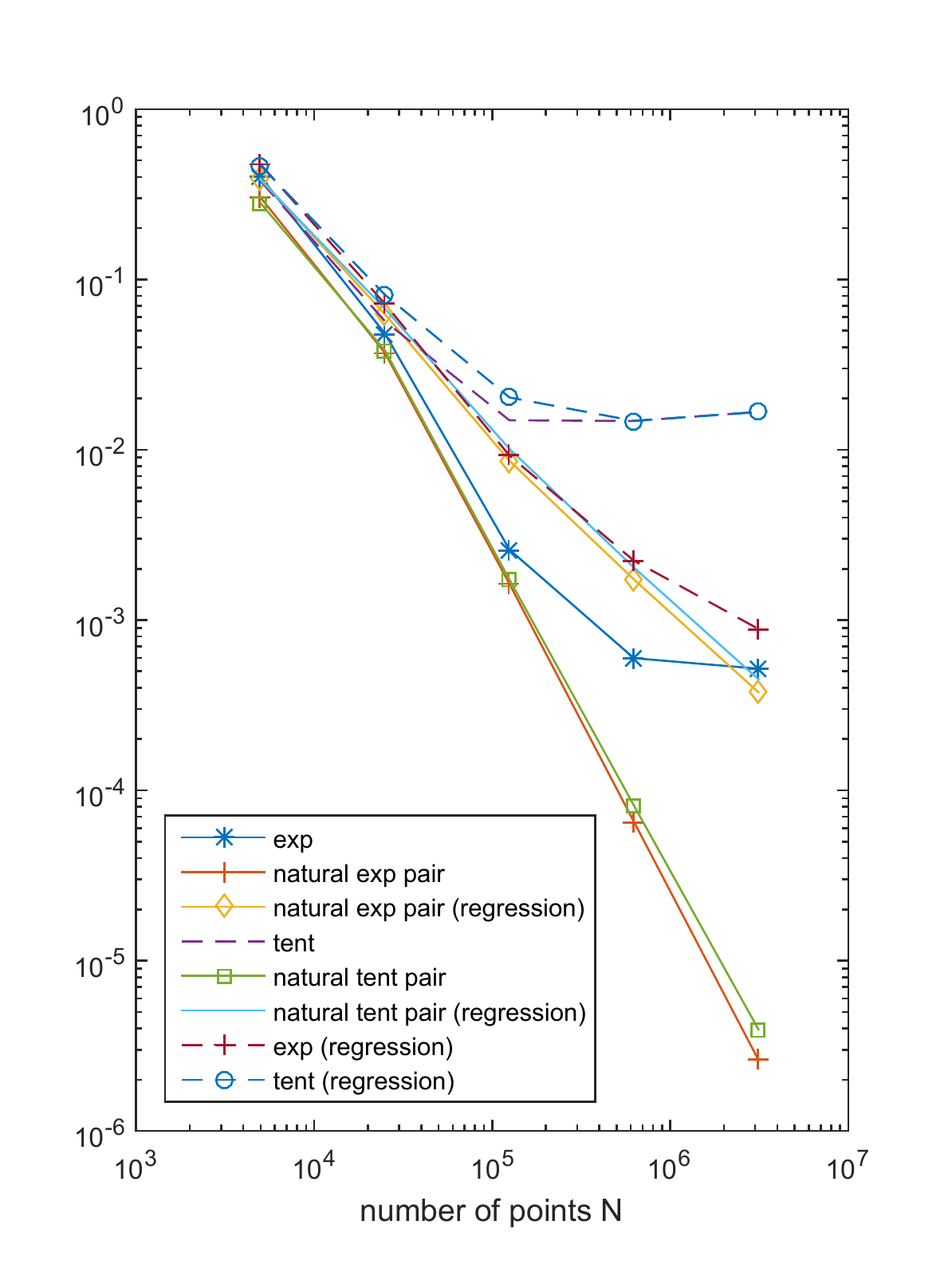}}
\subfigure[\label{fig_Neps_10_34}]{\includegraphics[width=0.40\textwidth]{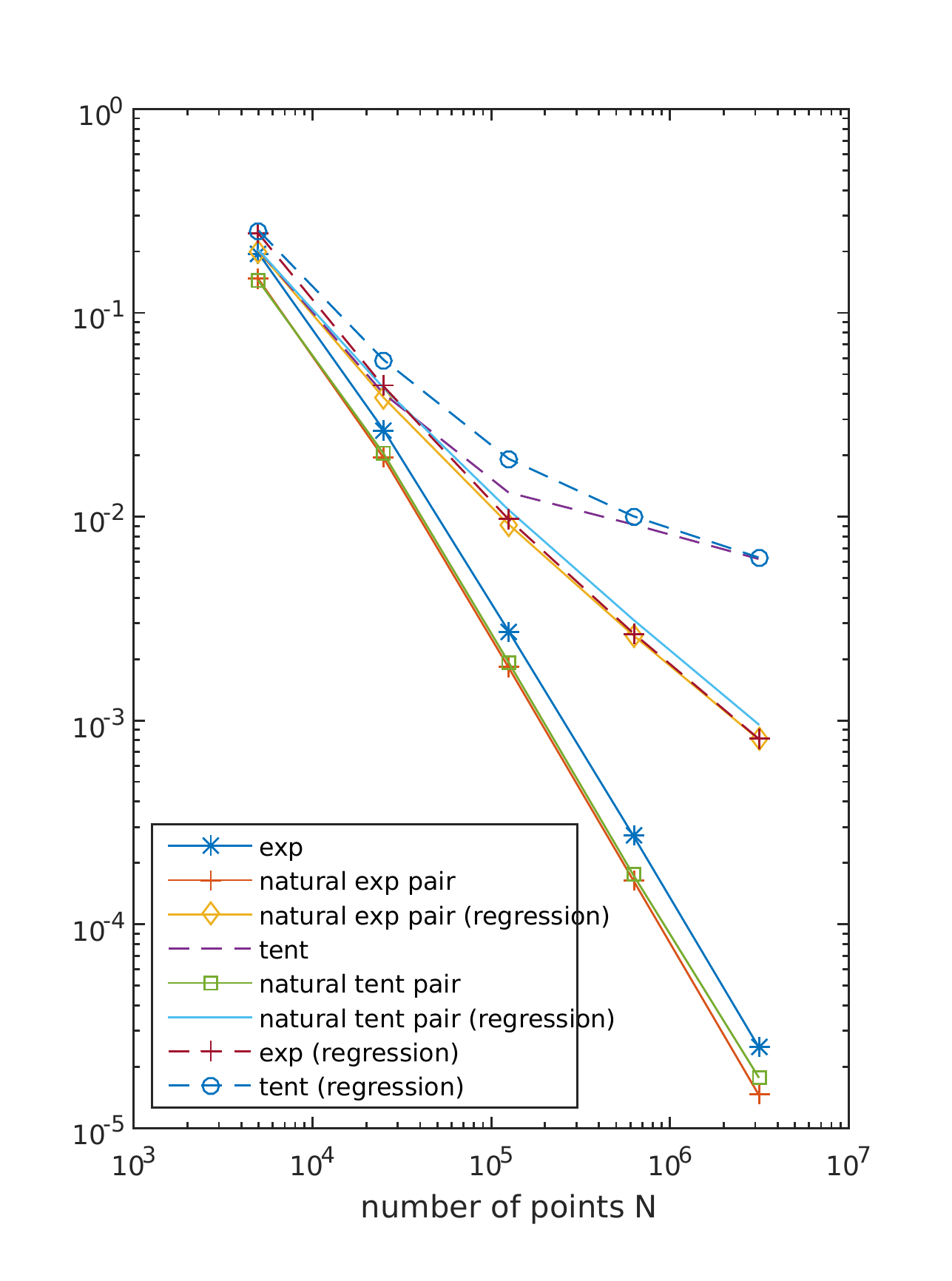}}
\caption{Average error (log-log scale) for the orthogonal approximate mean curvature of the subsampled parametric flower, for increasing values of $N$, and with either $\epsilon=\frac{100}{N}$ (left) or $\epsilon= \left( \frac{10}{N} \right)^{3/4}$ (right). The number of points in the neighborhood used for estimating the curvature is constant for the left experiment, and scales as $10N^{1/4}$ for the right experiment.}
\end{center}
\end{figure}

We use a log-log scale to represent the resulting relative average error \eqref{eq_rel_average_error} as a function of the number of sample points $N$ for $\epsilon = \frac{100}{N}$ (Figure~\ref{fig_Neps_100}) and $\epsilon = \left( \frac{10}{N} \right)^{3/4}$ (Figure~\ref{fig_Neps_10_34}).
We remark that the number $N_{\mbox{\tiny neigh}}$ of points in a neighborhood $B_{\e}(x)$ is proportional to $\epsilon N$, which takes the values $100$ and $10^{3/4}N^{1/4}$, respectively, for the above choices of $\epsilon$. Interestingly, the experiments show a good convergence rate when choosing a natural kernel pair, even in the cases when $\frac{1}{N \e}$ is constant (thus when it does not converge to $0$!). Furthermore, the convergence using natural kernel pairs and approximate tangents computed by regression is even faster than when using exact tangents and the tent kernel. We recall that the tent kernel does not satisfy Hypothesis \ref{rhoepsxieps} since it is only Lipschitz, nevertheless the corresponding natural pair $(\rho_{tent},\xi_{tent})$ shows the same convergence properties as the smooth natural pair $(\rho_{exp},\xi_{exp})$. This suggests that the (NKP) property is even more effective than the smoothness of the kernel profiles. Finally, when the tangents are not exact the convergence is slower. This is consistent with the fact that parameter $d_{i,2}$ in \eqref{eq_thm_pointwise_cvSmooth_hyp3} depends on the radius $R$ of the ball used to compute the regression tangent line (we recall that $R=\e/2$) which represents an additional parameter to be possibly optimized.

\subsubsection{Varying density and projection onto the normal}
\label{sec:varyingdensity}
\begin{figure}[!htbp]
\begin{center}
\includegraphics[width=0.40\textwidth]{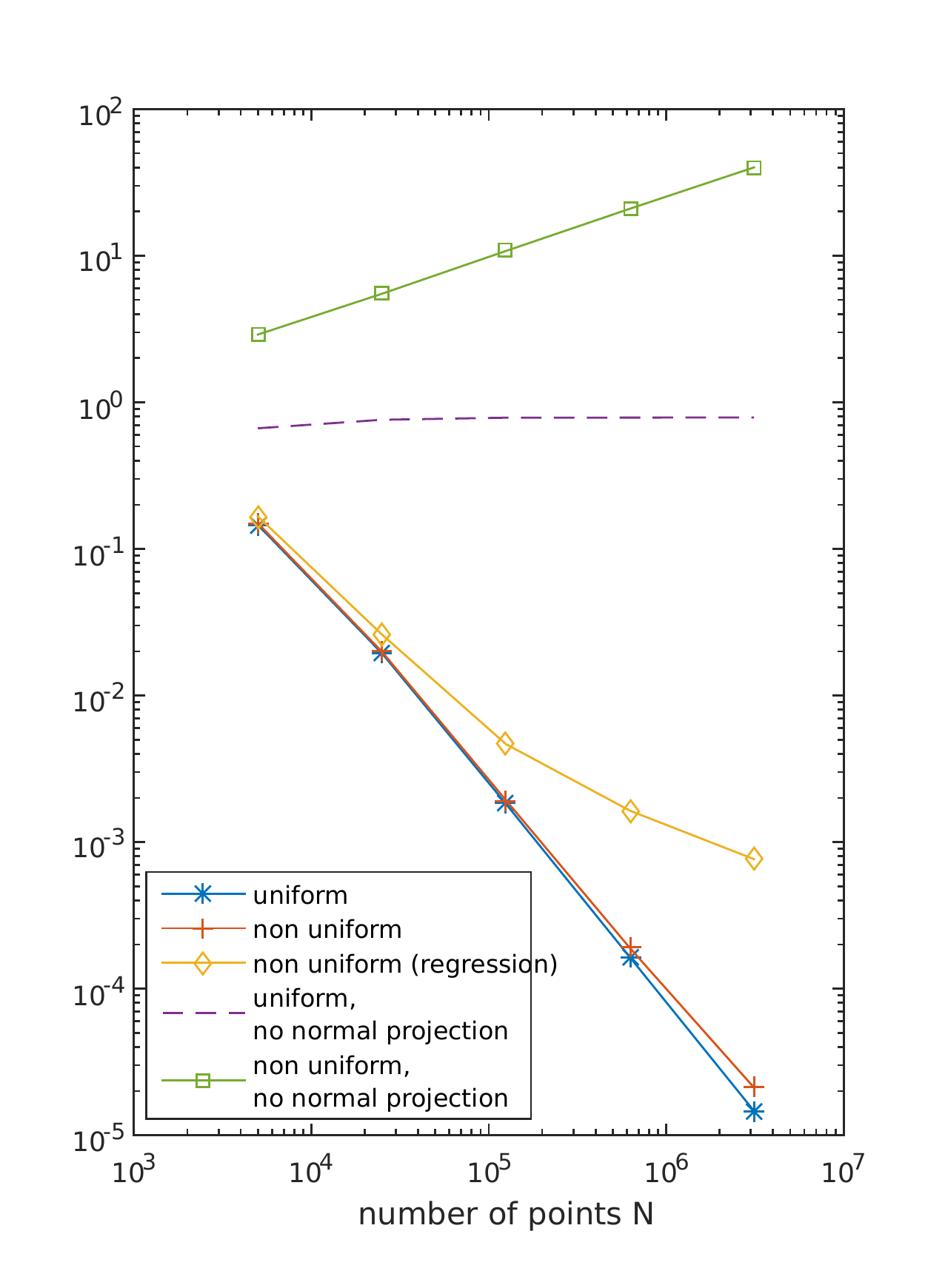}
\caption{Average error (log-log scale) for the approximate mean curvature of the subsampled parametric flower, for increasing values of $N$ and $\epsilon=(10/N)^{3/4}$. Comparison of the behavior of $H_{\rho,\xi,\e}^{V_N}$ and $H_{\rho,\xi,\e}^{V_N,\perp}$ on an (almost) uniform and a non uniform sampling of the flower.\label{fig_DensityNotUniform}
}
\end{center}
\end{figure}

In all previous numerical tests,  we assumed that $m_j \simeq m$ is true, at least locally, which yields the simplified formula \eqref{eq_formulaModifiedMC_PointCloud}. This assumption makes sense for a point cloud with almost uniform distribution of masses, but is less realistic in the other cases. However, when considering a smooth surface $M$ endowed with a smoothly varying density function $\theta$, the tangential component of the generalized mean curvature of the associated varifold is non-zero in general (and related to the tangential gradient of $\theta$) while the normal component still coincides with the classical mean curvature of $M$. Therefore formula \eqref{eq_formulaModifiedMC_PointCloud} allows to cancel the tangential perturbations artifacts due to the non-uniformity of the sampling. 

In order to illustrate this property of the orthogonal approximate mean curvature $H_{\rho,\xi,\e}^{V,\perp}$, we consider a non-uniform discretization of the flower. To this aim, starting from the uniform discretization $jh$ of the parametrization interval $[0,2\pi]$, with $j=0,\dots,N-1$, we define $t_{j} = (j + n_{j})h$ where $n_{j}$ are i.i.d.  Gaussian random variables with zero mean and variance $1$. Then we define $V_{N}$ exactly as in \eqref{eq_parametric_varifold} or in \eqref{eq_parametric_varifold_approx} replacing $jh$ with $t_{j}$. The relative error $E^{rel}$ is computed also by replacing $jh$ with $t_{j}$ in \eqref{eq_rel_average_error}. We perform the tests using the natural kernel pair $(\rho_{exp},\xi_{exp})$. The aim is to compare the behavior of $H_{\rho,\xi,\e}^{V,\perp}$ (default choice in Figure \ref{fig_DensityNotUniform}) and $H_{\rho,\xi,\e}^{V}$ (labeled as ``no normal projection'' in Figure \ref{fig_DensityNotUniform}) on the flower, in both uniformly and non-uniformly discretized cases (respectively labeled as ``uniform'' and ``non uniform'' in Figure \ref{fig_DensityNotUniform}). We also consider the sub-case of approximate tangents in the non-uniformly discretized case, but only for $H_{\rho,\xi,\e}^{V,\perp}$ since we observed that the error associated with $H_{\rho,\xi,\e}^{V}$ does not converge to zero, even when the tangents are exact.

Figure~\ref{fig_DensityNotUniform} shows the plots of the relative errors  computed with respect to the number of points $N$ in a log-log scale, with $\e = (10/N)^{3/4}$. On the one hand we observe that in both uniformly and non-uniformly discretized cases the error associated with $H_{\rho,\xi,\e}^{V}$ does not converge to zero, or even diverges, which is not incompatible with Theorem \ref{thm:convergence2} since we have that $\frac{d_{i}}{\e_{i}^{2}} \sim \sqrt N$, which of course is not infinitesimal. On the other hand, for the reason given above, the convergence of $H_{\rho,\xi,\e}^{V,\perp}$ is comparable in both uniformly and non-uniformly discretized cases, even when tangent planes are computed approximately by regression.

\subsubsection{Convergence rate in the $2D$--smooth case}

Up to now, it has been evidenced that in the smooth case it is reasonable to use $H_{\rho,\xi,\e}^{V,\perp}$ with a smooth natural kernel pair. The experiments in Figure~\ref{fig_convergenceRate} are therefore obtained for $H_{\rho,\xi,\e}^{V,\perp}$ computed with the natural kernel pair $(\rho_{exp},\xi_{exp})$. As already mentioned, Theorem~\ref{thm:convergence3} gives a convergence speed of order at least $\frac{1}{N \e}$ in the case where the tangents are exact. When the tangents are computed by regression, then the convergence is of order (at least) $\frac{d_2}{\e}$ where $d_2$ is the maximal pointwise error on the tangents resulting from the regression. We thus plot in Figure~\ref{fig_convergenceRate} the decay of the relative error $E^{rel}$ with respect to this ratio $\frac{1}{N \epsilon}$ for different choices of $\e, \, N$, with either exact or approximate tangents (computed by regression in a ball of radius $R$), and with or without an additional Gaussian white noise of variance $\sigma = \frac{1}{N}$. More precisely, we fix i.i.d. Gaussian random variables $(n_j^1)_j, \, (n_j^2)_j$ with zero mean and variance $\sigma$, and we define from $V_N$ in \eqref{eq_parametric_varifold} a noisy point cloud varifold $V_N^\sigma$ as
\begin{equation*}
\text{either}\quad V_N^\sigma = \sum_{j=1}^N m_j \delta_{((x(jh),y(jh)+(n_j^1,n_j^2))} \otimes \delta_{T(jh)} \quad \text{or} \quad V_N^\sigma = \sum_{j=1}^N m_j \delta_{((x(jh),y(jh)+(n_j^1,n_j^2))} \otimes \delta_{T_j^{app}}
\end{equation*} 
Here, $T_j^{app}$ is computed by linear regression from the noisy positions $\{ (x(kh),y(kh)+(n_k^1,n_k^2) \}_k$ in a ball of radius $R$. The relative error is defined as
\[
E^{rel}=\frac{1}{N} \sum_{j=1}^N \frac{| H_{\rho,\xi,\e}^{V_N^\sigma,\perp} ((x(jh),y(jh)+(n_j^1,n_j^2)) - H(jh)|}{\| H \|_\infty} \: .
\]

\begin{figure}[!htbp]
\begin{center}
\includegraphics[width=0.80\textwidth]{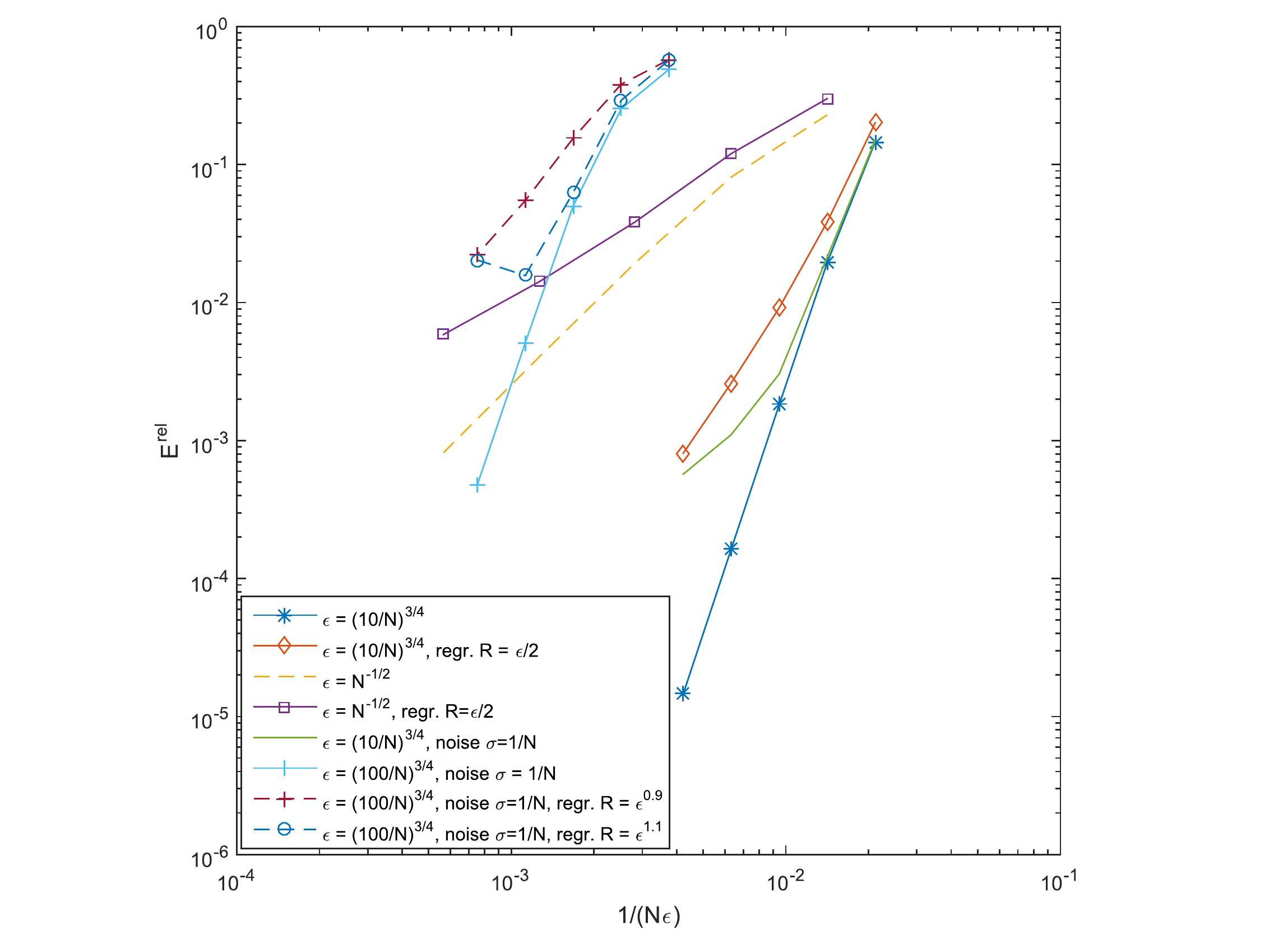}
\caption{Average error $E^{rel}$ (log-log scale) for the approximate mean curvature $H_{\rho_{exp},\xi_{exp},\e}^{V_N}$ of the subsampled parametric flower, plotted with respect to $\frac{1}{N \e}$.\label{fig_convergenceRate}}
\end{center}
\end{figure}

We observe in Figure~\ref{fig_convergenceRate} that the convergence is at least of order $1$ with respect to $\frac{1}{N \e}$ as foreseen by Theorem~\ref{thm:convergence3}, and even quite better in the case where there is no additional noise. In this latter case, when the tangents are exact, convergence is still of good quality. But when the tangents are computed by regression in a ball of radius $R$, the relative error $E ^{rel}$ is very sensitive to the regression error and thus to $R$. It seems that, in some cases where additional white noise is introduced, taking $R$ larger than $\e$ ($R=\e^{9/10}$ for instance) produces a lower error.

\subsection{The approximate mean curvature near singularities}\label{cross}

In this section, we illustrate the specific features of the approximate mean curvatures $H_{\e,\rho,\xi}^{V}$ and $H_{\e,\rho,\xi}^{V,\perp}$ near singularities. Consistently with the properties of the classical generalized mean curvature of varifolds, $H_{\e,\rho,\xi}^{V}$ and $H_{\e,\rho,\xi}^{V,\perp}$ both preserve the zero mean curvature of straight crossings, as confirmed by the experiment on the "eight" (see Figure~\ref{fig_eight}). In this case using $H_{\e,\rho,\xi}^{V,\perp}$ does not affect the reconstruction of the zero curvature at the crossing point, while it has the advantage of being more consistent at regular points (see the discussion in Section \ref{sec:varyingdensity}). In Figure~\ref{fig_eight} we plot the curvature vectors and intensities computed using the natural kernel pair $(\rho_{exp},\xi_{exp})$ on the eight curve sampled with $N=10000$ points and exact tangents, for $\e=100/N=0.01$.

\begin{figure}[!htbp]
\centering
\includegraphics[width=0.65\textwidth]{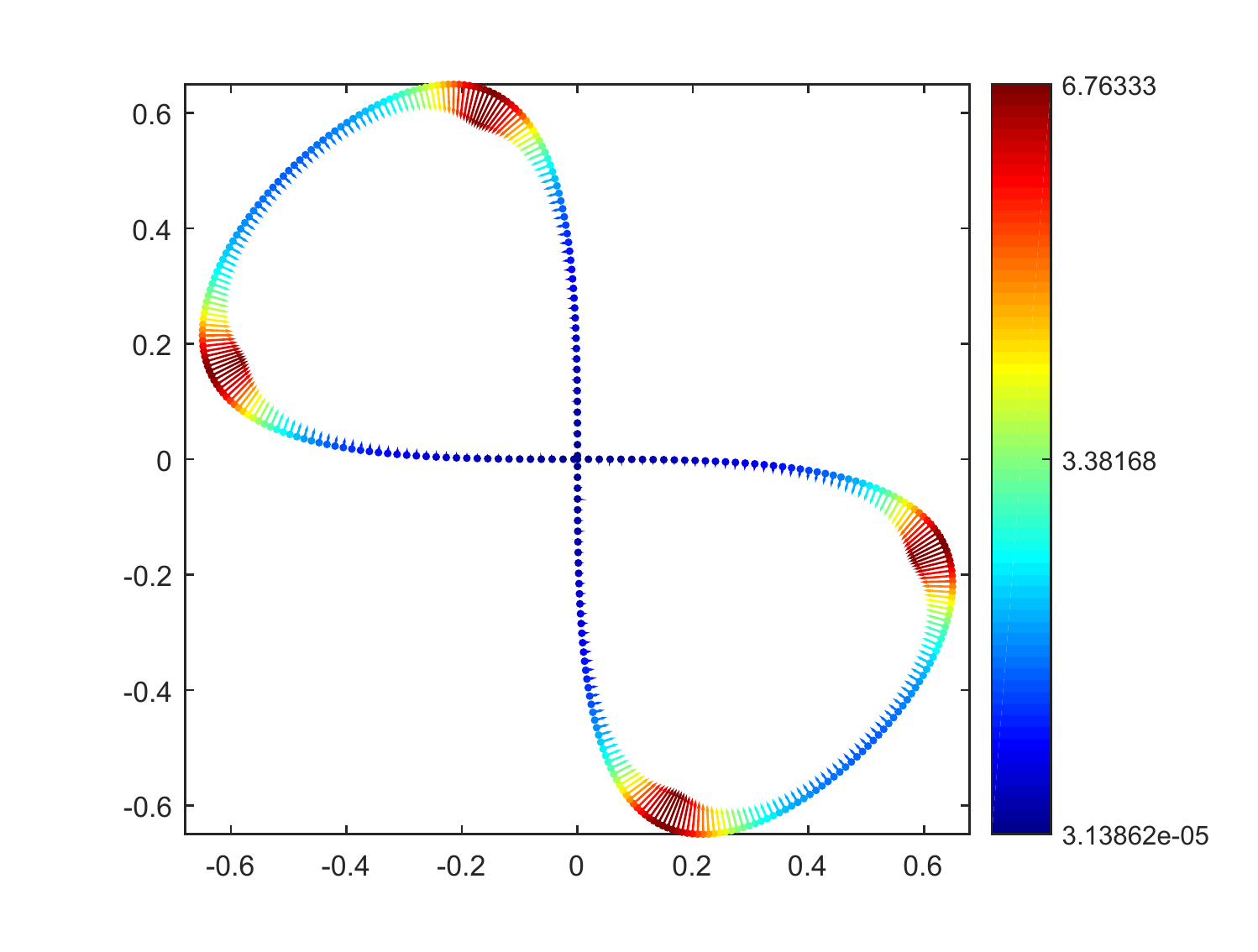}
\caption{Curvature vector and intensity computed with the natural kernel pair $(\rho_{exp},\xi_{exp})$ on the eight sampled with $N=10000$ points and with $\epsilon = 100/N = 0.01$, with exact tangents. For visualization purposes we only show $5\%$ of the points in the cloud.} \label{fig_eight}
\end{figure}

More generally, our model is able to deal correctly with singular configurations whose canonically associated varifold has a first variation $\delta V$ which is absolutely continuous with respect to $\V$. To illustrate this, we show the results of some tests performed on a union of two circles with equal radius and on a standard double bubble in the plane. 

First, we compare the behavior of $H_{\e,\rho,\xi}^{V}$ and $H_{\e,\rho,\xi}^{V\perp}$ in a neighborhood of an intersection point of the two circles (see Figure \ref{fig:doublecircle}). We define a point cloud as a uniform sampling of the union of both circles, with a total number of points $N=10000$ and with exact tangents. We also choose $(\rho_{exp},\xi_{exp})$ as natural kernel pair, and $\e = 100/N=0.01$ as in the previous test with the "eight". Figure \ref{fig:doublecircle} (a) and (b) show the curvature vectors and intensities of $H_{\e,\rho,\xi}^{V}$, while Figure \ref{fig:doublecircle} (c) shows them for $H_{\e,\rho,\xi}^{V\perp}$. From the point of view of pointwise almost everywhere convergence, both approximate curvatures behave equivalently well, since the error in the reconstruction of the curvature is localized in an $\e$-neighborhood of the crossing point. On one hand, due to the linearity of the first variation $\delta V$, the expected curvature $H$ of the union $\cC_{1}\cup \cC_{2}$ of the two circles at the crossing point $p$ is the average of the curvatures $H_{1}$ and $H_{2}$ of, respectively, $\cC_{1}$ and $\cC_{2}$ at $p$. Indeed $\delta V = H_{1}\, d\cH^{1}_{|\cC_{1}} + H_{2}\, d\cH^{1}_{|\cC_{2}}$, whence one deduces that $H(p) = \frac{H_{1}(p)+H_{2}(p)}{2}$ and if $p$ is an intersection point of the two circles, $|H(p)|=\sqrt{3}\approx1.73$ which is consistent with the numerical value obtained at $p$ (see Figure \ref{fig:doublecircle} (b)). On the other hand, the crossing point is negligible with respect to $\V$ and therefore the pointwise value of $H(p)$ is not relevant in the continuous setting. Nevertheless, in the discrete setting there is a significant difference between the two proposed definitions of approximate mean curvature. More precisely, the one provided by $H_{\e,\rho,\xi}^{V}$ enforces a continuous mean curvature even at the crossing point, where one obtains the expected average value $H(p) = \frac{H_{1}(p)+H_{2}(p)}{2}$, see Figure~\ref{fig:doublecircle} (b), whereas continuity cannot hold for $H_{\e,\rho,\xi}^{V,\perp}$, as one can see in Figure \ref{fig:doublecircle} (c).

\begin{center}
\setcounter{subfigure}{0}
\begin{figure}[!htbp]
\subfigure[]{\includegraphics[width=0.43\textwidth]{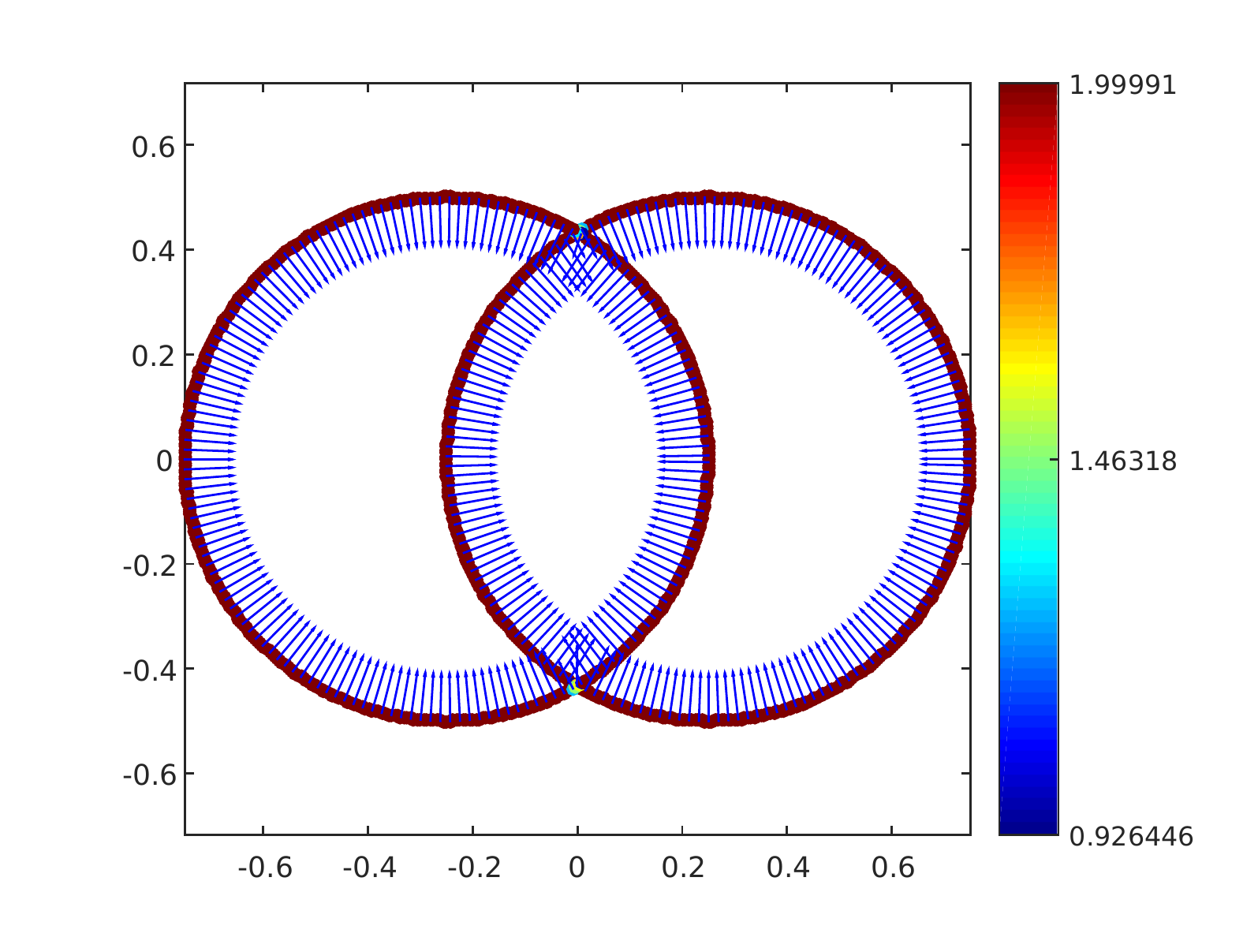}}
\subfigure[]{\includegraphics[width=0.43\textwidth]{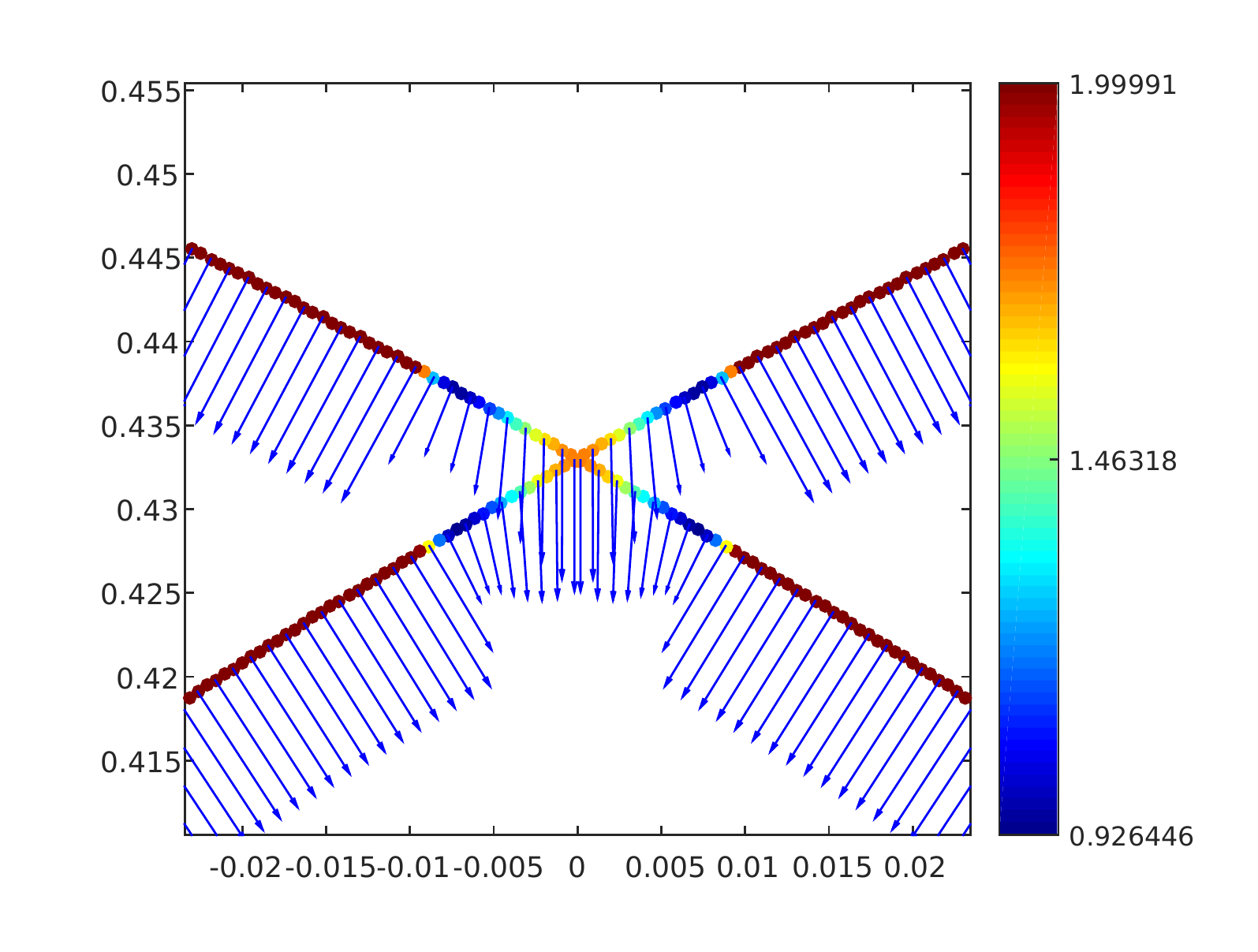}}
\subfigure[]{\includegraphics[width=0.43\textwidth]{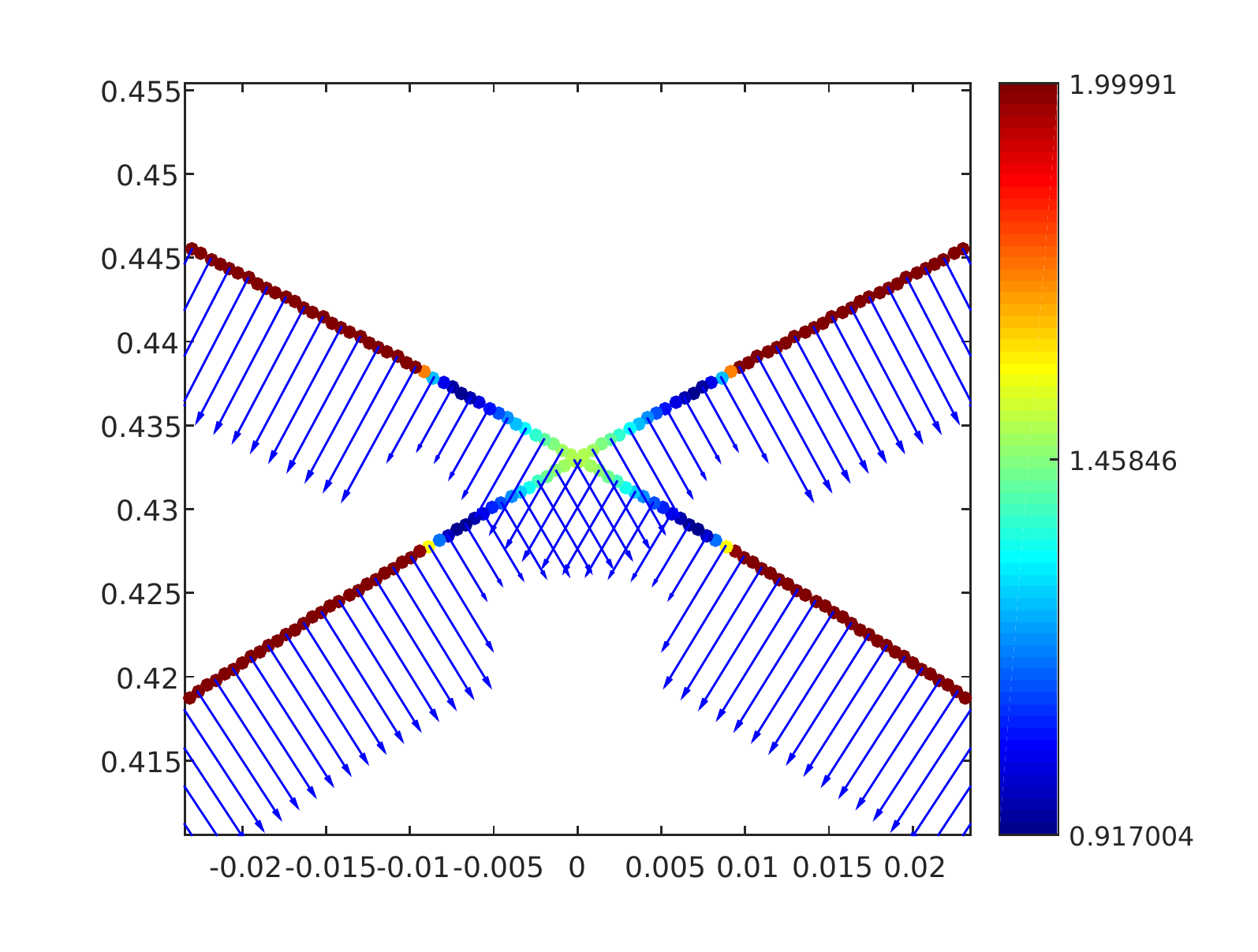}}
\caption{Curvature vector and intensity computed with the natural kernel pair $(\rho_{exp},\xi_{exp})$ on two intersecting circles sampled with $N =10000$ points and with $\epsilon = 100/N = 0.01$, without projection onto the normal in $(a)$ and $(b)$ and with projection on $(c)$. Tangents are exact.}\label{fig:doublecircle}
\end{figure}
\end{center}

Second, we consider a standard double bubble in $2$ dimensions (see Figure \ref{fig:2bubble2D}(a) and \cite{CicaleseLeonardiMaggi} for details on double bubbles), whose radii of the external boundary arcs are, respectively, $1$ and $0.6$. The corresponding point cloud varifold $V$ is obtained by a uniform sampling of $800$ points taken on the three arcs of the bubble, each endowed with a unit mass and tangent computed by regression. Again, we choose $(\rho_{exp},\xi_{exp})$ as natural kernel pair, and $\e = 0.15$. Figure \ref{fig:2bubble2D}(b) shows the curvature vectors and intensities of $H_{\e,\rho,\xi}^{V}$ (up to a fixed renormalization that is applied for a better visualization). In order to get rid of the oscillation of the curvature near the singularities (as it occurred in the previous test, see again Figure \ref{fig:doublecircle}) we have also applied a simple averaging of the reconstructed curvature at the scale $2\e$, which gives the nicer result shown in Figure \ref{fig:2bubble2D}(c). 
We remark that the curvature vector defined on points that are very close to the theoretical singularity is consistent with the one obtained by direct computation on the (continuous) standard double bubble. More precisely, we obtain a numerical value of $(0.107,-0.809)$ for the mean curvature near the singularity shown in Figure \ref{fig:2bubble2D}, to be compared with the expected value $(0,-0.839)$, hence with a relative error of $13\%$. If we redo the same experiment but with twice the number of points, that is $N=1600$ and  $\e=0.075$, we get a relative error of $7\%$. Further tests involving standard double bubbles will be described in the next section.

\begin{center}
\setcounter{subfigure}{0}
\begin{figure}[!htbp]
\subfigure[]{\includegraphics[width=0.25\textwidth]{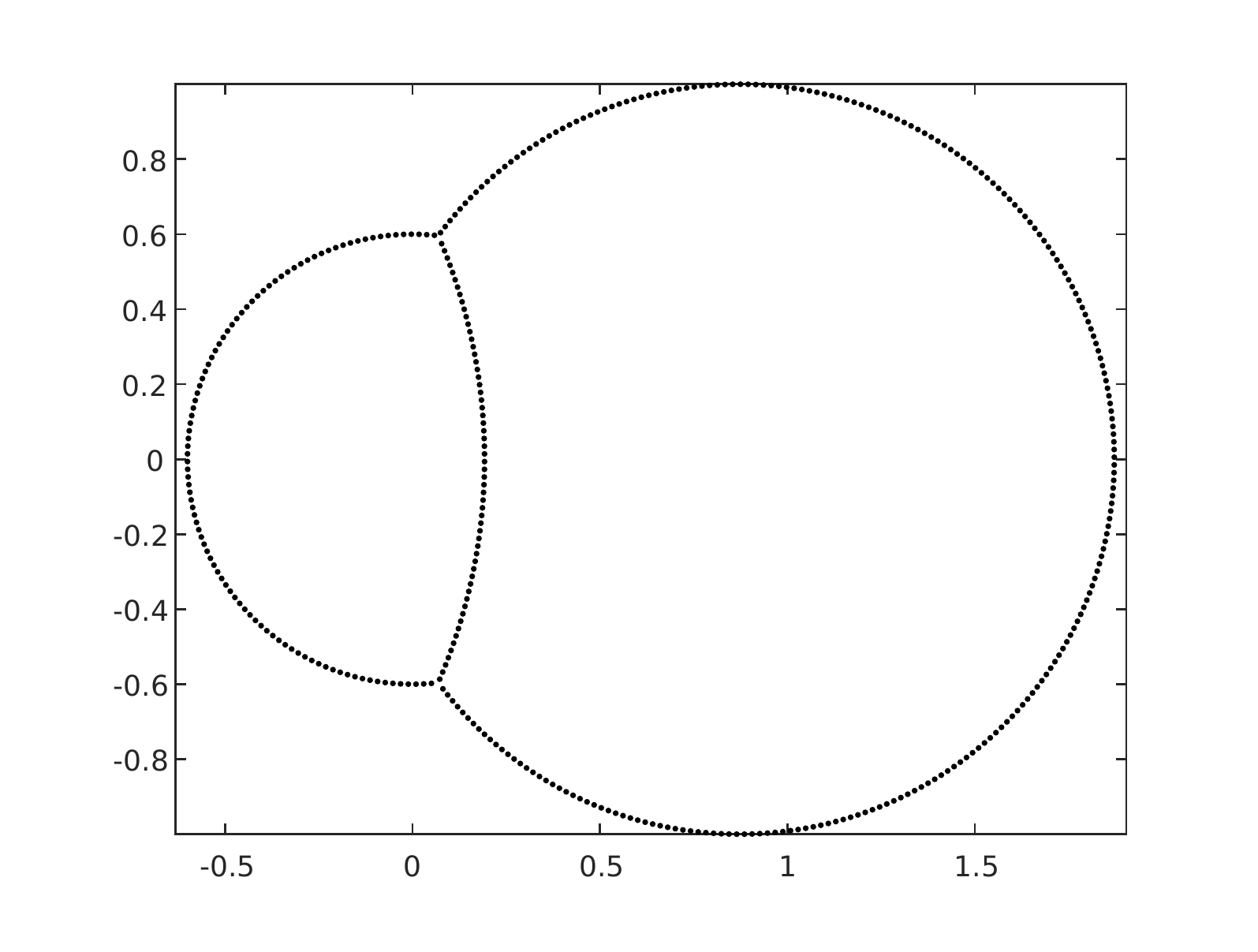}}\\
\subfigure[]{\includegraphics[width=0.48\textwidth]{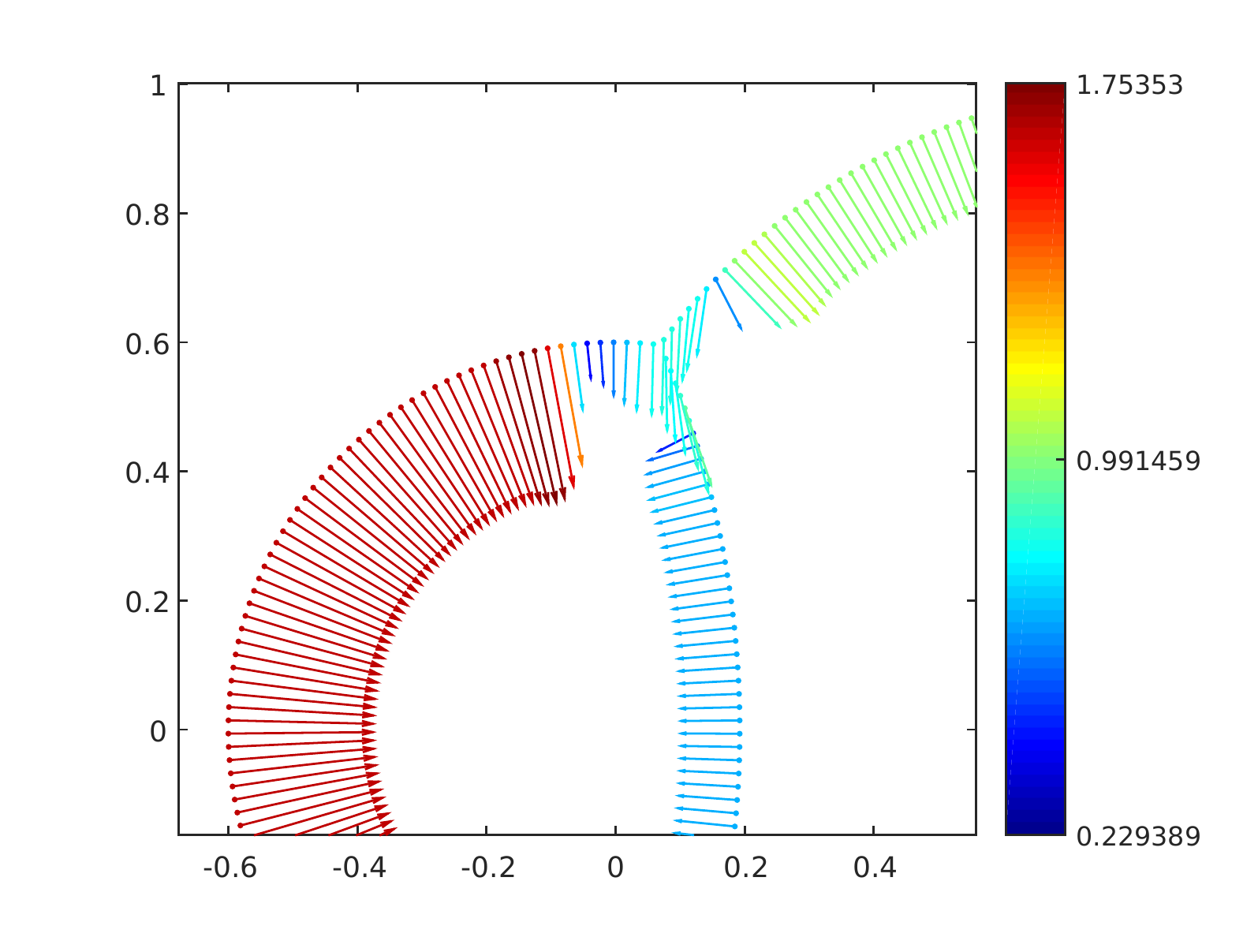}}
\subfigure[]{\includegraphics[width=0.48\textwidth]{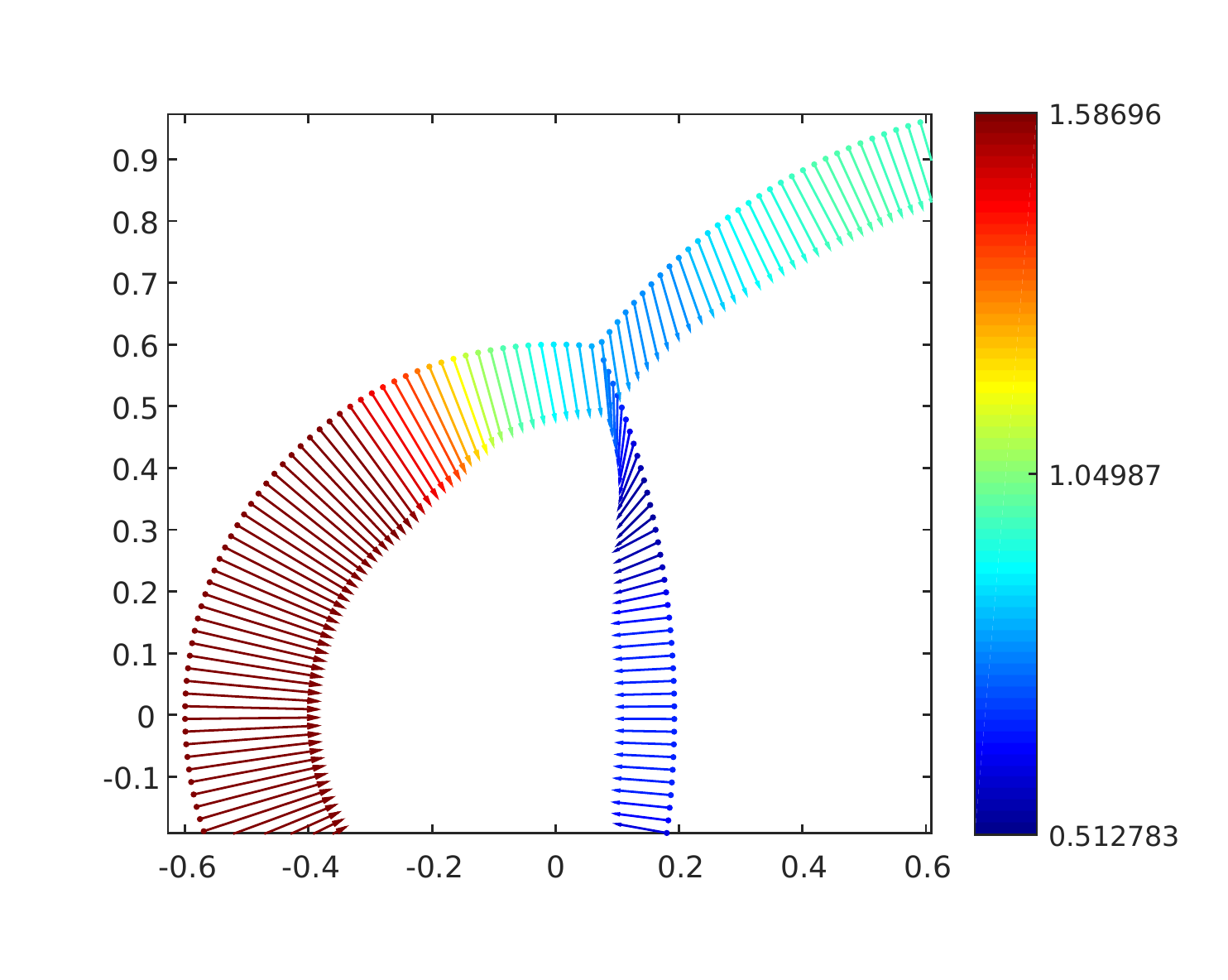}}
\caption{Curvature vectors and intensities computed with the natural kernel pair $(\rho_{exp},\xi_{exp})$ on a standard double bubble (radii $1$ and $0.6$) sampled with $N =800$ points and with $\epsilon = 0.15$, without projection onto the normal in $(b)$ and with additional averaging of the curvature at scale $2\e$ in $(c)$. Tangents are computed by regression.}\label{fig:2bubble2D}
\end{figure}
\end{center}

\subsection{Approximate mean curvatures of 3D point clouds}\label{3D}

In this last section we present some tests on 3D point clouds obtained either from parametrized shapes (specifically, a standard double bubble) or from given point cloud samples (a dragon and a statue).

In the first test (Figure~\ref{fig:2bubble3D}, a) we use colors (from blue for smaller values to red for larger values) to represent the intensities of the approximate mean curvature vectors computed for a point cloud discretization of the three spherical caps forming the boundary of a standard double bubble in 3D. The radii of the external caps are, respectively, $1$ and $0.7$. The cloud contains $N = 34378$ points and is endowed with tangents reconstructed via regression at the scale $\e \approx 0.111$. We compute the covariance matrix of centered coordinates (in a ball of radius $\e$) and we define the normal as the eigenvector associated with the smallest eigenvalue. The computation of the approximate mean curvature is performed in a ball of radius $\e$ as well, and we further average the approximate curvature at the scale $2\e$, as done in the test on the $2D$--double bubble. Moreover the cloud is an ``almost uniform'' discretization of the double bubble in the sense that some small ``holes'' are created along three meridian curves, as a consequence of rounding-type discretization errors.  All curvature vectors intensities are shown, but only (minus) the curvature vectors near the singular arc are represented for the sake of readability. As can be observed, these approximate vectors lie essentially in the same expected plane.
Even though a more uniform discretization can be constructed, we have preferred to keep the almost-uniform one in order to show the behavior of the (averaged) approximate mean curvature. The results of the test show a pattern similar to the one obtained in the 2D case in proximity of the singular circle (see Figure \ref{fig:2bubble3D}, a)). Moreover, when numerically computing the average of the intensity of the mean curvature along the singular circle, we obtain $1.51$, to be compared with $1.46$ which is an approximate value of the norm of the average of mean curvature vectors of the three  intersecting spheres.
Some small deviations from the true mean curvature are localized near the small ``holes''. Of course such deviations can be reduced by refining the discretization and by taking curvature averages over neighborhoods containing more points. The overall outcome shows that the reconstruction of the mean curvature near singularities is consistently enforced by our method. We provide in Figure~\ref{fig:2bubble3D} two more examples with a sampled double bubble. In b) and c) the bubble has external caps with same radius $r=1$. The central cap is therefore a disk. Again, the consistency of the curvature vectors computed near the singular arc can be observed. Only these vectors are shown in b), but all (minus) curvature vectors are shown in c).

\begin{figure}[!htbp]
\begin{tabular}{cc}
\multicolumn{2}{c}{\subfigure[]{\includegraphics[width=0.65\textwidth]{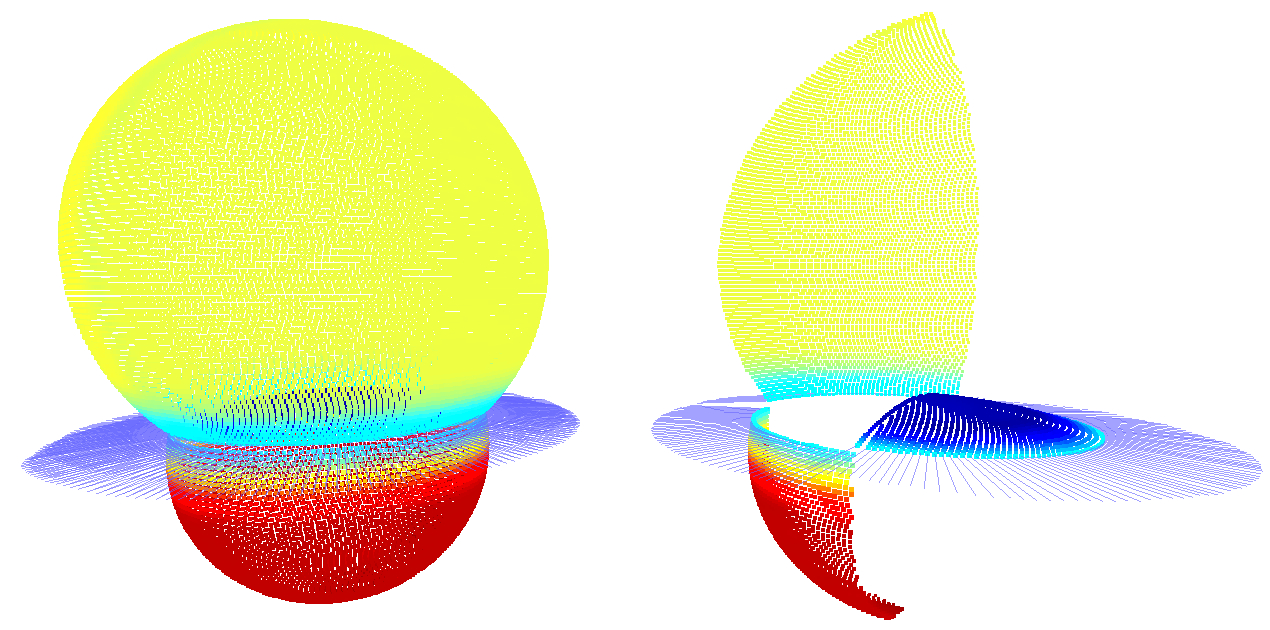}}}\\
\subfigure[]{\includegraphics[width=0.5\textwidth]{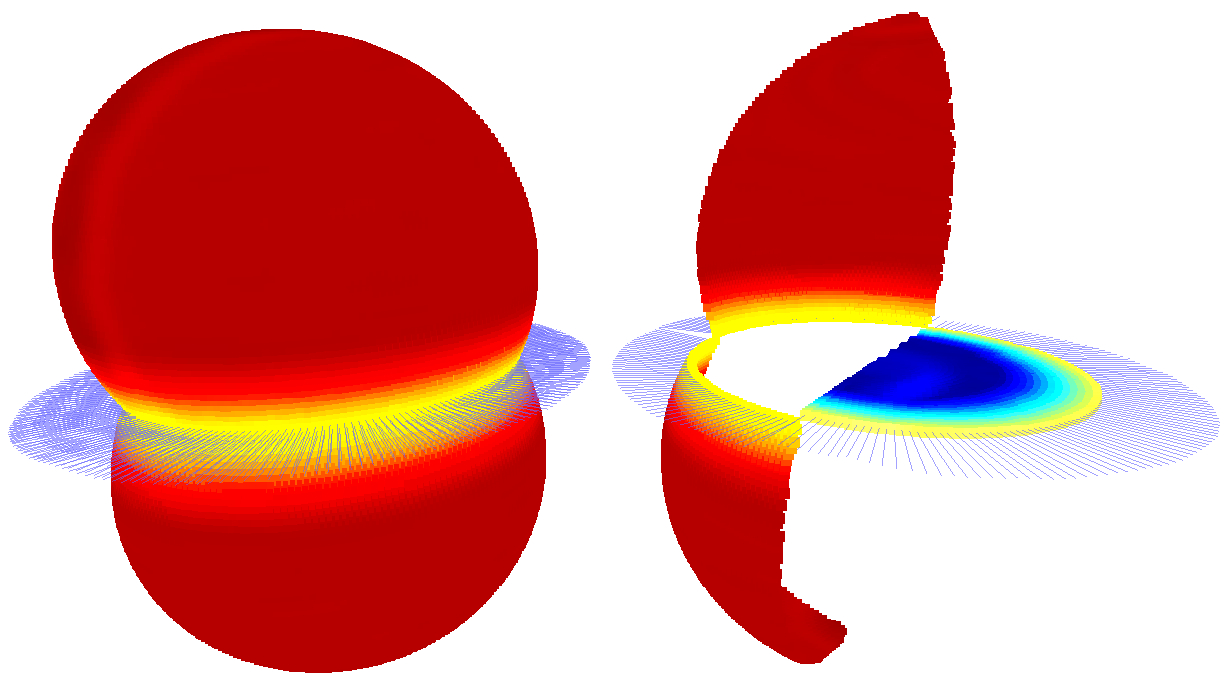}}&
\subfigure[]{\includegraphics[width=0.5\textwidth]{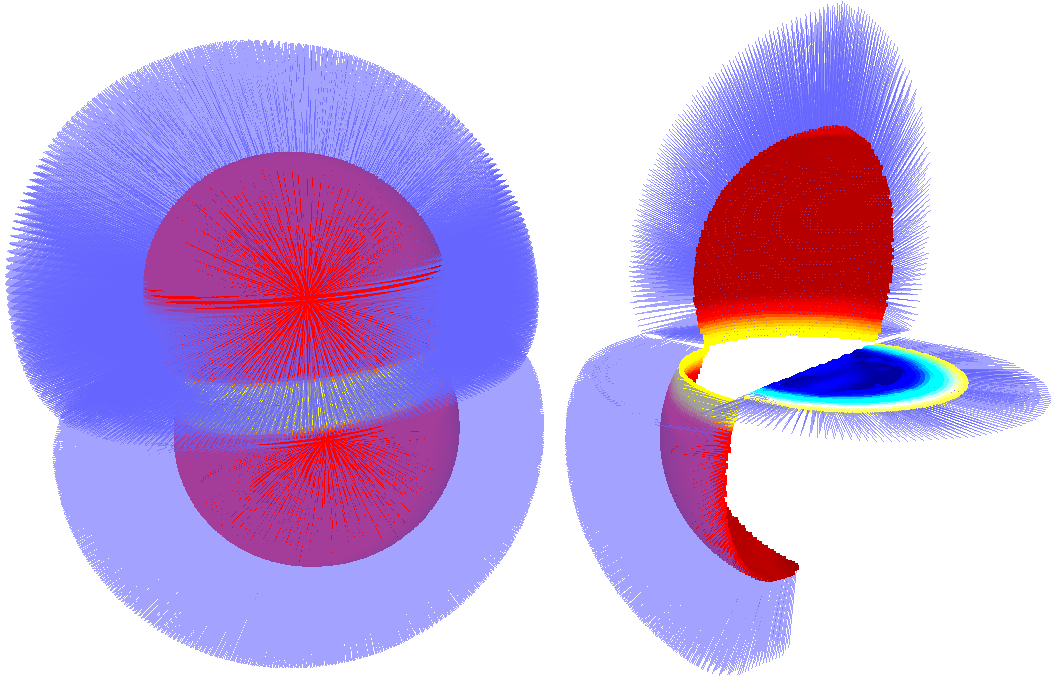}}
\end{tabular}
\caption{Curvature vectors and their intensities computed with the natural kernel pair $(\rho_{exp},\xi_{exp})$ on sampled $3D$--double bubbles (show in full and partial views). In Figure a), the bubble has external caps with radii $0.7$ and $1$,  is sampled with $N=34378$ points, and the computations are made with $\epsilon \approx 0.111$. The curvature vectors (with minus sign for the sake of readability) are shown only for the points which are closest to the singular circle. In b) and c), the bubble has externals caps with same radius $1$, is sampled with 33275 points, and $\epsilon\approx 0.131$. All curvature vectors (with minus sign) are shown in c). To improve the visualization, points are shown with larger size in b) and c). 
\label{fig:2bubble3D}}
\end{figure}

Our next 3D test point clouds are a "dragon" with $N = 435\,545$ points (Figure~\ref{fig_dragon}), and a statue with $543\,524$ points (Figure~\ref{buddha-fig}). We show with colors the norm of the approximate mean curvature vectors (computed with the natural kernel pair $(\rho_{exp},\xi_{exp})$) with post-projection onto the normals. In both cases, as the tangent plane is not a-priori known, we compute the normal direction at each point using regression. As the shapes are assumed to be regular, we use Formula~\eqref{eq_formulaModifiedMC_PointCloud}, that is, with projection onto the normal at the point and without additional averaging.

\setcounter{subfigure}{0}
\begin{figure}[!htp]
\begin{tabular}{cc}
\multicolumn{2}{c}{\subfigure[]{\includegraphics[width=0.98\textwidth]{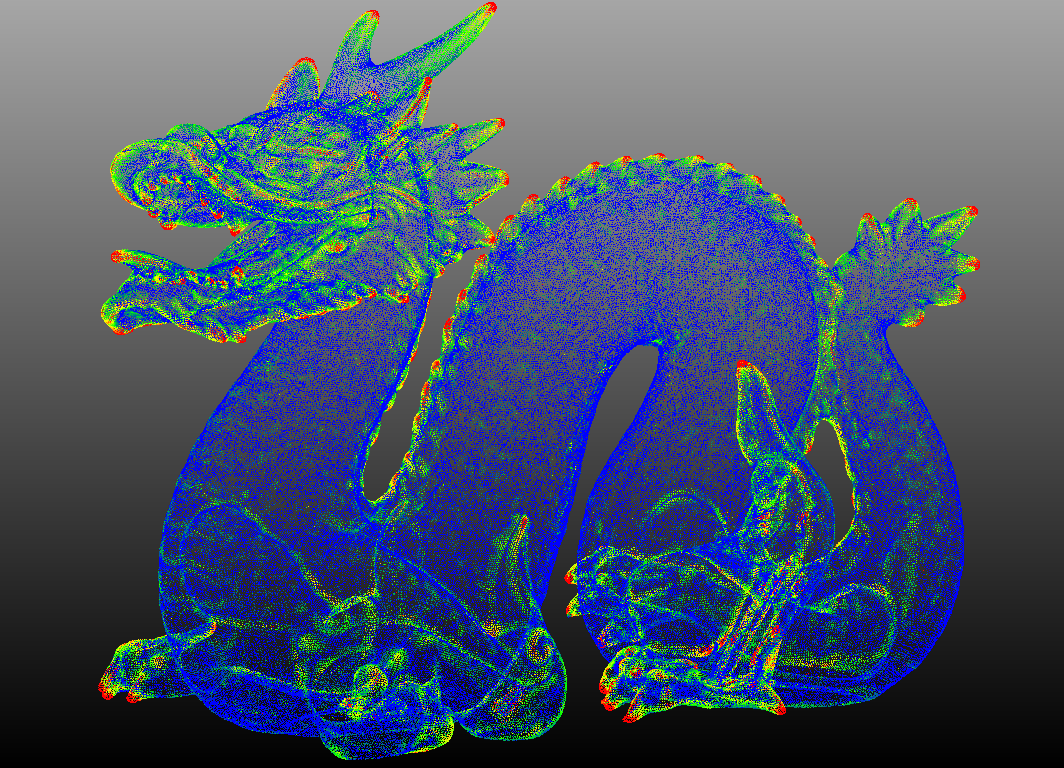}}} \\

\subfigure[Zoom in on dragon's tail]{\includegraphics[width=0.36\textwidth]{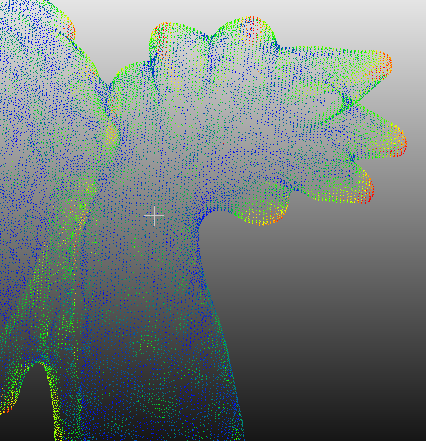}} & \subfigure[Zoom in  on dragon's head]{\includegraphics[width=0.60\textwidth]{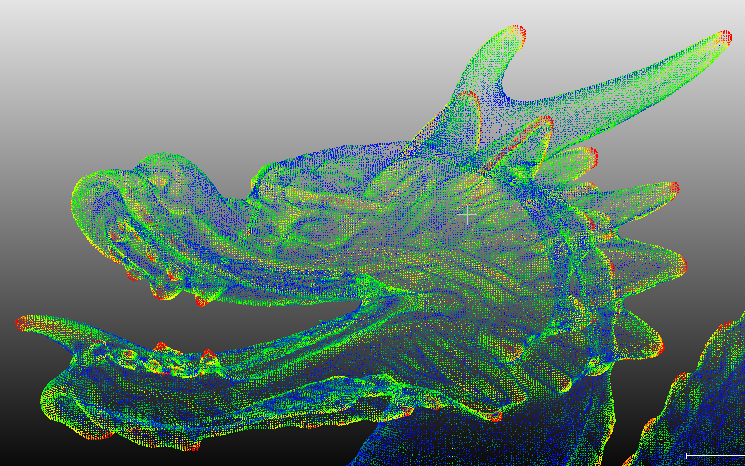}} 
\end{tabular}
\caption{Intensities of the approximate mean curvature of a dragon point cloud ($435\,545$ points, diameter$=1$) with $\epsilon = 0.007$.} \label{fig_dragon}
\end{figure}
\begin{figure}[!htp]
\centering
\subfigure{\includegraphics[width=0.45\textwidth]{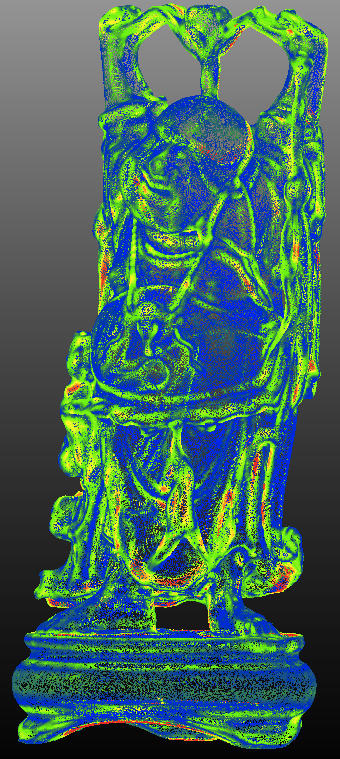}}
\subfigure{\includegraphics[width=0.45\textwidth]{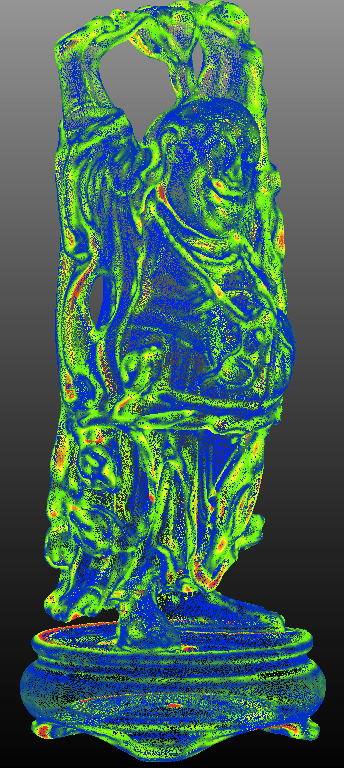}}
\caption{Intensities of the approximate mean curvature of a point cloud with 543\,524 points (horizontal side length=$0.41$, height=$1$.)}\label{buddha-fig}
\end{figure}

\bibliographystyle{alpha} 
\bibliography{biblio_regularisation}

\end{document}